\documentclass[11pt, letterpaper]{amsart}

\usepackage[letterpaper, margin=4cm, headsep=25pt]{geometry}

\usepackage{amsmath,amssymb,amscd,amsthm,indentfirst}
\usepackage{amsfonts}

\usepackage{booktabs}
\usepackage{verbatim}
\usepackage{enumerate}

\usepackage[small,nohug,heads=LaTeX]{diagrams}
\diagramstyle[labelstyle=\scriptstyle]
\usepackage{graphicx}
\usepackage{pstricks}

\numberwithin{equation}{section}

\newcommand{\bF}{\mathbb{F}}

\newcommand{\bN}{\mathbb{N}}

\newcommand{\bP}{\mathbb{P}}
\newcommand{\bQ}{\mathbb{Q}}
\newcommand{\bR}{\mathbb{R}}

\newcommand{\bZ}{\mathbb{Z}}

\newcommand{\MT}[2]{\bold{MT #1}(#2)}
\newcommand{\MTtheta}{\bold{MT \theta}}

\newcommand\lra{\longrightarrow}
\newcommand\lla{\longleftarrow}

\newcommand\Diff{\mathrm{Diff}}
\newcommand\Emb{\mathrm{Emb}}
\newcommand\End{\mathrm{End}}
\newcommand\Bun{\mathrm{Bun}}
\newcommand\Th{\mathrm{Th}}
\newcommand\colim{\operatorname*{colim}}
\newcommand\hocolim{\operatorname*{hocolim}}

\newcommand\Ker{\operatorname*{Ker}}

\newcommand{\X}{\mathbf{X}}
\newcommand{\K}{\mathcal{K}}
\newcommand{\Conn}{\mathbf{Conn}}
\newcommand{\Fib}{\mathrm{Fib}}

\newcommand{\bp}{^\bullet}
\newcommand{\ob}{\mathrm{ob}}
\newcommand{\map}{\mathrm{map}}

\newarrow{DashTo}{}{dash}{}{dash}>

\newcommand{\R}{\bR}

\newcommand{\RP}{\R P}

\newcommand{\Int}{\mathrm{int}}

\newcommand{\st}{\mathrm{st}}

\newcommand{\dist}{\mathrm{dist}}

\newcommand{\IM}{\mathrm{Im}}
\newcommand{\moddd}{/\!\!/}

\renewcommand{\epsilon}{\varepsilon}

\newcommand{\Map}{\mathrm{Map}}

\newcommand{\supp}{\mathrm{supp}}

\newcommand{\manif}{M}

\usepackage[all]{xy}
\CompileMatrices

\usepackage{amsthm}
\theoremstyle{plain}
\newtheorem{MainThm}{Theorem}
\newtheorem{MainCor}[MainThm]{Corollary}

\newtheorem{theorem}{Theorem}[section]
\newtheorem{proposition}[theorem]{Proposition}
\newtheorem{lemma}[theorem]{Lemma}

\newtheorem{corollary}[theorem]{Corollary}

\theoremstyle{definition}
\newtheorem{definition}[theorem]{Definition}

\newtheorem{example}[theorem]{Example}

\theoremstyle{remark}

\newtheorem*{remark*}{Remark}

\usepackage[hypertexnames=false]{hyperref}

\title[Monoids of moduli spaces of manifolds]{Monoids of moduli spaces of manifolds}

\author{S{\o}ren Galatius}
\thanks{S. Galatius was partially supported by NSF grant DMS-0805843 and the Clay Mathematics Institute.}
\email{galatius@stanford.edu}
\address{Department of Mathematics\\
	Stanford University\\
	Stanford CA, 94305}

\author{Oscar Randal-Williams}
\thanks{O. Randal-Williams was supported by an EPSRC Studentship, DTA grant number EP/P502667/1}
\email{randal-w@maths.ox.ac.uk}
\address{Mathematical Institute\\
	24-29 St Giles'\\
	Oxford\\
	OX1 3LB\\
	United Kingdom}

\date{\today}

\begin{document}
\begin{abstract}
We study categories of $d$-dimensional cobordisms from the perspective of \cite{Tillmann} and \cite{GMTW}.  There is a category $\mathcal{C}_\theta$ of closed smooth $(d-1)$-manifolds and smooth $d$-dimensional cobordisms, equipped with generalised orientations specified by a fibration $\theta : \X \to BO(d)$. The main result of \cite{GMTW} is a determination of the homotopy type of the classifying space $B\mathcal{C}_\theta$.  The goal of the present paper is a systematic investigation of subcategories $\mathcal{D} \subseteq \mathcal{C}_\theta$ with the property that $B\mathcal{D} \simeq B\mathcal{C}_\theta$,  the smaller such $\mathcal{D}$ the better. 

We prove that in most cases of interest, $\mathcal{D}$ can be chosen to be a homotopy commutative monoid. As a consequence we prove that the stable cohomology of many moduli spaces of surfaces with $\theta$-structure is the cohomology of the infinite loop space of a certain Thom spectrum $\MTtheta$. This was known for certain special $\theta$, using homological stability results; our work is independent of such results and covers many more cases.
\end{abstract}
\maketitle

\section{Introduction and statement of results}

To state our results, we first recall from \cite{GMTW} the
definition of the $d$-dimensional cobordism category. Let us give the definition in outline here, and in full detail in \S\ref{sec:CobordismCategoryDefinition}.
A tangential structure is a fibration $\theta: \X \to
BO(d)$, and we write $\gamma \to BO(d)$ for the canonical vector bundle. A $\theta$-structure on a vector bundle is a bundle map (i.e.\ fibrewise linear isomorphism) from the vector bundle to $\theta^* \gamma$, and a $\theta$-structure on a manifold is a $\theta$-structure on its tangent bundle.

The objects of the cobordism category $\mathcal{C}_\theta$ are pairs $(M,\ell)$, where $M
\subseteq \R^\infty$ is a closed $(d-1)$-manifold and $\ell$ is a $\theta$-structure on $\epsilon^1 \oplus TM$. The non-identity morphisms of $\mathcal{C}_\theta$, $(M_0, \ell_0) \to (M_1, \ell_1)$, are pairs $(t,W)$ with $t > 0$ and $W \subseteq [0,t] \times
\R^\infty$ an embedded cobordism, together with a $\theta$-structure $\ell$ on $W$ that agrees with $\ell_0$ and $\ell_1$ over the boundary. There is a ``collar'' condition on $(W,\ell)$ near $\partial W$, and the
space of all such $W$ has a nice topology, described in detail in \S\ref{sec:CobordismCategoryDefinition} (and also \cite{GMTW}).  The homotopy type of morphism spaces is given by
\begin{align}\label{eq:MorphismsInCobCat}
  \mathcal{C}_\theta(M_0, M_1) \simeq \coprod_{W}
  B\Diff_\theta(W,\partial W),
\end{align}
where $W$ ranges over connected compact cobordisms from $M_0$ to $M_1$, one in each diffeomorphism class rel $\partial W$.  $\Diff(W,\partial W)$ is the topological
group of diffeomorphisms of $W$ which restrict to the identity on a
neighbourhood of $\partial W$, and $B\Diff_\theta(W,\partial W)$ denotes
the homotopy quotient
\begin{align*}
  B\Diff_\theta(W,\partial W) = \Bun^\partial(TW, \theta^*\gamma_d)
  \moddd \Diff(W,\partial W),
\end{align*}
where $\Bun^\partial(TW,\theta^*\gamma_d)$ is the space of bundle maps
fixed near the boundary. We will say that $B \Diff_\theta(W, \partial W)$ is the moduli space of $d$-manifolds with $\theta$-structure, with underlying manifold diffeomorphic to $W$.

Finally, we recall the main theorem of \cite{GMTW} (reproved below as Theorem \ref{thm:MainThmGMTW}, largely following Chapter 6 of \cite{galatius-2006}).  It determines the
homotopy type of the classifying space of the cobordism category,
\begin{align}
  B\mathcal{C}_\theta \simeq \Omega^{\infty-1} \MTtheta,
\end{align}
where $\MTtheta = \X^{-\theta}$ is the Thom spectrum of the inverse of
the vector bundle classified by $\theta: \X \to BO(d)$.

To state our results, let us consider a version of
$\mathcal{C}_\theta$ where manifolds have ``basepoints''. Pick once and for all a $\theta$-structure on the vector space $\bR^d$. This induces a $\theta$-structure on any framed manifold, which we shall call the \textit{standard $\theta$-structure} on that manifold.
\begin{definition}\label{defn:ThetaDotCategory}
Let $d \geq 2$.  Let $\mathcal{C}^\bullet_\theta$ be the subcategory of
  $\mathcal{C}_\theta$ where objects $M \subseteq \R^\infty$ are
  required to contain the interval $(-\epsilon,\epsilon)^{d-1} \times
  \{0\} \subseteq \R^\infty$ for some $\epsilon > 0$ and to be
  \emph{connected}.
  Morphisms $W \subseteq [0,t] \times \R^\infty$ are required to be
  connected and to contain the strip $[0,t] \times
  (-\epsilon,\epsilon)^{d-1} \times \{0\}$ for some $\epsilon > 0$.
  Furthermore, the tangential structure $l: TW \to \theta^*\gamma$ is
  required to be standard on the strip $[0,t] \times (-\epsilon,\epsilon)^{d-1} \times
  \{0\}$.
\end{definition}
Homotopically, the result of this condition is to replace \eqref{eq:MorphismsInCobCat}
by
\begin{align}\label{eq:CdotDefn}
  \mathcal{C}_\theta^\bullet(M_0, M_1) \simeq \coprod_{W}
  B\Diff_\theta(W,L \cup \partial W),
\end{align}
where $W$ ranges over connected compact cobordisms between connected manifolds $M_0$ and $M_1$, containing an embedded arc $L$ connecting the two boundary components, one such $W$ in each diffeomorphism class rel to $\partial W \cup L$.

Tangential structures and diffeomorphisms are
required to be fixed near the line $L \subseteq W$.  Thus we can cut
open the surface $W$ along $L$, so the spaces in~(\ref{eq:CdotDefn}) are
effectively moduli spaces of surfaces with \emph{one} boundary
component.

\begin{MainThm}\label{thm:SubcategoryConnected}
Let $d \geq 2$ and $\theta : \X \to BO(d)$ be a tangential structure such that $\X$ is path connected and $S^d$ admits a $\theta$-structure. Then the inclusion
$$B \mathcal{C}_\theta\bp \to B \mathcal{C}_\theta \simeq \Omega^{\infty-1}\MTtheta$$
is a weak homotopy equivalence.
\end{MainThm}

Our main result is the following.  Let us say that two objects in
$\mathcal{C}^\bullet_\theta$ are \textit{cobordant} if they give the same element in $\pi_0 B \mathcal{C}_\theta\bp = \pi_{-1}\MTtheta$. In section \ref{sec:ReversingMorphisms} we show that this happens precisely when there is a morphism between the objects (not just a zig-zag).

\begin{MainThm}\label{thm:SubcategoryMonoids}
  Let $\theta: \X \to BO(2)$ be a tangential structure such that $\X$ is
  path connected and that $S^2$ admits a $\theta$-structure.  Let
  $\mathcal{D} \subseteq \mathcal{C}_\theta^\bullet$ be a full
  subcategory. Then the inclusion
  \begin{align*}
    B\mathcal{D} \to B\mathcal{C}_\theta\bp
  \end{align*}
  is a weak homotopy equivalence of each component of $B\mathcal{D}$ onto a component of $B\mathcal{C}_\theta\bp$.

If $\mathcal{D}$ has exactly one object, then it is homotopy commutative monoid; if it has at most one object in each cobordism class, then it is a disjoint union of homotopy commutative monoids.
\end{MainThm}
The assumption that $\X$ be path connected is innocent, since
a disconnected $\X$ could be considered one path component at a time.
The assumption that $S^2$ admits a $\theta$-structure is necessary for
our proof, which uses surgery techniques.  Under that assumption the
connected sum of two surfaces with $\theta$-structure will again have
a $\theta$-structure (in contrast, the connected sum of e.g.\
\emph{framed} surfaces is not framable; this corresponds to $\X =
EO(2)$).

Let us spell out our result explicitly in the case of ordinary orientation, (although it is not new in this case). Then $W$ in (\ref{eq:CdotDefn}) ranges over connected oriented
surfaces, and $B\Diff_\theta(W,\partial W \cup L)$ is homotopy
equivalent to $B\Gamma_{g,1}$, where $\Gamma_{g,1}$ is the mapping
class group of an oriented genus $g$ surface with one boundary
component.  The monoid $\mathcal{D}$ is homotopy equivalent to the
disjoint union
\begin{align*}
  \mathcal{D} \simeq \coprod_{g \geq 0} B\Gamma_{g,1},
\end{align*}
and the composition is the ``pair of pants'' composition of
\cite{Miller}.  Our result then says that this monoid has the same
classifying spaces as the full cobordism category of \cite{GMTW}. Our proof is entirely geometric and does not rely on Harer stability.

\subsection{Stabilisation and Madsen--Weiss' theorem}\label{sec:harer-stab-mads}

Let us explain an application of Theorem \ref{thm:SubcategoryMonoids} that highlights the advantage of homotopy commutativity. First we briefly discuss group completion of a homotopy commutative monoid $\mathcal{M}$, following McDuff and Segal \cite{McDuff-Segal}. There is a canonical map $\mathcal{M} \to \Omega B \mathcal{M}$ which is a homotopy equivalence if and only $\mathcal{M}$ is grouplike, i.e.\ $\pi_0\mathcal{M}$ is a group. There is an induced map in homology, 
$$H_*(\mathcal{M}) \to H_*(\Omega B \mathcal{M})$$
which sends the multiplicative subset $\pi_0\mathcal{M} \subset H_*(\mathcal{M})$ to units of $H_*(\Omega B \mathcal{M})$, so there is an induced map from the localisation
\begin{align}\label{eq:GpCompletion}
H_*(\mathcal{M})[\pi_0 \mathcal{M}^{-1}] \to H_*(\Omega B \mathcal{M}).
\end{align}
The main theorem about group completion for homotopy commutative monoids \cite{McDuff-Segal} is that (\ref{eq:GpCompletion}) is an isomorphism of rings.

In many cases of interest, $\pi_0 \mathcal{M}$ is a finitely generated monoid, so the localisation (\ref{eq:GpCompletion}) can be formed as a sequential direct limit 
$$H_*(\mathcal{M}) \to H_*(\mathcal{M}) \to H_*(\mathcal{M}) \to \cdots$$
over multiplication by an element $m$ which is the product of a set of generators. In fact, $\pi_0\mathcal{M}$ need not be finitely generated, it is only necessary that $\pi_0\mathcal{M}$ may be group completed by inverting finitely many elements. This direct limit is the homology of the space $\mathcal{M}_\infty$ obtained as the mapping telescope of the analogous direct system of spaces $\mathcal{M} \to \mathcal{M} \to \mathcal{M} \to \cdots$. Consequently we get a map $\mathcal{M}_\infty \to \Omega B \mathcal{M}$, inducing an isomorphism
$$H_*(\mathcal{M}_\infty) \to H_*(\Omega B \mathcal{M}).$$

If $k$ is a field, the isomorphism (\ref{eq:GpCompletion}) can be reinterpreted in terms of invariants of the action of $\pi_0\mathcal{M}$ on $H^*(\mathcal{M} ; k)$, namely that
\begin{align}\label{eq:GpCompletionCohomology}
H^*(\mathcal{M} ; k)^{\pi_0 \mathcal{M}} \cong H^*(\Omega_0 B \mathcal{M} ; k),
\end{align}
where $\Omega_0$ means the basepoint component of the loop space. Here the monoid $\pi_0\mathcal{M}$ acts on homology and cohomology of both spaces by translation. To deduce (\ref{eq:GpCompletionCohomology}) we take coinvariants of the isomorphism (\ref{eq:GpCompletion}) in $k$-homology and algebraically dualise, and note that the invariants of the action on $H^*(\Omega B \mathcal{M} ; k)$ is isomorphic to the cohomology of $\Omega_0 B \mathcal{M}$.

\vspace{2ex}

We now apply this to the homotopy commutative monoid $\mathcal{D}$ produced by Theorem \ref{thm:SubcategoryMonoids}. The assumption that $\pi_0\mathcal{D}$ may be group completed by inverting finitely many elements holds in many cases of interest. The homology of a component of $\mathcal{D}_\infty$ can be interpreted as the homology of $B\Diff_\theta(W_\infty)$, the moduli space of infinite genus $\theta$-surfaces with parametrised germ at infinity.

The cohomology $H^*(\mathcal{D})$ is the ring of characteristic classes of bundles of $\theta$-surfaces with one parametrised boundary component. Then $H^*(\Omega_0^\infty \MTtheta)$ is the ring of characteristic classes that are invariant under gluing a trivial bundle onto the boundary.

Ebert \cite{Ebert09} has recently proved that a similar result cannot hold for odd dimensional manifolds. Indeed he shows there are non-trivial rational classes in $H^*(\Omega^\infty \MT{SO}{2n+1} ; \bQ {} )$ which vanish in $H^*(B\Diff^+(M) ; \bQ {} )$ for all oriented $(2n+1)$-manifolds $M$.

\vspace{2ex}
 
We now give some particular cases of interest. Returning to the case of ordinary orientations, we reproduce the Madsen--Weiss theorem \cite{MW}, that there is a homology equivalence
\begin{equation}\label{thm:MW}
  \bZ \times B\Gamma_{\infty,1} \to \Omega^\infty \MT{SO}{2},
\end{equation}
where $\Gamma_{\infty,1}$ is the limit of the mapping class groups
$\Gamma_{g,1}$ as $g \to \infty$.
Again we point out that we prove this result \emph{without} quoting the homological stability results of Harer \cite{H} and Ivanov \cite{Ivanov}. The case of unoriented and spin surfaces can be treated similarly, cf.\ \S\ref{sec:Applications}.

For an oriented surface $F$, Cohen and Madsen \cite{CM} introduced spaces
$$\mathcal{S}_{g,n}(Y) = \Map^\partial(F_{g,n}, Y) \moddd \Diff^+(F_{g,n})$$
where $\Map^\partial$ is the space of maps taking $\partial F$ to the basepoint of $Y$, and $\Diff^+$ denotes diffeomorphisms that preserve orientation and boundary. These are morphism spaces in the category $\mathcal{C}_\theta$, for $\theta : BSO(2) \times Y \to BO(2)$, and it follows from \cite{CM} that, if $Y$ is simply connected, there is a homology equivalence
$$\bZ \times \mathcal{S}_{\infty,1}(Y) \to \Omega^\infty \MT{SO}{2} \wedge Y_+.$$
Here, $\pi_0(\mathcal{D}) = \bN \times H_2(Y;\bZ)$ and $\bZ \times \mathcal{S}_{\infty,1}(Y)$ is the direct limit $\mathcal{D} \to \mathcal{D} \to \dots$ over multiplication by an element corresponding to $(1,0) \in \bN \times H_2(Y,\bZ)$. Our result provides an analogue to this for all connected spaces $Y$ and also to surfaces with any tangential structure satisfying the assumption of Theorem \ref{thm:SubcategoryMonoids}, such as unoriented or spin surfaces.

The special case $Y = B\bZ = S^1$ can be interpreted in complete analogy with (\ref{thm:MW}).  Let $v \in
H^1(F_{g,1};\bZ)$ be a ``primitive'' cohomology class, i.e.\ one that can
be extended to a symplectic basis of $H^1$ (equivalently, the Poincar\'{e}
dual to an embedded non-separating curve).  Let $\Gamma'_{g,1} \leq
\Gamma_{g,1}$ denote the stabilizer of $v$.  Then our results give a
homology isomorphism
\begin{equation*}
\bZ \times B\Gamma'_{\infty,1} \to
\Omega^\infty \MT{SO}{2} \wedge S^1_+.
\end{equation*}

Similarly the case $X = B\bZ/n$ gives a homology equivalence
\begin{equation*}
\bZ \times B\Gamma'_{\infty,1}(n) \to \Omega^\infty \MT{SO}{2} \wedge B\bZ/n_+,
\end{equation*}
where $\Gamma'_{g,1}(n) \leq \Gamma_{g,1}$ denotes
the subgroup that stabilizes the mod $n$ reduction of the primitive
vector $v$.  The space $B\Gamma'_{g,1}(n)$ has the same homotopy type
as the moduli space of pairs $(\overline \Sigma, \Sigma)$ where
$\Sigma$ is a genus $g$ Riemann surface with one boundary component
and $\overline \Sigma \to \Sigma$ is an unbranched $n$-fold cyclic Galois cover.

For a completely general space $X$ the monoid $\pi_0 \mathcal{D}$
is difficult to understand explicitly, and our result must be stated
as an algebraic isomorphism $H_*(\mathcal{D})[\pi_0 \mathcal{D}^{-1}] \simeq H_*(\Omega^\infty
\MTtheta)$.



\section{Spaces of Manifolds}\label{SoM}

\subsection{Definitions}\label{sec:definitions}
\begin{definition}
  Let $U\subseteq \R^N$ be an open set.  Let $\Psi_d(U)$ be
  the set of subsets $\manif \subseteq U$ which are smooth $d$-dimensional
  submanifolds without boundary, and such that $\manif$ is closed as a
  subset. When the dimension is not important, we may simply write $\Psi(U)$.
\end{definition}
If $V \subseteq U$, there is a restriction map $\Psi(U) \to \Psi(V)$
given by $\manif \mapsto \manif \cap V$.  This makes $\Psi(U)$ into a sheaf of
sets.  We now define three topologies on $\Psi(U)$.  The first two are
used only as intermediate steps for defining the third.  In Theorem \ref{thm:sheaf} below we prove that the third topology makes $\Psi(-)$ into a \emph{sheaf of topological spaces}.

\textbf{Step 1.} We first define the \emph{compactly supported
  topology}.  We will write $\Psi(U)^{\mathit{cs}}$ for $\Psi(U)$
equipped with this topology.  In fact, $\Psi(U)^{\mathit{cs}}$ will
be an infinite-dimensional smooth manifold, in which a neighbourhood of
$\manif \in \Psi(U)^{\mathit{cs}}$ is homeomorphic to a neighbourhood of
the zero-section in the vector space $\Gamma_c(N\manif)$ consisting of
compactly supported sections of the normal bundle $N\manif$ of $\manif \subseteq
U$.

Let $C^\infty_c(\manif)$ denote the set of compactly supported smooth
functions on $\manif$.  Given a function $\epsilon: \manif \to (0,\infty)$ and
finitely many vector fields $X = (X_1, \dots, X_r)$ on $\manif$, let
$B(\epsilon, X)$ denote the set of functions such that $|(X_1 X_2
\dots X_r f)(x)| < \epsilon(x)$ for all $x$.  Declare the family of
sets of the form $f + B(\epsilon,X)$ a subbasis for the topology on
$C^\infty_c(\manif)$, as $f$ ranges over $C^\infty_c(\manif)$, $\epsilon$ over
functions $\manif \to (0,\infty)$, and $X$ over $r$-tuples of vector
fields, and $r$ over non-negative integers.  This makes
$C^\infty_c(\manif)$ into a locally convex vector space.

We define the normal bundle $N\manif$ to be the subbundle of $\epsilon^n$
which is the orthogonal complement to the tangent bundle $T \manif
\subseteq \epsilon^n$.  This identifies $\Gamma_c(N\manif)$ with a linear
subspace of $C^\infty_c(\manif)^{\oplus n}$.  We topologise it as a
subspace.

By the tubular neighbourhood theorem, the standard map $N\manif \to \R^n$
restricts to an embedding of a neighbourhood of the zero section.
Associating to a section $s$ its image $s(\manif)$ gives a partially
defined injective map
\begin{align}
  \Gamma_c(N\manif) \stackrel{c_\manif}{\dasharrow} \Psi(U)^{\mathit{cs}}
\end{align}
whose domain is an open set. Topologise $\Psi(U)^{\mathit{cs}}$ by
declaring the maps $c_M$ to be homeomorphisms onto open sets. This makes $\Psi(U)^{\mathit{cs}}$ into
an \emph{infinite dimensional manifold}, modelled on the topological
vector spaces $\Gamma_c(N\manif)$.

\textbf{Step 2.}  For each compact set $K$, we define a topology on
$\Psi(U)$, called the \textit{$K$-topology}.  We will write $\Psi(U)^K$ for
$\Psi(U)$ equipped with this topology.

Let
\begin{align*}
\Psi(U)^{\mathit{cs}} \stackrel{\pi_K}\to \Psi(K\subseteq U)
\end{align*}
be the quotient map that identifies elements of
$\Psi(U)^{\mathit{cs}}$ if they agree on a neighbourhood of $K$.  The
image of a manifold $\manif \in \Psi(U)^{\mathit{cs}}$ is the germ of $\manif$
near $K$, and we shall also write $\pi_K(\manif) = \manif|_{K}$.  Give
$\Psi(K\subseteq U)$ the quotient topology.

Now, let $\Psi(U)^K$ be the topological space with the same underlying
set as $\Psi(U)^{\mathit{cs}}$, and with the coarsest topology making
$\pi_K: \Psi(U)^K \to \Psi(K\subseteq U)$ continuous.  It is a formal
consequence of the universal properties of initial and quotient
topologies that the identity map $\Psi(U)^{L} \to \Psi(U)^K$ is
continuous when $K\subseteq L$ are two compact sets.  That is, the
$L$-topology is finer than the $K$-topology.

\textbf{Step 3.} Finally, let $\Psi(U)$ have the coarsest topology
finer than all the $K$-topologies.  In other words, $\Psi(U)$ is the
inverse limit of $\Psi(U)^K$ over larger and larger compact sets.

\begin{example}
The simplest example is taking $\{ t \} \times \bR \subseteq \bR^2$ as a function of $t$. This a a path in $\Psi^1(\bR^2)$, and as $t \to \infty$ it converges to the empty manifold $\emptyset$. This is because for each compact subset $K \subseteq \bR^2$, $K \cap
\{t\}\times \bR$ is empty for all sufficiently large $t$, so it converges to $\emptyset$ in the $K$-topology, for all $K$.
\end{example}

\subsection{Elementary properties and constructions}\label{sec:properties}

Let $V \subseteq U$.  The restriction map $\Psi(U)^{\mathit{cs}} \to
\Psi(V)^{\mathit{cs}}$ is not continuous.  We have the following
result instead.
\begin{lemma}\label{lem:r-open}
  The restriction map $r: \Psi(U)^{\mathit{cs}} \to
  \Psi(V)^{\mathit{cs}}$ is an open map.
\end{lemma}
\begin{proof}
  Let $\manif \in \Psi(U)^{\mathit{cs}}$.  We have the following
  commutative diagram
  \begin{align*}
    \xymatrix{
      {\Gamma_c(N\manif)} \ar@{-->}[r]^{c_\manif} & {\Psi(U)^{\mathit{cs}}} \ar[d] \\
      {\Gamma_c(N(\manif)|_V)} \ar@{-->}[r]^-{c_{\manif\cap V}}\ar[u]^{z} &
      \Psi(V)^{\mathit{cs}} }
  \end{align*}
  where $z$ denotes extension by 0, which is continuous.  This induces a
  partially defined right inverse $\Psi(V)^{\mathit{cs}} \to
  \Psi(U)^{\mathit{cs}}$ to the restriction map.  The right inverse is
  defined in an open neighbourhood of every point in the image of the
  restriction; this proves the claim.
\end{proof}

The following technical result will be used several times.
\begin{lemma}\label{lem:lambda}
  Let $K \subseteq U$ be compact.  Let $0 < 3\epsilon \leq \dist(K,
  \R^n - U)$ and let $\lambda: \R^n \to [0,1]$ be a smooth function
  such that $\lambda(x) = 1$ if $\dist(x,K) \leq \epsilon$ and
  $\lambda(x) = 0$ if $\dist(x,K) \geq 2\epsilon$.  If the support of
  $\lambda$ is contained in an open set $V\subseteq U$, then multiplication
  by $\lambda$ gives a (continuous!) map $\overline \lambda:
  \Gamma_c(N\manif) \to \Gamma_c(N(\manif\cap V))$.  If we let $z: \Gamma_c(N(\manif
  \cap V)) \to \Gamma_c(N\manif)$ denote extension by zero, we have the
  following diagram
  \begin{align*}
    \xymatrix{ \Gamma_c(N(\manif\cap V)) \ar[r]^-z &
      \Gamma_c(N\manif)\ar@{-->}[d]^{c_\manif}
      \ar@{-->}[rd]^{\pi_K \circ c_\manif} &\\
      \Gamma_c(N\manif)\ar[u]^{\overline \lambda} \ar@{-->}[r]_{c_\manif} & \Psi(U)^{\mathit{cs}}
      \ar[r]_-{\pi_K} & \Psi(K\subseteq U) }
  \end{align*}
  in which the triangle and the outer pentagon commute after possibly
  restricting to a smaller neighbourhood of $0 \in \Gamma_c(N\manif)$
\end{lemma}
\begin{proof}
  The pentagon commutes for sections $h \in \Gamma_c(N\manif)$ satisfying
  $|h(x)| < \epsilon$, and this is an open condition.
\end{proof}

\begin{lemma}\label{lem:open-map}
  The quotient map $\pi_K: \Psi(U)^{\mathit{cs}} \to \Psi(K\subseteq
  U)$ is an open map.
\end{lemma}
\begin{proof}
  We are claiming that $(\pi_K)^{-1}(\pi_K(A))$ is open for all
  open $A\subseteq \Psi(U)^{\mathit{cs}}$.  Let $\manif \in
  (\pi_K)^{-1}(\pi_K(A))$.  This means that $\manif \cap V = \manif'
  \cap V$ for some open $V \supseteq K$ and some $\manif' \in A$.  Now
  the composition $\pi_K \circ c_\manif: \Gamma_c(N\manif)
  \dashrightarrow \Psi(U)^{\mathit{cs}} \to \Psi(K \subseteq U)$ can
  be factored as in Lemma~\ref{lem:lambda}.  Thus, if we want to check
  that some point $\manif \in (\pi_K)^{-1}(\pi_K(A))$ is
  interior, it suffices to check that the inverse image of
  $\pi_K(\manif)$ in $\Gamma_c(N(\manif\cap V))$ is a neighbourhood of
  0.

  This holds for $\manif' \in A \subseteq
  (\pi_K)^{-1}(\pi_K(\manif))$, and hence also holds for $\manif$
  since they agree inside $V$.
\end{proof}

\begin{lemma}\label{lem:restriction}
  Let $V\subseteq U$.  The injection $\rho: \Psi(K\subseteq U) \to
  \Psi(K\subseteq V)$ is a homeomorphism onto an open subset.
\end{lemma}
\begin{proof}
  Continuity follows from Lemma~\ref{lem:lambda}.  Indeed, we get a diagram
  \begin{align*}
    \xymatrix{ \Gamma_c(N\manif) \ar[d]_{\overline{\lambda}}\ar@{-->}[r]^-{c_\manif} &
      {\Psi(U)^{\mathit{cs}}} \ar[d]_{r} \ar[r]^-{\pi_K} &
      \Psi(K\subseteq U) \ar[d]^\rho \\
      \Gamma_c(N(\manif\cap V)) \ar@{-->}[r]^-{c_{\manif \cap V}} & \Psi(V)^{\mathit{cs}}\ar[r]^-{\pi_K}
      & \Psi(K\subseteq V) }
  \end{align*}
where the outer rectangle and right hand square commute.
  This proves that $\rho\circ \pi_K\circ c_\manif$ is continuous. Since
  $\pi_K$ is a quotient map and $c_\manif$ is a local homeomorphism, this
  proves that $\rho$ is continuous.

  To see that $\rho(A)$ is open when $A\subseteq \Psi(K \subseteq U)$
  is open, let $B = (\pi_K)^{-1}(A)$.  Then $\rho(A) = \pi_K\circ
  r(B)$.  Since $\pi_K$ and $r$ are both open maps, by
  Lemmas~\ref{lem:open-map} and~\ref{lem:r-open}, $\rho(A)$ is open.
\end{proof}

\begin{theorem}\label{thm:restr-cont}
  For $V \subseteq U$, the restriction map $\Psi(U) \to \Psi(V)$ is
  continuous.
\end{theorem}
\begin{proof}
  We need to show that the composition $\Psi(U) \to \Psi(V) \to \Psi(K
  \subseteq V)$ is continuous for all compact $K \subseteq V$.  This
  follows from Lemma~\ref{lem:restriction}.
\end{proof}

\begin{lemma}\label{lem:cover}
  Let $K_i \subseteq U$ be compact, $i = 1, \dots, r$, and let $K =
  \cup_i K_i$.  Then the diagonal map $\delta: \Psi(U)^K \to \prod
  \Psi(U)^{K_i}$ is a homeomorphism onto its image.
\end{lemma}
\begin{proof}
  Continuity of $\delta$ follows from continuity of each $\Psi(U)^K
  \to \Psi(U)^{K_i}$.  Let $\Delta = \delta(\Psi(U)^K)$.  We need to
  see that $\delta: \Psi(U)^K \to \Delta$ is open.  We have a diagram
  \begin{align*}
    \xymatrix{
      \Psi(U)^K \ar[r]^-\delta & {\Delta} \ar[d]_\pi\ar[r] & {\prod
        \Psi(U)^{K_i}} \ar[d]^{\prod \pi_{K_i}}\\
      &{D} \ar[r] & {\prod \Psi(K_i \subseteq U)},
    }
  \end{align*}
  where $D = (\prod \pi_{K_i})(\Delta)$, the horizontal maps are
  inclusions of subsets, and $\pi$ is the restriction of $\prod
  \pi_{K_i}$.  A set-theoretic consideration shows that for any subset
  $A \subseteq \Psi(U)$ we have
  \begin{align}
    \label{eq:100}
    \pi^{-1} \pi \delta(A) = \big(\prod \pi_{K_i}^{-1}
    \pi_{K_i}(A)\big) \cap \Delta.
  \end{align}
  Now consider the diagram
  \begin{align}\label{eq:101}
    \begin{aligned}
      \xymatrix{
        {\Delta} \ar[r]^\pi & D\ar[d]^p\\
        {\Psi(U)^K}\ar[u]^\delta \ar[r]^-{\pi_K} & \Psi(K\subseteq U), }
    \end{aligned}
  \end{align}
  where $p$ is defined by commutativity of the diagram:
  $p(\pi\delta(\manif)) = \pi_K(\manif)$.  This is well defined because both
  $\pi$ and $\delta$ are surjective, and because if $\pi \delta(\manif) =
  \pi\delta(\manif')$, then $\manif$ and $\manif'$ agree in a neighbourhood of each
  $K_i$ and hence agree in a neighbourhood of $K$, so $\pi_K(\manif) =
  \pi_K(\manif')$.  By Lemma~\ref{lem:open-map}, the open neighbourhoods of
  $\manif \in \Psi(U)^K$ are precisely of the form $\pi_K^{-1}\pi_K(A)$,
  for $\manif \in A\subseteq \Psi(U)^{\mathit{cs}}$ open.  We need to prove
  that the set
  \begin{align*}
    \delta (\pi_K^{-1}\pi_K(A)) \subseteq \Delta
  \end{align*}
  is a neighbourhood of $\delta(\manif)$.  By diagram~(\ref{eq:101}), we can
  replace $\pi_K(A)$ by $p\pi\delta(A)$, and replace
  $\delta\pi_K^{-1}$ by $\pi^{-1}p^{-1}$.  Using~(\ref{eq:100}) this
  gives
  \begin{align*}
    \delta (\pi_K^{-1}\pi_K(A)) = \pi^{-1}p^{-1}p\pi\delta(A)
    \supseteq \pi^{-1}\pi \delta(A) = \big( \prod \pi_{K_i}^{-1}
    \pi_{K_i}(A) \big) \cap \Delta,
  \end{align*}
  which is an open subset of $\Delta$, containing $\delta(\manif)$.
\end{proof}

\begin{theorem}\label{thm:sheaf}
  Let $f: X \to \Psi(U)$ be a map such that each $x \in U$ has an open
  neighbourhood $U_x\subseteq U$ such that the composition $X \to
  \Psi(U) \to \Psi(U_x)$ is continuous.  Then $X \to \Psi(U)$ is
  continuous.
\end{theorem}
\begin{proof}
  We need to show that $f: X \to \Psi(U)^{K}$ is continuous for any
  compact $K \subseteq U$.  $K$ is covered by finitely many of the
  $U_x$'s, say $U_1, \dots, U_r$.  Pick $K_i \subseteq U_i$ compact
  with $K \subseteq \cup K_i$.  Then the composition
  \begin{align*}
    X \xrightarrow{f} \Psi(U)^{K_i} \to \Psi(K_i \subseteq U) \to
    \Psi(K_i \subseteq U_i)
  \end{align*}
  is continuous for each $i$ by assumption.  By Lemma~\ref{lem:restriction}, the composition $X \to \Psi(U) \to \Psi(K_i
  \subseteq U)$ is continuous, and hence each $X \to \Psi(U)^{K_i}$ is
  continuous.  It now follows from Lemma~\ref{lem:cover} that $X \to
  \Psi(U)^K$ is continuous as required.
\end{proof}

\begin{lemma}
  Let $\Emb(U,V)$ denote the space of embeddings of one open subset of $\bR^N$ into another, and let $j_0 \in
  \Emb(U,V)$.  Then there exists a partially defined map
  \begin{align*}
    \Emb(U,V) \stackrel{\varphi}{\dashrightarrow} \Diff_c(U),
  \end{align*}
  defined in an open neighbourhood of $j_0$, such that $j(x) = j_0
  \circ (\varphi(j))(x)$ for $x$ in a neighbourhood of $K$.
\end{lemma}
\begin{proof}
  Pick a compactly supported function $\lambda: U \to [0,1]$ with $K
  \subseteq \Int(\lambda^{-1}(0))$.  Then let
  \begin{align*}
    (\varphi(j))(x) = (1-\lambda(x))x + \lambda(x)(j_0^{-1} \circ j)(x),
  \end{align*}
  which defines a compactly supported diffeomorphism $U \to U$ for all
  $j$ sufficiently close to $j_0$.
\end{proof}
\begin{theorem}\label{thm:action-cont}
  The map
  \begin{align*}
    \Emb(U,V) \times \Psi(V) \to \Psi(U),
  \end{align*}
  given by $(j,\manif) \mapsto j^{-1}(\manif)$, is continuous.
\end{theorem}
\begin{proof}
  It suffices to see that the composition
  \begin{align*}
    \Emb(U,V) \times \Psi(V) \to \Psi(U) \to \Psi(K \subseteq U)
  \end{align*}
  is continuous in a neighbourhood of each $\{j_0\} \times \Psi(V)$, for each
  $K\subseteq U$ compact.  But this follows from the previous lemma,
  which implies that the map factors through $\Diff_c(U)$, which acts continuously on $\Psi(U)$.
\end{proof}

\subsection{Tangential structures}

We will study analogues of the point-set topological results from the
previous section, where all manifolds $\manif \in \Psi_d(U)$ are equipped
with some tangential structure. Examples are orientations, spin structures, almost complex structures, etc.

\begin{definition}
  Let $\theta: \X \to BO(d)$ be a Serre fibration.  A $\theta$-structure on $\manif \in \Psi_d(U)$ is a bundle map (i.e.\ fibrewise linear
  isomorphism) $\ell: T\manif \to \theta^*\gamma$.  Let $\Psi_\theta(U)$
  denote the set of pairs $(\manif,\ell)$, where $\manif \in \Psi_d(U)$ and $\ell$ is a
  $\theta$-structure on $\manif$.

More generally, if $\theta : \X \to BO(d+k)$ is a Serre fibration, a $\theta_d$-structure on $\manif \in \Psi_d(U)$ is a bundle map $\ell : \epsilon^k \oplus T\manif \to \theta^*\gamma$. Let $\Psi_{\theta_d}(U)$ denote the set of pairs $(\manif, \ell)$, where $\manif \in \Psi_d(U)$ and $\ell$ is a
  $\theta$-structure on $\manif$.
\end{definition}
Often we will omit the tangential structure and just write $M \in \Psi_\theta(\bR^n)$ instead of $(M, \ell) \in \Psi_\theta(\bR^n)$. The second case is a simple generalisation of the first, so we will only discuss the face of a fibration over $BO(d)$.

Clearly $\Psi_\theta$ is again a sheaf of sets.  We wish to endow it
with a topology so that it is a sheaf of topological spaces.  We start by
defining the analogue of the compactly supported topology.  For
$\manif \in \Psi(U)^{\mathit{cs}}$, let $\Bun(T\manif,\theta^*
\gamma)$ be the space of bundle maps, topologised in the
compact-open topology, and let $\Gamma_c(N\manif) \times
\Bun(T\manif,\theta^*\gamma)$ have the product topology.  For $s \in
\Gamma(N\manif)$ close to the zero section, we have $c_\manif(s) =
s(\manif) \in\Psi(U)^{\mathit{cs}}$.  We also get the bundle
isomorphism $Ds: T\manif \to T(c_\manif(s))$ and hence $(Ds)^{-1}
\circ l$ is a $\theta$-structure on $c_\manif(s)$.  This gives a map
\begin{align*}
  c_\manif^\theta: \Gamma_c(N\manif)\times \Bun(T\manif,\theta^*\gamma)
  \dashrightarrow \Psi_\theta(U)^{\mathit{cs}}
\end{align*}
viz.\ $(s,\ell) \mapsto (s(\manif), (Ds)^{-1} \circ \ell)$.  It is clear that
$c_\manif^\theta$ is injective and that the image is $u^{-1}(\IM(c_\manif))$,
where $u: \Psi_\theta(U) \to \Psi(U)$ is the forgetful map $(\manif,\ell)
\mapsto \manif$.  Topologise $\Psi_\theta(U)^{\mathit{cs}}$ by declaring
the maps $c_\manif^\theta$ homeomorphisms onto open sets.  

Then define $\Psi_\theta(K\subseteq U)$, $\Psi_\theta(U)^K$ and
$\Psi_\theta(U)$ from $\Psi_\theta(U)^{\mathit{cs}}$ as in ``Step 2''
and ``Step 3'' in \S\ref{sec:definitions}.  It is not hard to
modify the proofs of Theorems~\ref{thm:restr-cont}, \ref{thm:sheaf}
and \ref{thm:action-cont} to work also in the presence of $\theta$
structures.  We summarise the result in the following theorem.
\begin{theorem}
  $\Psi_\theta$ is a sheaf of topological spaces on $\R^n$.  The map
  \begin{align*}
    \Emb(U,V)\times \Psi_\theta(U) \to \Psi_\theta(V)
  \end{align*}
  is continuous.
\end{theorem}
\begin{proof}[Proof sketch]
  First we have, with the same proof, an analogue of
  Lemma~\ref{lem:lambda} where the two $\Psi$ are replaced by
  $\Psi_\theta$ and all three spaces of sections of normal bundles in
  the diagram are replaced with their product with $\Bun(T\manif,
  \theta^*\gamma)$.  Once that is established, the $\theta$
  analogues of Theorems~\ref{thm:restr-cont} and \ref{thm:sheaf} are
  proved in the exact same way.  This proves that $\Psi_\theta$
  is a sheaf of spaces.

  An embedding $j: U \to V$ gives a map $\Psi_\theta(U) \to
  \Psi_\theta(V)$ by $(\manif,\ell) \mapsto (j^{-1}\manif, \ell \circ
  (Dj|_{j^{-1}(\manif)}))$.  This defines the map in the theorem, and its
  continuity is proved exactly as in Theorem~\ref{thm:action-cont}.
\end{proof}

\subsection{Smooth maps}

In practice it can be tedious to check continuity of a map $X \to \Psi_\theta(\bR^n)$. If $X$ is smooth there is an easier property to verify.

\begin{definition}
Let $f: X \lra \Psi_\theta(U)$ be a continuous map, and write $f(x) = (M_x, \ell_x)$. Define the \textit{graph} of $f$ to be the space of pairs
$$\Gamma(f) = \cup_{x \in X} \{x\} \times  M_x  \subset X \times U$$
and the vertical tangent bundle to be
$$T^v \Gamma(f) = \cup_{x \in X} \{x\} \times TM_x \subset X \times TU,$$
both with the subspace topology. The $\ell_x$ determine a bundle map $\ell(f) : T^v \Gamma(f) \lra \theta^*\gamma$.
\end{definition}

\begin{definition}
If $X$ is a manifold, say a continuous map $f : X \to \Psi_\theta(\bR^n)$ is \textit{smooth} if $\Gamma(f) \subseteq X \times U$ is a smooth submanifold and $\pi: \Gamma(f) \lra X$ is a submersion. 

More generally say it is \textit{smooth near} $(x, u) \in X \times U$ if there are neighbourhoods $A \subset X$ of $x$ and $B \subset U$ of $u$ such that $A \to X \to \Psi_\theta(U) \to \Psi_\theta(B)$ is smooth. For a closed $C \subset X \times U$ say $f$ is \textit{smooth near} $C$ if it is smooth near every point of $B$.
\end{definition}

\begin{lemma}\label{lem:SmoothMapsAreContinuous}
If $X$ is a $k$-dimensional smooth manifold, there is a bijection between the set of smooth maps $f:X \to \Psi_\theta(U)$ and the set of pairs $(\Gamma, \ell)$ where $\Gamma \subseteq X \times U$ is a smooth $(d+k)$-dimensional submanifold that is closed as a subspace and such that $\pi_X : \Gamma \to X$ is a submersion, and $\ell : \Ker(D\pi_X : T \Gamma \to T X) \to \theta^*\gamma$ is a bundle map.
\end{lemma}
\begin{proof}
A smooth map $f$ determines a graph $\Gamma(f)$, and a bundle map $\ell(f): T^v\Gamma(f) \to \theta^*\gamma$, having the required properties.

Given a pair $(\Gamma, \ell)$, there is certainly a map of sets $f : X \to \Psi_\theta(U)$ given by $x \mapsto (\pi_X^{-1}(x),\ell|_{\pi_X^{-1}(x)} : T\pi_X^{-1}(x) \to \theta^*\gamma)$. Then $\Gamma(f) = \Gamma$, $T^v\Gamma(f) = T^v \Gamma$ and $\ell(f) = \ell$. We must just check that this $f$ is continuous.

By the local submersion theorem, we can choose local coordinates $\bR^k \to X$ and $\bR^n \to U$ inside which $\Gamma \subseteq X \times U$ is the image of an embedding $\bR^{d+k} \to \bR^k \times \bR^n$ such that $\bR^{d+k} \to \bR^k$ is the projection onto the first $k$ coordinates. This produces a continuous map $\bR^k \to \Emb(\bR^d, \bR^n)\times \Bun(T\bR^d, \theta^*\gamma)$, which is smooth on the first coordinate. We may suppose that the origin in $\bR^k$ is sent to the standard embedding $e:\bR^d \hookrightarrow \bR^n$. There is a map, defined near $e$, from $\Emb(\bR^d, \bR^n)$ to $\Emb(\bR^d, \bR^d) \times C^\infty(\bR^d, \bR^{n-d})$, and we identify the second factor with the space $\Gamma(Ne(\bR^d))$. Let us write $B_k(r)$ for the ball of radius $r$ inside $\bR^k$. There is also a map $(-)^{-1} : \Emb(\bR^d, \bR^d) \dashrightarrow \Emb(B_d(2), \bR^d)$ given on a neighbourhood of the identity by $e \mapsto e^{-1} : B_d(2) \to \bR^d$, as on a neighbourhood of the identity embeddings $e$ contain the ball $B_d(2)$ in their image.

Composing with the action via pullback gives a map
\begin{eqnarray*}
\Emb(\bR^d, \bR^n) &\dashrightarrow & \Emb(\bR^d, \bR^d) \times \Gamma(Ne(\bR^d)) \\ &\dashrightarrow & \Emb(B_d(2), \bR^d) \times \Gamma(Ne(\bR^d)) \to \Gamma(Ne(B_d(2)))
\end{eqnarray*}
which sends an embedding $f$ to a section of $Ne(B_d(2))$ having the same image inside $B_n(1) \subset \bR^n$.

Choosing a compactly supported function $\varphi : B_d(2) \to [0,1]$ that is identically 1 inside the unit ball and multiplying sections by it gives a continuous map to $\Gamma_c(Ne(B_d(2)))$, so we have a sequence of continuous maps
$$\bR^k \to \Emb(\bR^d, \bR^n) \dashrightarrow \Gamma(Ne(B_d(2))) \overset{\cdot \varphi}\to \Gamma_c(Ne(B_d(2))).$$
In total, we obtain a continuous map
$$\bR^k \dashrightarrow \Gamma_c(Ne(B_d(2))) \times \Bun(TB_d(2), \theta^*\gamma) \dashrightarrow \Psi_\theta(B_n(2))$$
defined near the origin, that agrees with $\bR^k \to X \overset{f}\to \Psi_\theta(U) \to \Psi_\theta(\bR^n)$ after restricting both to the unit ball $B_n(1) \subset \bR^n$. In particular, by Theorem \ref{thm:sheaf} $f$ is continuous.
\end{proof}

\begin{lemma}
  Let $X$ be a smooth manifold and $f: X \to \Psi(U)$ be a continuous
  map.  Let $V \subseteq X \times U$ be open, and $W \subseteq X
  \times U$ be such that $\overline{V} \subseteq \Int(W)$.  Then there
  exists a homotopy $F: [0,1]\times X \to \Psi(U)$ starting a $f$,
  which is smooth on $(0,1] \times V \subseteq [0,1] \times X \times
  U$ and is the constant homotopy outside $W$.  Furthermore, if $f$ is
  already smooth on an open set $A \subseteq V$, then the homotopy can
  be assumed smooth on $[0,1] \times A$.
\end{lemma}
\begin{proof}
  Let us say that the open set $W \subseteq X \times U$ is
  \emph{small} if there are closed sets $K \subseteq X$, $L \subseteq
  U$, such that $W \subseteq K \times L$ and the composition
  \begin{align*}
    K \to X \to \Psi(U) \to \Psi(L)
  \end{align*}
  factors through a continuous map $K \to \Psi(U)^\mathit{cs}$.  Then
  we can use the manifold structure on $\Psi(U)^\mathit{cs}$ to find a
  homotopy $[0,1] \times K \to \Psi(U)^\mathit{cs}$ which is smooth on
  $(0,1] \times K$.  Using a suitable bump function we can ensure that
  it is the constant homotopy outside $W \subseteq K \times U$, and
  hence extends to a homotopy $[0,1]\times X \to \Psi(U)$ which is
  constant outside $W$.

  Next we consider the case where $W$ is the disjoint union of small
  open sets.  This is easy, since we can superimpose the corresponding
  homotopies.  After restricting the homotopy to $[0,\epsilon] \times
  X$ and composing with the linear diffeomorphism $[0,1] \times
  [0,\epsilon]$, we can assume that the homotopy is arbitrarily small.

  For general W, we first triangulate $X \times U$ such that every
  simplex is contained in a small open set.  For each $p$-simplex
  $\sigma \subseteq X \times U$, let $\st(\sigma) \subseteq X \times
  U$ be the open star of $\sigma$, thought of as a vertex of the
  barycentric subdivision.  Let $W^p \subseteq X \times U$ be the
  union of $\st(\sigma)$ over all $p$-simplices $\sigma$.  Then $X
  \times U$ is the (finite) union of the $W^p$.  We can pick slightly
  smaller open sets $V^p \subseteq W^p$ such that $\overline{V^p}
  \subseteq \Int(W^p)$ and such that the $V^p$ still cover $X \times
  U$.  Using the previous cases we can now proceed by induction as
  follows: First find a homotopy $F^0: [0,1] \times X \to \Psi(U)$
  which is constant outside $W^0 \subseteq X \times U$ and smooth on
  $(0,1] \times V^0$.  By construction, $[0,1]\times W^1$ is again a disjoint union of small sets, so we may find a homotopy $F^1: [0,1]^2 \times X \to
  \Psi(U)$ starting at $F^0$ which is constant outside $[0,1] \times
  W^1$ and smooth on $(0,1] \times [0,1] \times V^1$, and so on.  In
  the end we get a map $[0,1]^k \times X \to \Psi(U)$, and we can let
  $F: [0,1] \times X \to \Psi(U)$ be the restriction to the diagonal.
\end{proof}

\section{The homotopy type of spaces of manifolds}\label{sec:SoaM}

This section is a self-contained proof of the main theorem of \cite{GMTW}.

\subsection{Constructions with tangential structures}

We describe certain constructions that can be made relating $\theta$-manifolds and their submanifolds, using the tangential structure $\theta_{d-1}$ on $(d-1)$-manifolds.

\begin{definition}\label{defn:cross-with-R}
  If $M \in \Psi_{\theta_{d-1}}(\R^{n-1})$, let $\R \times M \subseteq
  \Psi(\R^n)$ have the $\theta$-structure induced by composing the obvious
  bundle map
  \begin{align*}
    T(\R^1\times M) \to \epsilon^1 \oplus TM
  \end{align*}
  with the bundle map $\epsilon^1 \oplus TM \to \theta^*\gamma$ defining the $\theta_{d-1}$-structure on $M$.
  This defines a map
  \begin{align}\label{eq:7:defn:cross-with-R}
\begin{aligned}
\Psi_{\theta_{d-1}}(\R^{n-1}) & \to \Psi_\theta(\R^n) \\
M &\mapsto \bR \times M.
\end{aligned}
  \end{align}
\end{definition}
The construction in Definition~\ref{defn:cross-with-R} can be
generalised a bit. Let $f : \bR \to \Psi_{\theta_{d-1}}(\bR^{n-1})$ be a smooth map, having graph $\Gamma(f) \subset \bR \times \bR^{n-1}$. Differentiating the first coordinate $x_1 : \Gamma(f) \to \bR$ gives a short exact sequence of vector bundles
\begin{align*}
  T^v(\Gamma(f)) \to T(\Gamma(f)) \to \epsilon^1
\end{align*}
and the standard inner product on $\bR^n$ gives an
inner product on $T(\Gamma(f))$, and so a
canonical splitting
\begin{align}\label{eq:20}
  \epsilon^1 \oplus T^v(\Gamma(f)) \xrightarrow{\cong} T(\Gamma(f)).
\end{align}
The $\theta_{d-1}$-structure on each fibre of $\Gamma(f)$ is described by a vector bundle map
\begin{align}
  \label{eq:21}
  \epsilon^1 \oplus T^v(\Gamma(f)) \to \theta^*\gamma.
\end{align}

\begin{definition}\label{defn:take-theta-graph}
  Let
  \begin{align*}
C^\infty(\R,\Psi_{\theta_{d-1}}(\R^{n-1})) &\to
    \Psi_\theta(\R^n)\\
    f & \mapsto \Gamma(f)
  \end{align*}
  be the map described above, where $\Gamma(f)$ is given the $\theta$
  structure obtained by composing (\ref{eq:21}) with the inverse of~(\ref{eq:20}).
\end{definition}

The process in Definition~\ref{defn:cross-with-R} is the special case
of Definition~\ref{defn:take-theta-graph} where the path $f$ is constant.
There is a partially defined reverse process which decreases
dimensions of manifolds by one.  Let $M \in
\Psi_\theta(\R^n)$, and again let $x_1: M \to \R$ be the projection to the
first coordinate in $\bR^n$.  If $a \in \R$ is a regular value of $x_1$, then
$M_a = M \cap x_1^{-1}(a)$ is a smooth manifold and we have a short exact
sequence
\begin{align*}
  TM_a \to TM|_{M_a} \to \epsilon^1.
\end{align*}
Using the inner product on $TM$ induced by $M \subseteq \R^n$ we get a
splitting
\begin{align}
  \label{eq:22}
  \epsilon^1 \oplus TM_a \xrightarrow{\cong} TM|_{M_a}.
\end{align}

\begin{definition}\label{defn:IntersectionMap}
  Let
  \begin{align}\label{eq:6:defn:IntersectionMap}
	\begin{aligned}
\Psi_\theta(\R^n) & \dashrightarrow
    \Psi_{\theta_{d-1}}(\R^{n-1})\\
    M & \mapsto M_a = M \cap x_1^{-1}(a) 
	\end{aligned}
  \end{align}
  be the partially defined map which gives $M_a$ the
  $\theta$-structure induced by composing~(\ref{eq:22}) with the
  bundle map $TM \to \theta^*\gamma$ defining the $\theta$-structure
  on $M$.
\end{definition}

Finally we discuss an extended functoriality of $\Psi_\theta$. Let $M \in \Psi_\theta(\R^n)$ and let $F: \R^n \to \R^n$ be a
map which is transverse to $M$.  Then $W = F^{-1}M$ is again a
$d$-manifold, but in general has no induced $\theta$-structure. For example if $F: \R^3
\to \R^3$ given by $F(x) = (0,0,|x|^2)$ then $F^{-1}(\R^2 \times
\{1\}) = S^2$, but the canonical framing of $\R^2 \times \{1\}$ cannot
induce a framing of $S^2$.

It is, however, possible to define such an induced structure in special cases.  The following will be needed below.  Let $\varphi: \R \to \R$ be a
smooth map with $\varphi' \geq 0$.  Let $F = \varphi \times \mathrm{Id}: \R^n \to
\R^n$.  If $M \in \Psi(\R^n)$ and $F$ is transverse to $M$, then $W =
F^{-1}(M)$ is again a $d$-manifold.  Both bundles $TW$ and $F^*(TM)$
are subbundles of $\epsilon^n$, and they are related via the
endomorphism $DF: \epsilon^n \to \epsilon^n$ by the equation
\begin{align*}
  TW = DF^{-1}(F^*TM).
\end{align*}
An exercise in linear algebra shows that the composition
\begin{align*}
  TW \hookrightarrow \epsilon^n \twoheadrightarrow F^*TM,
\end{align*}
the inclusion followed by orthogonal projection, defines an
isomorphism $TW \to F^*TM$.  Using this isomorphism to give $W$ an
induced $\theta$-structure, we get a partially defined map
\begin{align}
\begin{aligned}
  \label{eq:10}
\Psi_\theta(\R^n) & \dashrightarrow \Psi_\theta(\R^n)\\
  M & \mapsto F^{-1}M =   (\varphi \times \mathrm{Id})^{-1}(M).
\end{aligned}
\end{align}

A typical application of this construction is given by the following
lemma.  Let us say that an $M \in \Psi_\theta(\R^n)$ is
\emph{cylindrical} in $x_1^{-1}(a,b)$ if there is an $N \in
\Psi_{\theta_{d-1}}(\R^{n-1})$ such that
\begin{align*}
  (\R \times N)|_{x_1^{-1}(a,b)} = M |_{x_1^{-1}(a,b)} \in \Psi_\theta(
  x_1^{-1}(a,b)).
\end{align*}

\begin{lemma}\label{Cylindrical}
  Let $f: X \to \Psi_\theta(\R^n)$ be continuous, and let $U,V
  \subseteq X$ be open, with $\overline{U} \subseteq V$.  Let $a \in
  \R$ be a regular value for $x_1: f(x) \to \R$ for all $x \in V$.
  Let $\epsilon > 0$.  Then there is a homotopy
  \begin{align*}
    f_t: X \to \Psi_\theta(\R^n),\quad t \in [0,1]
  \end{align*}
  with $f_0 = f$, and
  \begin{enumerate}[(i)]
  \item $f_1(x)$ cylindrical in $x_1^{-1}([a-\epsilon,a+\epsilon])$
    for $x \in U$.
  \item The restriction to $[0,1] \times (X - V) \to
    \Psi_\theta(\R^n)$ is a constant homotopy.
  \item The composition $[0,1]\times X \to \Psi_\theta(\R^n) \to
    \Psi_\theta(\R^n - x_1^{-1}([a-2\epsilon,a+2\epsilon]))$ is a
    constant homotopy.
  \end{enumerate}
\end{lemma}
\begin{proof}
  Choose once and for all a smooth function $\lambda: \R \to \R$ with
  $\lambda(s) = 0$ for $|s| \leq 1$, $\lambda(s) = s$ for $|s| \geq
  2$, and $\lambda'(s) > 0$ for $|s|> 1$.  For $\tau \in [0,1]$, let
  $\varphi_\tau(s) = (1-\tau) s + \tau \epsilon\lambda(\frac{s-a}{\epsilon})$. Then
  pick a function $\rho: X \to [0,1]$ with $\supp(\rho) \subseteq V$
  and $U \subseteq \rho^{-1}(1)$ and define the homotopy by
  \begin{align*}
    [0,1] \times X & \to \Psi_\theta(\R^n)\\
    (t,x) &\mapsto (\varphi_{t\rho(x)} \times \mathrm{Id})^{-1} (f(x)).\qedhere
  \end{align*}
\end{proof}

Finally, let us introduce subspaces of $\Psi_\theta(\R^n)$ which are
very important in the sequel.
\begin{definition}\label{defn:SpacesOfManifolds}
  Let $\psi_\theta(n,k) \subseteq \Psi_\theta(\R^n)$ be the subspace
  defined by the condition $M \subseteq \R^k \times (-1,1)^{n-k}$. Let $\psi_{\theta_{d-1}}(n-1,k-1) \subseteq \Psi_{\theta_{d-1}}(\bR^{n-1})$ be the subspace defined by the condition $M \subseteq \bR^{k-1} \times (-1,1)^{n-k}$.
\end{definition}
We see that the map~(\ref{eq:7:defn:cross-with-R}) and the partially defined
map~(\ref{eq:6:defn:IntersectionMap}) restrict to maps
\begin{align}
\begin{aligned}\label{eq:17}
\psi_{\theta_{d-1}}(n-1,k-1) &\to \psi_\theta(n,k)\\
M &\mapsto \bR \times M,
\end{aligned}
\end{align}
\begin{align}
\begin{aligned}\label{eq:19}
\psi_\theta(n,k) &\dashrightarrow \psi_{\theta_{d-1}}(n-1,k-1) \\
M &\mapsto M_a = M \cap x_k^{-1}(a).
\end{aligned}
\end{align}

\begin{proposition}\label{prop:IdentifyingComponents}
  For $k > 1$, the map~(\ref{eq:17}) induces an isomorphism on
  $\pi_0$, with inverse induced by~(\ref{eq:19}).  Consequently $\pi_0
  \big(\psi_\theta(n,k)\big) \cong \pi_0 \big( \psi_{\theta_{d-k+1}}(n-k+1,1)\big)$.
\end{proposition}
\begin{proof}
By Sard's theorem, every element $M \in \psi_\theta(n,k)$ is in the domain of the map (\ref{eq:19}) for some $a$. We will show that the element $[M_a] \in \pi_0 \big(\psi_{\theta_{d-1}}(n-1,k-1)\big)$ is independent of $a$.

Let $M \in \psi_\theta(n,k)$ and $a < b$ be two regular values of $x_1 : M \to \bR$. Near $x_1^{-1}(a)$ and $x_1^{-1}(b)$, the function $x_1$ has no critical points. We can perturb it, relative to these subsets, to be a Morse function (as they are dense in the space of smooth functions) and still give an embedding (as the space of embeddings is open in the space of all smooth maps). We then obtain a manifold $M'$, such that $x_1$ has isolated critical points. We may then perturb $M'$ slightly so that there are no critical points in $[a, b] \times \{0\}^{k-1} \times  [-1,1]^{n-k}$. Let us still call this $M'$. As critical points are isolated, there is an $\epsilon > 0$ such that there are no critical points in $[a,b] \times (-\epsilon, \epsilon)^{k-1} \times [-1,1]^{n-k}$. Choosing an isotopy of embeddings $e_t : \bR \to \bR$ from the identity to a diffeomorphism onto $(-\epsilon, \epsilon)$, we can form $h_t = \bR \times e_t^{k-1} \times [-1,1]^{n-k}$ and let $M'(t) = h_t^{-1}(W')$.

This gives a path from $M' = M'(0)$ to a manifold $M'(1)$ such that $x_1$ has no critical values in $[a,b]$. Thus there are paths
$$M_a = M'_a \overset{M_a'(t)}\leadsto M_a'(1) \overset{M_t'(1)}\leadsto M_b'(1) \overset{M_b'(t)}\leadsto M'_b = M_b.$$

We have shown that (\ref{eq:19}) gives a well defined map $\psi_\theta(n,k) \to \pi_0 \big(\psi_{\theta_{d-1}}(n-1,k-1)\big)$. This map is locally constant so it factors through $\pi_0\big( \psi_{\theta}(n,k) \big)$. The composition
$$\pi_0\big(\psi_{\theta_{d-1}}(n-1, k-1)\big) \to \pi_0\big(\psi_\theta(n,k)\big) \to \pi_0\big(\psi_{\theta_{d-1}}(n-1, k-1)\big)$$
is the identity map, so the first map is injective. To see surjectivity let $M \in \psi_\theta(n,k)$ and $a$ be a regular value of $x_1 : M \to \bR$. Similarly to the proof of Lemma \ref{Cylindrical}, let $\varphi_t : \bR \to \bR$, $t \in [0,1]$ be given by $\varphi_t(s) = (1-t)\cdot (s-a) + a$. Then $M(t) = (\varphi_t \times \mathrm{Id})^{-1}(W)$ gives a path from $M$ to the cylindrical manifold $\bR \times M_a$, where tangential structures are handled using the extended functoriality of (\ref{eq:10}). Thus the first map is also surjective, and so a bijection.
\end{proof}

\subsection{The cobordism category}\label{sec:CobordismCategoryDefinition}

Let us give a definition of the embedded cobordism category
$C_\theta(\R^n)$ from \cite{GMTW}, using the topological sheaf $\Psi_\theta$
from the previous section.  There are several versions of the
definition, but they all give homotopy equivalent categories, where we say a functor $F: C \to D$ is a \textit{homotopy equivalence of categories} if $N_kF: N_k C \to N_k D$
is a homotopy equivalence for all $k$.  

\begin{definition}\label{defn:CobordismCategory}
Let $C_\theta(\bR^n)$ have object space $\psi_{\theta_{d-1}}(n-1,0)$.  The set of
  non-identity morphisms from $M_0$ to $M_1$ is the set of $(t,W) \in
  \R \times \psi_{\theta}(n,1)$ such that $t > 0$ and such that there exists an
  $\epsilon > 0$ such that
  \begin{align*}
    W|_{(-\infty,\epsilon) \times \R^{n-1}} &= (\R \times
    M_0)|_{(-\infty,\epsilon)\times\R^{n-1}} \in \Psi_\theta((-\infty, \epsilon)\times \bR^{n-1})\\
    W|_{(t-\epsilon,\infty) \times \R^{n-1}} &= (\R \times
    M_1)|_{(t-\epsilon,\infty)\times\R^{n-1}} \in \Psi_\theta((t-\epsilon, \infty)\times \bR^{n-1}),
  \end{align*}
  where $\R \times M_\nu \in \Psi_\theta(\bR^n)$ is as explained in~(\ref{eq:7:defn:cross-with-R}).  Composition in the category is defined by
  \begin{align*}
    (t,W) \circ (t',W') = (t + t', W''),
  \end{align*}
  where $W''$ agrees with $W$ near $(-\infty,t] \times \R^{n-1}$ and
  with $t \cdot e_1 + W'$ near $[t,\infty) \times \R^{n-1}$.  The total space
  of morphisms is topologised as a subspace of $(\{0\} \amalg
  (0,\infty)) \times \psi_\theta(n,1)$, where $(0,\infty)$ is given
  the usual topology.
\end{definition}
Of course, the important part of a morphism $(a_0 < a_1, M)$ is the
part of $M$ that lies in $[a_0,a_1]\times \R^{n-1}$, since it uniquely determines the rest of $M$.

The second version of the definition of the cobordism category is a topological
poset.
\begin{definition}
  Let
  \begin{align*}
    D_\theta(\bR^n) \subseteq \R \times \psi_\theta(n,1)
  \end{align*}
  denote the space of pairs $(t, M)$ such that $t$ is a regular value of
  $f: M \to \R$.  Topologise $D_\theta(\bR^n)$ as a subspace and order it by
  declaring $(t,M) \leq (t',M')$ if and only if $M = M'$ and $t \leq
  t'$.
\end{definition}

\begin{theorem}\label{thm:CobordismCategoryEquivalence}
There is a zig-zag
$$C_\theta(\bR^n) \overset{c}\lla D_\theta^\perp(\bR^n) \overset{i}\lra D_\theta(\bR^n)$$
of homotopy equivalences of categories. In particular, the classifying spaces $BC_\theta(\bR^n)$ and $BD_\theta$ are homotopy equivalent.
\end{theorem}
\begin{proof}
We must first describe $D_\theta^\perp(\bR^n)$. It is defined as $D_\theta(\bR^n)$, except that we only allow pairs $(t, M)$ such that $M$ is cylindrical in $x_1^{-1}(t-\epsilon, t+\epsilon)$ for some $\epsilon>0$. The functor $D_\theta^\perp(\bR^n) \lra D_\theta(\bR^n)$ is by inclusion. It is a levelwise homotopy equivalence on simplicial nerves by Lemma \ref{Cylindrical}.

Let us digress for a moment to a construction similar to (\ref{eq:10}). Let $\varphi_s(a,b)$ be the function as in the following diagram, where we allow $s=\infty$ to be the obvious limit. Note that it is smooth away from the set of points $\{a-s, a, b, b+s\}$, and a diffeomorphism on $(-\infty, a-s] \cup [a,b] \cup [b+s, \infty)$.

\begin{center}
\includegraphics[bb = 152 608 274 718, scale=1]{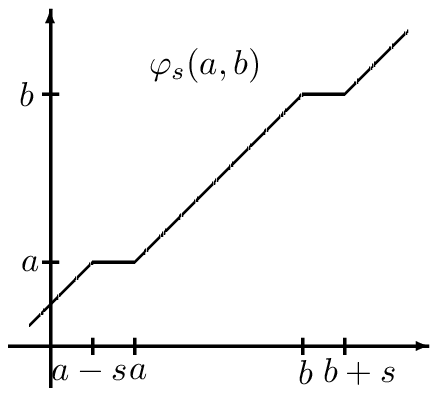}
\end{center}

Suppose $M \in \psi_\theta(n, 1)$ is cylindrical near $a$ and $b$, then $(\varphi_s(a,b) \times \mathrm{Id})^{-1}(M)$ defines an element of $\psi_\theta(n, 1)$ as follows. It certainly defines a smooth submanifold of $\bR \times (-1,1)^{n-1}$ that agrees with the underlying manifold of $M$ inside $x_1^{-1}(a,b)$, is cylindrical in $x_1^{-1}(a-s,a)$ and $x_1^{-1}(b, b+s)$, and agrees with translated copies of the underlying manifold $M$ on $x_1^{-1}(-\infty, a-s]$ and $x_1^{-1}[b+s, \infty)$. The $\theta$-structure is defined to be that of $M$ on $x_1^{-1}((-\infty, a-s) \cup (a,b) \cup (b+s, \infty))$ where there are diffeomorphisms to $M$. On $x_1^{-1}[a-s,a]$ the underlying manifold is cylindrical and so agrees with $\bR \times x_1^{-1}\{a\}$. The $\theta$-structure on $x_1^{-1}(\{a-s,a\})$ also agrees, so we can use this identification to give a $\theta$-structure on $x_1^{-1}[a-s,a]$. Similarly $x_1^{-1}[b,b+s]$.

\vspace{2ex}

The functor $D_\theta^\perp(\bR^n) \lra C_\theta(\bR^n)$ sends an object $(t, M)$ to $x_1^{-1}(t)$ as in (\ref{eq:6:defn:IntersectionMap}). On non-identity morphisms it is
$$(t_0 < t_1, M) \mapsto (t_1-t_0, (\varphi_\infty(t_0, t_1) \times \mathrm{Id})^{-1}(M) -t_0 \cdot e_1),$$
where $e_1 \in \bR^n$ is the first basis vector and $-t_0\cdot e_1$ denotes the parallel translation.

On simplicial nerves there is a map in the reverse direction, $N_l C_\theta(\bR^n) \to N_l D_\theta^\perp(\bR^n)$, that is the inclusion of those $(0 = t_0 < t_1 < \cdots < t_{l}, M)$ such that $M$ is cylindrical on $x_1^{-1}(-\infty, 0)$ and on $x_1^{-1}(t_{l}, \infty)$. The functor can then be viewed as the self map $h_1$ of $N_l D_\theta^\perp(\bR^n)$ that sends $(t_0 < t_1 < \cdots < t_{l}, M)$ to
$$(0 < t_1-t_0 < \cdots < t_{l}-t_0, (\varphi_\infty(t_0, t_{l}) \times \mathrm{Id})^{-1}(M) - t_0 \cdot e_1).$$
This is isotopic to the identity via the homotopy $h_s$, where $h_s$ sends $(t_0 < t_1 < \cdots < t_{l}, M)$ to
$$(t_0 - s \cdot t_0 < t_1-s \cdot t_0 < \cdots < t_{l}-s\cdot t_0, \left(\varphi_{\tfrac{s}{1-s}}(t_0, t_{l}) \times \mathrm{Id}\right)^{-1}(M) - s\cdot t_0 \cdot e_1).$$
\end{proof}

Finally we determine the homotopy type of $BC_\theta(\bR^n) \simeq
BD_\theta(\bR^n)$.
\begin{theorem}\label{thm:PosetModel}
  The forgetful map induces a weak equivalence
  \begin{align*}
    BD_\theta(\R^n) \stackrel{u}{\to} \psi_\theta(n,1).
  \end{align*}
\end{theorem}
\begin{proof}
  The inverse image of a point $M \in \psi_\theta(n,1)$ is a subspace
  of the infinite simplex $B(\R,\leq)$.  It is the simplex whose
  vertices is the space of regular values of $x_1: M \to \R$.  Points
  in the inverse image of $M$ can be represented as formal affine
  combinations (i.e.\ formal linear combinations where the
  coefficients are non-negative and sum to 1) of regular values of
  $x_1$.  This inverse image is contractible (the space of vertices is
  non-empty by Sard's theorem) which suggests the map might be a weak
  equivalence.  To give a rigorous proof we calculate the relative
  homotopy groups.

  Let
  \begin{align*}
    f: D^m &\to \psi_\theta(n,1) \\
    \hat f: \partial D^m &\to BD_\theta(\R^n)
  \end{align*}
  be continuous maps with $u \circ \hat f = f|_{\partial D^m}$.

  For $a \in \R$, let $U_a \subseteq D^m$ be the set of points $x$
  such that $a$ is a regular value of $x_1: f(x) \to \R$.  This is an
  open subset of $D^m$, so by compactness we can pick finitely many
  $a_1, \dots, a_k \in \R$ such that the $U_{a_i}$ cover $D^m$.  Pick
  a partition of unity $\lambda_1, \dots, \lambda_k: D^m \to [0,1]$
  subordinate to the cover.  Using $\lambda_i$ as a formal coefficient
  of $a_i$ gives a map
  \begin{align*}
    g: D^m \to BD_\theta(\R^n)
  \end{align*}
  which lifts $f$, i.e.\ $u \circ g = f$.  Finally we produce a
  homotopy between the two maps $g |_{\partial D^m}$ and $\hat f$.
  Since they are both lifts of $f|_{\partial D^m}$, we can just use the affine
  structure on the fibres of $u$ to give the straight-line homotopy.

  This proves that the relative homotopy groups (of $BD_\theta(\R^n)$
  as a subspace of the mapping cylinder of $u$) vanish and hence the
  map $u$ is a weak equivalence.
\end{proof}

We can now calculate the set of path components of $\psi_\theta(n,1)$.
We can define a product of two elements $W_1$ and $W_2$ of the space
$\psi_\theta(n,1)$ in the following way.  Take the union of the
disjoint manifolds $W_1$ and $W_2 + e_2$ and scale the second
coordinate by $1/2$.  This product
\begin{align*}
  \psi_\theta(n,1) \times \psi_\theta(n,1) \to \psi_\theta(n,1)
\end{align*}
makes $\psi_\theta(n,1)$ into an $H$-space (in fact it is an $E_{n-1}$
space) and hence $\pi_0 (\psi_\theta(n,1))$ is a monoid.  We have the
following corollary.
\begin{corollary}\label{cor:grouplike}
  The monoid $\pi_0 \psi_\theta(n,1)$ is isomorphic to the monoid of
  cobordism classes of $\theta_{d-1}$ manifolds in $\R^{n-1}$.  In
  particular the monoid is a group.
\end{corollary}
\begin{proof}
  By Theorem \ref{thm:PosetModel} and Proposition \ref{prop:IdentifyingComponents} we have
  \begin{align*}
    \pi_0(\psi_\theta(n,1)) = \pi_0(BD_\theta(\R^n)) =
    \pi_0(BC_\theta(\R^n)),
  \end{align*}
  which can be identified by the set objects of $C_\theta$, modulo the
  equivalence relation generated by the morphisms.  This proves the
  first claim.

  The monoid structure on $\pi_0(BC_\theta(\R^n))$ comes from an
  $H$-space structure on the objects and morphisms of $C_\theta$,
  defined by disjoint union in the second coordinate direction of
  $\R^n$, as for $\psi_\theta(n,1)$.  To see that it is a group, let
  $M \in \psi_{\theta_{d-1}}(n-1,0)$ be an object, and let $\R \times
  M \in \psi_\theta(n,1)$ be the corresponding cylindrical element.
  There is another object ``$-M$'', such that $\R \times (-M)$ is
  obtained from $\R \times M$ by changing signs of the first two
  coordinates in $\R^n$.

There is an embedding $e$ of $\bR \times I$ into itself given by the following picture. Then $(e \times I^{n-2}) (\bR \times M)$ gives a morphism from $M \coprod -M$ to the empty manifold, showing that $[M]$ and $[-M]$ are inverse points in the monoid structure on $\pi_0(B\mathcal{C}_\theta(\bR^n))$.
\begin{center}
\includegraphics[bb = 145 597 264 721, scale=1]{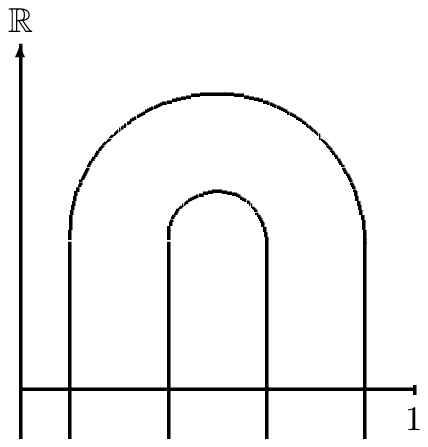}
\end{center}
\end{proof}

\subsection{The homotopy type of the space of all $\theta$-manifolds}
In this section we finish our proof of the main theorem of \cite{GMTW}. There is a space $\X(\bR^n)$ is defined by the fibre product
\begin{diagram}
\X(\bR^n) & \rTo & \X \\
\dTo^{\theta_{n}} & & \dTo^{\theta} \\
Gr_d(\bR^{n}) & \rTo^{inc} & BO(d)
\end{diagram}
where we consider $BO(d)$ as the infinite Grassmannian $Gr_d(\bR^\infty)$ and the inclusion map as that induced by $\bR^n \subset \bR^\infty$. Thus $\X(\bR^n)$ is simply that part of $\X$ which lies over those $d$-planes which are contained in $\bR^n \subset \bR^\infty$.
\begin{theorem}\label{thm:MainThmGMTW}
There is a weak homotopy equivalence
$$B \mathcal{C}_\theta(\bR^n) \simeq \psi_\theta(n, 1) \overset{\simeq} \lra \Omega^{n-1}\Th\big(\theta_{n}^*(\gamma_{d,n}^\perp) \to \X(\bR^n)\big).$$
Letting $n$ go to infinity, we get the main theorem of \cite{GMTW},
$$B \mathcal{C}_\theta \simeq \psi_\theta(\infty, 1) \overset{\simeq} \lra \Omega^{\infty-1} \MTtheta.$$
\end{theorem}
Theorem \ref{thm:MainThmGMTW} will be proved in Theorems \ref{thm:main-cpt-3} and \ref{thm:Scanning} below. First we construct the relevant maps. Consider the map
\begin{align}\label{eq:16}
\begin{aligned}
  \R \times \psi_\theta(n,k-1) &\to \psi_\theta(n,k)\\
  (t,M) &\mapsto M - t \cdot e_k,
\end{aligned}
\end{align}
where $e_k \in \R^n$ is the $k$th standard basis vector and $M - t \cdot e_k$
denotes the inverse image of $M$ under the diffeomorphism $x \mapsto x
+ t \cdot e_k$. This is a continuous map of spaces as it is induced by the action of the subgroup $\bR \subseteq \Diff(\bR^n)$.

If we let $\emptyset \in \psi_\theta(n,k)$ be the
basepoint, then~(\ref{eq:16}) extends uniquely to a continuous map
$S^1 \wedge \psi_\theta(n,k-1) \to \psi_\theta(n,k)$ with adjoint map
\begin{align}\label{eq:18}
  \psi_\theta(n,k-1) \to \Omega \psi_\theta(n,k).
\end{align}
\begin{theorem}\label{thm:main-cpt-3}
  The map~(\ref{eq:18}) is a homotopy equivalence for $k \geq 2$.
  Consequently there is a homotopy equivalence
  \begin{align*}
    \psi_\theta(n,1) \stackrel{\simeq}{\longrightarrow} \Omega^{n-1} \psi_\theta(n,n).
  \end{align*}
\end{theorem}

The homotopy equivalence~(\ref{eq:18}) will be proved in three steps, Propositions \ref{prop:DeloopPsi}, \ref{prop:IdentifyPsi0} and \ref{prop:Psi0IntoPsiDot}.  Let $k \geq 2$ be fixed.  The idea of the proof is to instead prove a homotopy
equivalence $B\psi_\theta(n,k-1) \simeq \psi_\theta(n,k)$ for a suitable ``monoid
structure'' on $\psi_\theta(n,k-1)$.  In fact it is convenient to work with
something which is not quite a monoid structure, but which contains the same
homotopical information.  We recall the following well known lemma.
\begin{lemma}\label{lem:monoid}
  Let $X_\bullet$ be a simplicial space such that the face maps induce
  a homotopy equivalence $X_k \simeq X_1 \times \dots \times X_1$.
  (When $k = 0$ this means that $X_0$ is contractible).  Then the
  natural map
  \begin{align*}
    X_1 \to \Omega \|X_\bullet\|
  \end{align*}
  is a homotopy equivalence if and only if $X_\bullet$ is
  \emph{group-like} i.e.\ $\pi_0 X_1$ is a group with respect to the
  product induced by $d_1: X_2 \to X_1$.
\end{lemma}
\begin{proof}
  Here $\Omega \|X_\bullet\|$ should be taken to mean the space of
  paths in the fat realisation $\|X_\bullet\|$ which start and end
  somewhere in $X_0 \subset \|X_\bullet\|$.  This space receives a natural map
  from $X_1$.  The lemma is a well known variant of the fact that $M
  \simeq \Omega BM$ if $M$ is a group-like topological monoid.  See \cite[Prop. 1.5]{SegalCatCohom}.

Given any pair $a, b \in X_0$, the space $\Omega_{a,b} \|X_\bullet\|$ of paths starting at $a$ and ending at $b$ is homotopy equivalent to $\Omega \|X_\bullet\|$ defined above.
\end{proof}
Instead of defining an actual monoid structure on $\psi_\theta(n,k-1)$ we
can define a simplicial space satisfying the lemma.  By abuse of
notation we will call it $N_\bullet \psi_\theta(n,k-1)$ since it plays the
role of the nerve of a monoid.  First we have a preliminary
definition.
\begin{definition}
  Let $A \subseteq \R$ be an open subset, and $x_k : \R^n \to \R$ denote the projection on the $k$th coordinate.  Let
  $\psi_\theta^A(n,k)\subseteq \psi_\theta(n,k)$ denote the subset defined by
  requiring $M \cap x_k^{-1}(\R - A) = \emptyset$.  In particular
  $\psi_\theta(n,k-1) = \psi_\theta^{(-1,1)}(n,k)$.

  Let $\psi_\theta^A(n,k)' \subseteq \psi_\theta(n,k)$ be the subset defined by the
  condition $M \cap [-1,1]^n \cap x_k^{-1}(\R - A) = \emptyset$.
\end{definition}

\begin{lemma}\label{lem:monoid-properties}
\ 
  \begin{enumerate}[(i)]
  \item\label{item:4:lem:monoid-properties} The inclusion $\psi_\theta^A(n,k) \to \psi_\theta^A(n,k)'$ is a homotopy equivalence.
  \item\label{item:1:lem:monoid-properties} If $A_1, \dots, A_l \subseteq \R$ are disjoint
    open sets and $A = \cup A_i$ then union of subsets defines a
    homeomorphism
    \begin{align*}
      \prod \psi_\theta^{A_i}(n,k) \xrightarrow{\cong} \psi_\theta^A(n,k).
    \end{align*}
  \item\label{item:2:lem:monoid-properties} If $A = (a_0,a_1)$ is an open interval then
    $\psi_\theta^A(n,k) \cong \psi_\theta(n,k-1)$.
  \item\label{item:3:lem:monoid-properties} If $A = (-\infty,a)$ or $A= (a,\infty)$ for some
    $a \in \R$, then $\psi_\theta^A(n,k)$ and $\psi_\theta^A(n,k)'$ are
    contractible.
  \end{enumerate}
\end{lemma}
\begin{proof}
  For part~(\ref{item:4:lem:monoid-properties}), we prove that the inclusion is a deformation
  retract.  Pick an isotopy $e_t: \R \to \R$, $t \in [0,1]$ of
  embeddings, with $e_0$ the identity and $e_1(\R) = (-1,1)$.  Let
  \begin{align*}
    j_t = (e_t)^{k-1} \times (\mathrm{Id})^{n-k+1}: \R^n \to \R^n.
  \end{align*}
  This gives a path $[0,1] \to \Emb(\R^n,\R^n)$, and we define a
  homotopy
  \begin{align*}
    \psi_\theta(n,k) \stackrel{h_t}{\longrightarrow} \psi_\theta(n,k),\quad t \in [0,1] 
  \end{align*}
  by $h_t(M) = e_t^{-1}(M)$.  This restricts to a deformation
  retraction because the homotopy preserves $\psi_\theta^A(n,k)$ and
  $\psi_\theta^A(n,k)'$, starts at the identity, and $h_1$ maps $\psi_\theta^A(n,k)'$
  into $\psi_\theta^A(n,k)$.

  (\ref{item:1:lem:monoid-properties}) is obvious and for~(\ref{item:2:lem:monoid-properties}), a homeomorphism is
  obtained by stretching the $k$th coordinate using the affine map $f: \R
  \to \R$ with $f(a_0)=-1$ and $f(a_1)=1$.  For~(\ref{item:3:lem:monoid-properties}) it suffices to consider $\psi_\theta^A(n,k)$ with $A =
  (-\infty,a)$.  We define a homotopy by
  \begin{align*}
    [0,1] \times \psi_\theta^{A}(n,k) &\to \psi_\theta^{A}(n,k)\\
    (t,M) & \mapsto M - \tfrac{t}{1-t} \cdot e_k \\
    (1,M) & \mapsto \emptyset.
  \end{align*}
  This contracts $\psi_\theta^A(n,k)$ to the point $\emptyset$.
\end{proof}
\begin{definition}\label{defn:cat-as-poset}
  Let $N_\bullet\psi_\theta(n,k-1)$ be the simplicial space defined by
  letting $N_l\psi_\theta(n,k-1)\subseteq \R^{l+1} \times \psi_\theta(n,k)$ be the
  set of pairs $(t,M)$ such that $0 < t_0 \leq \dots \leq t_l < 1$ and
  such that $M \in \psi_\theta^A(n,k)$ with $A = \R - \{t_0 , \dots t_l\}$.
\end{definition}
\begin{lemma}
  $N_\bullet \psi_\theta(n,k-1)$ satisfies the assumptions of Lemma~\ref{lem:monoid}.
\end{lemma}
\begin{proof}
  This follows immediately from Lemma~\ref{lem:monoid-properties}.
\end{proof}

It remains to see that $\pi_0N_1 \psi_\theta(n,k-1)$ is a group, but this follows
from Corollary \ref{cor:grouplike}.  We have proved
\begin{proposition}\label{prop:DeloopPsi}
  The natural map
  \begin{align*}
    \psi_\theta(n,k-1) \to \Omega |N_\bullet \psi_\theta(n,k-1)|
  \end{align*}
  is a homotopy equivalence.\qed
\end{proposition}
\begin{proposition}\label{prop:IdentifyPsi0}
  The forgetful map $(t,M) \mapsto M$ defines a homotopy equivalence
  \begin{align*}
    |N_\bullet\psi_\theta(n,k-1)| \to \psi_\theta^0(n,k),
  \end{align*}
  where $\psi_\theta^0(n,k)\subseteq \psi_\theta(n,k)$ is the subset where $M$
  satisfies $M \cap [-1,1]^n \cap x_k^{-1}(t) = \emptyset$ for some $t
  \in (0,1)$.
\end{proposition}
\begin{proof}
  There is a simplicial space $N_\bullet \psi_\theta(n,k-1)'$ defined as in
  Definition~\ref{defn:cat-as-poset}, but with $\psi_\theta^A(n, k-1)'$ instead of
  $\psi_\theta^A(n,k-1)$.  By Lemma~\ref{lem:monoid-properties}(\ref{item:4:lem:monoid-properties}), the
  inclusion $N_\bullet\psi_\theta(n,k-1) \to N_\bullet\psi_\theta(n,k-1)'$ is a levelwise homotopy equivalence, so it suffices to prove that $|N_\bullet
  \psi_\theta(n,k-1)'| \to \psi_\theta^0(n,k)$ is a homotopy equivalence.

  The fibre over $M \in \psi_\theta^0(n,k)$ is the classifying space of the
  poset of $t \in (0,1)$ such that $M \cap [-1,1]^n \cap x_k^{-1}(t) =
  \emptyset$, ordered as usual.  This is a totally ordered
  non-empty set, so the realisation is a simplex. Therefore the map is a weak homotopy equivalence as in the proof of Theorem \ref{thm:PosetModel}.
\end{proof}
\begin{proposition}\label{prop:Psi0IntoPsiDot}
  Let $\psi_\theta^\emptyset(n,k)\subseteq \psi_\theta(n,k)$ denote the path component
  of empty set.  Then the inclusion
  \begin{align*}
    \psi_\theta^0(n,k) \to \psi_\theta^\emptyset(n,k)
  \end{align*}
  is a weak homotopy equivalence.
\end{proposition}
\begin{proof}
  For ease of notation we will switch the roles of the coordinates
  $x_1$ and $x_k$ in this proof.  Thus $\psi_\theta^0(n,k)$
  becomes the subspace consisting of manifolds $M$ satisfying $M \cap x_1^{-1}(a) \cap [-1,1]^n = \emptyset$ for some $a$.  We prove
  that the relative homotopy groups vanish.  Let
  \begin{align*}
    f: (D^m, \partial D^m) \to (\psi_\theta^\emptyset(n,k),
    \psi_\theta^0(n,k))
  \end{align*}
  represent an element of relative $\pi_m$. We may assume $f$ is smooth.

  For each $a \in \R$, let $U_a \subseteq D^m$ the set of points $y
  \in D^m$ such that $x_1: f(y) \to \R$ has no critical points
  in $\{a\} \times I^{k-1} \times \R^{n-k}$.  This is an open
  condition on $f(y)$, so all $U_a\subseteq D^m$ are open.  As in the proof of Lemma \ref{lem:monoid-properties}(\ref{item:4:lem:monoid-properties}), pick an
  isotopy of embeddings $e_t: \R \to \R$, $t \in [0,1]$ starting at
  $e_0 = \mathrm{Id}$ and ending at a diffeomorphism $e_1: \R \to
  (-1,1)$.  Let $h_t: \R^n \to \R^n$ be given by $h_t = \mathrm{Id}
  \times e_t^{k-1} \times \mathrm{Id}$ and define a homotopy $f_t: D^m
  \to \psi_\theta^\emptyset(n,k)$, $t \in [0,1]$ by
  \begin{align*}
    f_t(x) = h_t^{-1}(f(x)).
  \end{align*}
  This gives a relative homotopy that starts at $f_0 = f$, and ends at a map
  $f_1$, where in $f_1(x) = h_1^{-1}(f(x))$ we have ``stretched''
  the space
  \begin{align*}
    \{a\} \times (-1,1)^{k-1} \times \R^{n-k}
  \end{align*}
  to be all of $\{a\} \times \R^{n-1}$.  Therefore $x_1: f_1(x) \to
  \R$ now has no critical points in $\{a\} \times \R^{n-1}$ for all $x
  \in U_a$.  We now replace our old $f$ by the homotopic $f_1$.

  By compactness of $D^m$, we can refine the cover by the $U_a$ to a
  cover by finitely many contractible sets $V_1, \dots, V_r \subseteq
  D^m$, with corresponding regular values $a_i \in \R$.  After
  possibly perturbing the $a_i$, we can assume they are different. We may choose an $\epsilon>0$ such that the intervals $(a_i-2\epsilon, a_i+2\epsilon)$ are disjoint.  By
  Lemma \ref{Cylindrical} we can suppose, after possibly changing $f$ by
  a homotopy concentrated in $(a_i-2\epsilon, a_i+2\epsilon)$, that for all $y \in V_i$, the element $f(y) \in
  \psi_\theta(n,k)$ is cylindrical in $x_1^{-1}(a_i-\epsilon, a_i+\epsilon)$, i.e.\ that there is an element $\lambda_i(y)
  \in \psi_{\theta_{d-1}}(n-1,k-1)$ such that the two elements
  \begin{align*}
    f(y) \text{ and }  (\R \times \lambda_i(y)) \in \Psi_\theta(\R^n)
  \end{align*}
  become equal in $\Psi_\theta(x_1^{-1}(a_i-\epsilon,a_i+\epsilon))$.
  Since $\pi_0 \big(\psi_\theta(n-1,k-1)\big) = \pi_0 \big(\psi_\theta(n,k)\big)$ by Proposition
  \ref{prop:IdentifyingComponents}, the element $\lambda_i(y)$ must be in the
  basepoint component of $\psi_\theta(n-1,k-1)$ and since the $V_i$ are
  contractible, we can pick a smooth homotopy
  \begin{align*}
[0,1] \times V_i \overset{\Lambda_i}\lra \psi_{\theta_{d-1}}(n-1,k-1),
  \end{align*}
  with $\Lambda_i(0,-) = \lambda_i$ and $\Lambda_i(1,-) = \emptyset$.

  Pick a $\delta > 0$ with $3 \delta < \epsilon$, and a smooth
  function $\rho: \R \to [0,1]$ which is 1 on $(-\delta,\delta)$ and
  has support in $(-2\delta,2\delta)$.  Finally pick $\tau_i: V_i \to
  [0,1]$ with compact support and with $\cup_i \tau_i^{-1}(1) = D^m$ and define a homotopy by
  \begin{align*}
    V_i & \overset{h_t^i}\to C^\infty(\R,\psi_{\theta_{d-1}}(n-1,k-1)), \quad t \in [0,1]\\
    x & \mapsto (b \mapsto \Lambda_i(t \tau_i(x)\rho(b - a_i),x)).
  \end{align*}
The homotopy starts with the map $h_0^i$ which sends all $x$ to the constant path at
  $\lambda_i(x)$. At any time $t$ the map $h_t^i(x) : \bR \to \psi_{\theta_{d-1}}(n-1,k-1)$ is constant outside of $(a_i - 2\delta,a_i + 2\delta)$. The homotopy ends at $h_1^i(x)$, which maps $(a_i-\delta, a_i+\delta)$ to the empty set.

By taking graphs (i.e.\
  composing with the function $\Gamma$ from Definition \ref{defn:take-theta-graph})
  we get a homotopy of maps
  \begin{align}\label{eq:27}
    \begin{aligned}
      V_i & \to \psi_\theta(n,k), \quad t \in [0,1]\\
      x & \mapsto \Gamma(h_t^i(x))
    \end{aligned}
  \end{align}
  which at $t=0$ is $x \mapsto \R \times \lambda_i(x)$, so it agrees
  with $f$ on $x_1^{-1}(a_i-3\delta, a_i+3\delta)$, and at $t=1$ maps any $x \in \tau_i^{-1}(1)$ to an
  element which is empty inside $x_1^{-1}(a_i-\delta,a_i+\delta)$.  At
  any time $t$, it agrees with $f$ when restricted to $x_1^{-1}(
  (a_i-3\delta,a_i+3\delta) - (a_i-2\delta,a_i+2\delta))$ so by the sheaf
  property of $\Psi^\theta$ we can define a homotopy of maps $V_i \to
  \psi_\theta(n,k)$ whose restriction to $x_1^{-1}(\R -
  (a-2\delta,a+2\delta))$ is the constant homotopy of $f|_{V_i}$ and
  whose restriction to $x_1^{-1}(a_i-3\delta,a_i+3\delta)$ is the
  homotopy~(\ref{eq:27}).  This homotopy is constant outside a compact
  subset of $V_i$, so it extends to a homotopy of the map $f$ which at
  time $t=1$ maps $\tau_i^{-1}(1)$ into $\psi_\theta^\emptyset(n,k)$.
  We have only changed $f$ inside
  $x_1^{-1}(a_i-2\delta,a_i+2\delta)$, so we can carry out this
  construction for other $a_j$'s as well.  In the end we have
  homotoped $f$ into a map to $\psi_\theta^\emptyset(n,k)$ as desired.

It remains to show that this is a relative homotopy. Suppose we have an $x \in \partial D^m$, so $f(x) \in \psi^0_\theta(n,k)$. Then there is a $t \in (0,1)$ that is a regular value of $x_1 : f(x) \to \bR$ such that $x_1^{-1}(t) = \emptyset$ (after we have replaced $f$ by $f_1$ as described in the second paragraph). If $t$ is not in $\coprod_i (a_i-\epsilon, a_i+\epsilon)$, then all the homotopies we perform are constant near height $t$, so the level set at $t$ is always empty. If $t \in (a_i-\epsilon, a_i+\epsilon)$ then $\lambda_i = \emptyset$, so we can choose $\Lambda_i$ to be constantly $\emptyset$. The homotopy is then constant near height $t$, so the level set at $t$ is always empty and we remain in $\psi^0_\theta(n,k)$.
\end{proof}
\begin{proof}[Proof of Theorem~\ref{thm:main-cpt-3}]
  We combine the three propositions to get the weak homotopy equivalences
  \begin{align*}
    \psi_\theta(n,k-1) \to \Omega|N_\bullet \psi_\theta(n,k-1)| \to \Omega
    \psi_\theta^0(n,k) \to \Omega \psi_\theta^\emptyset (n,k) = \Omega \psi_\theta(n,k).
  \end{align*}
\end{proof}

We will now explain how to identify the homotopy type of $\psi_\theta(n,n) = \Psi_\theta(\bR^n)$. 

\begin{theorem}\label{thm:Scanning}
There is a homotopy equivalence
$$\Psi_\theta(\bR^n) \simeq \Th(\theta_{n}^*(\gamma_{d,n}^\perp) \to \X(\bR^n)),$$
where $\gamma_{d,n}^\perp$ is the orthogonal complement to the tautological bundle over $Gr_d(\bR^{n})$.
\end{theorem}
\begin{proof}
First let $\Psi_\theta(\bR^n)^\circ \subset \Psi_\theta(\bR^n)$ be the subspace of those $\theta$-manifolds which contain the origin. Write $L^\theta$ for the subspace of $\Psi_\theta(\bR^n)^\circ$ consisting of \textit{linear} $\theta$-manifolds, i.e.\ those where the underlying manifold is a $d$-plane and the $\theta$-structure is constant. There is a map $\X(\bR^n) \to L^\theta$ which sends a pair $(V, x)$ of a $d$-dimensional plane $V$ in $\bR^d$ and a point $x \in \X$ over $i(V)$ to the pair $(V, \ell)$ of an element of $Gr_d(\bR^n)$ and a bundle map $\ell : TV = V \times V \to \theta^*\gamma$ given by $(v, \bar{v}) \mapsto (x, i(\bar{v}))$. This gives a map of fibrations over $Gr_d(\bR^n)$, with map on fibres over $V$ given by the inclusion $\Fib(\theta) \to \Bun(V, \theta^*\gamma)$. This is a homotopy equivalence (it is essentially the inclusion of the fibre of $\theta$ into its homotopy fibre, and $\theta$ has been assumed to be a Serre fibration).

Let $F_t : \bR^n \to \bR^n$ be scalar multiplication by $(1-t)$ and define a homotopy $S: [0,1] \times \Psi_\theta(\bR^n)^\circ \lra \Psi_\theta(\bR^n)^\circ$ as follows: on underlying manifolds let
$$S(t, W) = \begin{cases}
F_t^{-1}(M) &\text{if} \,\,\, t < 1\\
T_0M  &\text{if} \,\,\, t=1.
\end{cases}$$
To define the $\theta$-structure on $S(t,M)$ we use the map $F_t \times \mathrm{Id} : S(t, M) \times \bR^n \to M \times \bR^n$, which restricts to a fiberwise linear isomorphism
\begin{align*}
T(S(t, M)) \to TM
\end{align*}
over $F_t$ (which is not the same as $DF_t$) and we give $S(t,M)$ the $\theta$-structure obtained by composition. This defines a continuous homotopy such that $S(0,W)=W$, $S(1, W) \in L^\theta$ and $S$ preserves $L^\theta$, so it gives a deformation retraction of $\Psi_\theta(\bR^n)^\circ$ to $L^\theta \simeq \X(\bR^n)$.

There is a $(n-d)$-dimensional vector bundle $\nu \lra \Psi_\theta(\bR^n)^\circ$ which at a point $W$ has fibre $\nu_0 W$ the normal space to the manifold at 0. There is a map $e:\nu \lra \Psi_\theta(\bR^n)$ sending $(W, v \in \nu_0 W)$ to the translated manifold $W + v$. Restricted to a neighbourhood of the 0-section in $\nu$, this gives an embedding onto the open subspace $U$ of $\Psi_\theta(\bR^n)$ of those manifolds having a unique closest point to the origin. The complement $C$ of this embedding consists of manifolds which do not have a unique closest point to the origin; in particular, they do not contain it. The isotopy $\cdot \frac{1}{1-t} : \bR^n \to \bR^n$, $t \in [0,1)$ produces a map $H:[0,1) \times C \to C$, as it moves points on a manifold uniformly away from the origin. We can extend it to a continuous map $H:[0,1] \times C \to C$ by $H(1, c) = \emptyset$, which gives a contraction of $C$.

The map $C \to \Psi_\theta(\bR^n)$ is a cofibration, as it is a push out of the cofibration $\nu \cap e^{-1}(C) \to \nu$, so collapsing $C$ gives a homotopy equivalent space. On the other hand, collapsing $C$ gives a space homeomorphic to that obtained by collapsing $\nu \cap e^{-1}(C)$ in $\nu$. This is the Thom space of $\nu$, so $\Psi_\theta(\bR^n) \simeq \Th(\nu \to \psi_\theta(n,n)^\circ)$.\qedhere
\end{proof}

Combining Theorems \ref{thm:main-cpt-3} and \ref{thm:Scanning} finishes the proof of Theorem \ref{thm:MainThmGMTW}.

\section{Proof of the main theorems}\label{ParametrisedSurgery}

In this section we will prove our main results, Theorems \ref{thm:SubcategoryConnected} and \ref{thm:SubcategoryMonoids}. The inclusion of a full subcategory $\mathcal{D}$ of $\mathcal{C}_\theta\bp$ into $\mathcal{C}_\theta$ will be considered in several steps. Recall that by Theorems \ref{thm:CobordismCategoryEquivalence} and \ref{thm:PosetModel}, $B \mathcal{C}_\theta \simeq \psi_\theta(\infty,1)$. We will give a similar model for $B\mathcal{D}$. We first describe some variations on the space $\psi_\theta(\infty, 1)$.

Recall from the introduction that we have chosen a $\theta$-structure on $\bR^d$, thought of as a vector bundle over a point, and there is an induced structure on any framed manifold called the \textit{standard $\theta$-structure}.

\begin{definition}
For $1>\epsilon>0$, let $\psi_\theta(n,1)^\epsilon \subseteq \psi_\theta(n,1)$ be the subspace where the manifold satisfies
  \begin{align*}
    M \subseteq \R \times (-1,1)^{d-1} \times [0,1)^{n-d}
  \end{align*}
  and
  \begin{align*}
    L_\epsilon = \R \times (-\epsilon,\epsilon)^{d-1} \times \{0\}
    \subseteq M,
  \end{align*}
and that the tangential structure $l$ is
  standard on $L_\epsilon$ with respect to the framing of $L_\epsilon$
  given by the vector fields $\partial/\partial x_1,
  \dots, \partial/\partial x_d$. If $\epsilon > \epsilon'$, there is an inclusion $\psi_\theta(n,1)^\epsilon \to \psi_\theta(n,1)^{\epsilon'}$. Define $\psi_\theta(n,1)^\bullet = \colim_\epsilon \psi_\theta(n,1)^\epsilon$, with the colimit topology. There is a continuous injection
$\psi_\theta(n,1)^\bullet \to \psi_\theta(n,1)$.

We define $\psi_{\theta_{d-1}}(n-1, 0)^\bullet$ similarly, where $\psi_{\theta_{d-1}}(n-1, 0)^\epsilon \subseteq \psi_{\theta_{d-1}}(n-1, 0)$ is the subspace of those manifolds $M$ that satisfy $M \subseteq (-1,1)^{d-1} \times [0,1)^{n-d}$ and $(-\epsilon, \epsilon)^{d-1} \times \{0\} \subseteq M$.
\end{definition}

In Definition \ref{defn:ThetaDotCategory} we briefly described the objects and morphisms of a category $\mathcal{C}_\theta\bp$. A more precise definition is as follows.

\begin{definition}
Let $\mathcal{C}_\theta\bp(\bR^n)$ have object space the subspace of $\psi_{\theta_{d-1}}(n-1,0)\bp$ consisting of connected manifolds. The set of non-identity morphisms from $M_0$ to $M_1$ is the set of $(t, W) \in \bR \times \psi_\theta(n, 1)\bp$ such that $t>0$, there is an $\epsilon>0$ such that
\begin{align*}
    W|_{(-\infty,\epsilon) \times \R^{n-1}} &= (\R \times
    M_0)|_{(-\infty,\epsilon)\times\R^{n-1}} \\
    W|_{(t-\epsilon,\infty) \times \R^{n-1}} &= (\R \times
    M_1)|_{(t-\epsilon,\infty)\times\R^{n-1}},
  \end{align*}
and such that $W \cap [0, t] \times \bR^{n-1}$ is connected.  Composition in the category is as in Definition \ref{defn:CobordismCategory}. The total space
  of morphisms is topologised as a subspace of $(\{0\} \amalg
  (0,\infty)) \times \psi_\theta(n,1)\bp$, where $(0,\infty)$ is given
  the usual topology. We shall be mostly concerned with the colimit $\mathcal{C}_\theta\bp = \colim_{n \to \infty} \mathcal{C}_\theta\bp(\bR^n)$
\end{definition}

\begin{definition}
Let the subspace
$$\psi^{nc}_{\theta}(n, 1)\bp \subset \psi_{\theta}(n, 1)\bp$$
consist of those manifolds $W$  having \textit{no compact path components}.
\end{definition}

\begin{definition}
For $\mathbf{C}$ a collection of elements of $\psi_{\theta_{d-1}}(\infty,0)$, let the subspace
$$\psi_{\theta}(\infty, 1)_{\mathbf{C}} \subset \psi_{\theta}(\infty, 1)$$
consist of those manifolds $W$ for which there exists a regular value $t \in (-1,1)$ of $x_1 : W \to \bR$ such that $W_t$ is in $\mathbf{C}$. Write $\mathbf{Conn}$ for the collection of elements of $\psi_{\theta_{d-1}}(\infty,0)$ which are connected manifolds. Thus $\psi_{\theta}(\infty, 1)_{\mathbf{Conn}}$ consists of those $W$ such that some $W_t$ is connected, for $t \in (-1,1)$.
\end{definition}
In this definition we have insisted on regular values in $(-1,1)$ exhibiting an element as a member of the space $\psi_\theta(\infty,1)_\mathbf{C}$. If we merely ask for such regular values in $\bR$, we get a weakly homotopy equivalent space, but the current definition will simplify certain constructions later.

\begin{theorem}\label{thm:HtpyTypeCobCat}
Let $\mathbf{C}$ be a collection of objects of $\mathcal{C}_\theta\bp$ (so $\mathbf{C} \subseteq \mathbf{Conn}$), and $\mathcal{D}$ be the full subcategory on $\mathbf{C}$. Then there is a weak homotopy equivalence
$$B \mathcal{D} \simeq \psi_{\theta}^{nc}(\infty, 1)\bp_{\mathbf{C}}.$$
\end{theorem}
\begin{proof}
This is exactly as the proof of Theorem \ref{thm:CobordismCategoryEquivalence}. Note that if $W \in \psi_\theta^{nc}(\infty,1)\bp$ and
  $a_0$ and $a_1$ are two regular values of $x_1: W \to \bR$ such that
  $W_{a_\nu} \in \mathbf{C}$, then the fact that elements of $\mathbf{C}$ are
  connected and $W$ is non-compact implies that the manifold
  $W \cap x_1^{-1}([a_0,a_1])$ is also connected. The analogous $D_\theta(\bR^\infty)_{\mathbf{C}}\bp$ is then the topological poset consisting of pairs $(t,W) \in \bR \times \psi_\theta^{nc}(\infty,1)\bp$ with $t$ regular for $x_1 : W \to \bR$ and $W_t \in \mathbf{C}$. This allows one to mimic the proof of Theorem \ref{thm:CobordismCategoryEquivalence}.

However the space $\psi_{\theta}^{nc}(\infty, 1)\bp_{\mathbf{C}}$ has such regular values in $(-1,1)$, not merely in $\bR$. Thus it does not have a map from the poset $D_\theta(\bR^\infty)_{\mathbf{C}}\bp$, but only from the full subposet $\mathcal{P} \subseteq D_\theta(\bR^\infty)_{\mathbf{C}}\bp$ of pairs $(t,W)$ with $t \in (-1,1)$. We must show that the inclusion of this subposet gives an equivalence on classifying-spaces. This is so as it induces a levelwise homotopy equivalence on simplicial nerves: a homotopy inverse to the inclusion $N_k \mathcal{P} \to N_k D_\theta(\bR^\infty)_{\mathbf{C}}\bp$ is giving by affine scaling in the $\bR$ direction until all regular values lie in the interval $(-1,1)$.

As in the proof of Theorem \ref{thm:PosetModel} the fibre of $\mathcal{P} \to \psi_{\theta}^{nc}(\infty, 1)\bp_{\mathbf{C}}$ over $W$ becomes the simplex with vertices all possible choices of $a \in (-1,1)$ with $W_a \in \mathbf{C}$, which is contractible.
\end{proof}

To finish the proof of Theorem~\ref{thm:SubcategoryConnected} we need to
prove that the inclusion
\begin{align}
  \label{eq:23}
  \psi^{nc}_\theta(\infty,1)^\bullet_\mathbf{Conn} \to \psi_\theta(\infty,1)
\end{align}
is a weak equivalence.  This will be done
in the rest of this section, and will be broken up into several steps
that we treat separately.  The inclusion~(\ref{eq:23}) factors as
\begin{align*}
\psi^{nc}_\theta(\infty,1)^\bullet_\mathbf{Conn} \to
  \psi_\theta^{nc}(\infty,1)^\bullet \to \psi_\theta(\infty,1)^\bullet \to
  \psi_\theta(\infty,1).
\end{align*}
Starting from the right, we prove in
Lemma~\ref{lem:AddingBasepoint} below that the inclusion
$\psi_\theta(n,1)^\bullet \to \psi_\theta(n,1)$ is a homotopy
equivalence by giving two explicit maps that are homotopy inverse.
Then in Lemma~\ref{lem:ElimCpctComponents} we show that
$\psi_\theta^{nc}(n,1)^\bullet \to \psi_\theta(n,1)^\bullet$ is a weak
homotopy equivalence.  These steps are both fairly easy.  In
\S\ref{sec:ProofThmA} we prove that $\psi^{nc}_\theta(\infty,1)^\bullet_\mathbf{Conn} \to \psi_\theta^{nc}(\infty,1)^\bullet$ is a weak equivalence. For the proof we use
0-surgery to make objects connected, but in order to get a homotopy
equivalence we need a fairly elaborate procedure for doing 0-surgeries
in families.

\begin{lemma}\label{lem:AddingBasepoint}
The inclusion $i: \psi_{\theta}(n, 1)\bp \lra \psi_{\theta}(n, 1)$ is a weak homotopy equivalence.
\end{lemma}
\begin{proof}
For $X$ a compact space, any map $X \to \psi_\theta(n,1)^\bullet$ factors through some $\psi_\theta(n,1)^\epsilon$ by a standard property of the colimit topology. Thus it is enough to show that $\psi_\theta(n,1)^\epsilon \to \psi_\theta(n,1)$ is a weak homotopy equivalence.

We can define a product $\amalg : \psi_\theta(n,1) \times \psi_\theta(n,1) \to \psi_\theta(n,1)$ as follows. The manifold $W_1 \amalg W_2$ is obtained by taking the union of the disjoint manifolds $W_1$ and $W_2 + e_{d+1}$ and scaling the $(d+1)$st coordinate by $\tfrac{1}{2}$.

Now pick a ``cylinder'' $W_0 \in \psi_\theta(n,1)^\epsilon$ and define a map
\begin{align*}
  c: \psi_\theta(n,1) & \to \psi_\theta(n,1)^\epsilon\\
  M & \mapsto W_0 \amalg M.
\end{align*}
Note that $W_0 \amalg M$ does lie in $\psi_\theta(n,1)^\epsilon$ by construction of the product. Then the composition
\begin{align*}
  \psi_\theta(n,1) \xrightarrow{c} \psi_\theta(n,1)^\epsilon \xrightarrow{i}
  \psi_\theta(n,1)
\end{align*}
is $M \mapsto i(W_0) \amalg M$, which is homotopic to the
identity map if we pick $W_0$ in the component of the basepoint.  The
effect of the reverse composition
\begin{align*}
  \psi_\theta(n,1)^\epsilon \xrightarrow{i} \psi_\theta(n,1) \xrightarrow{c}
  \psi_\theta(n,1)^\epsilon
\end{align*}
is shown in the following figure
\begin{center}
\includegraphics[bb = 152 623 420 716, scale=1]{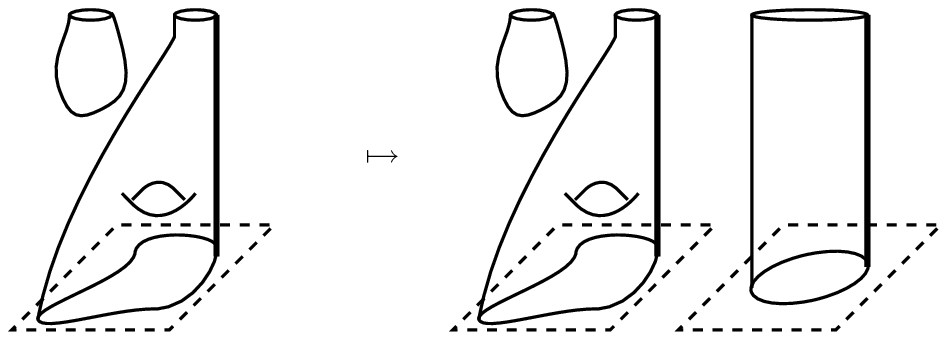}
\end{center}
and the homotopy to the identity map is as follows.
\begin{center}
\includegraphics[bb = 156 623 297 716, scale=1]{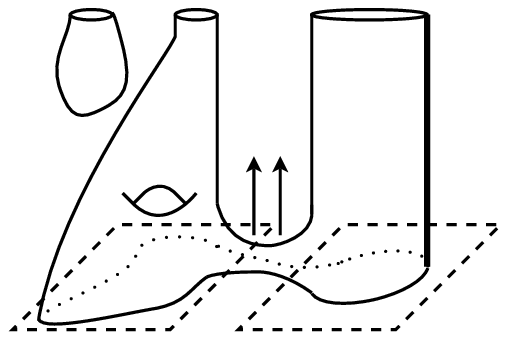}
\end{center}
\end{proof}

\begin{lemma}\label{lem:ElimCpctComponents}
The inclusion $\psi^{nc}_{\theta}(n, 1)\bp \lra \psi_{\theta}(n, 1)\bp$ is a weak homotopy equivalence.
\end{lemma}
\begin{proof}
Again we will show that the relative homotopy groups vanish. Let $f: (D^m, \partial D^m) \lra (\psi_{\theta}(n, 1)\bp, \psi^{nc}_{\theta}(n, 1)\bp)$ be a smooth map. We will show it is relatively homotopic to a map into $\psi^{nc}_{\theta}(n, 1)\bp$.

Consider the height function $x_1: \Gamma(f) \lra \bR$. After perturbing $f$, we way assume that for all $x \in D^m$, no path component of $f(x)$ is contained in $x_1^{-1}(0)$. For each $x \in X$, we can choose an $\epsilon_x > 0$ such that no path component of $f(x)$ is contained in $f(x) \cap x_1^{-1}[-\epsilon_x, \epsilon_x]$. There is an open neighbourhood $U_x \ni x$ for which this is still true.

The $\{  U_x \}_{x \in X}$ form an open cover of $X$: take a finite subcover $\{ U_i \}$, and let $\epsilon = \min(\epsilon_i)$. Then there is no path component of $f(x)$ in $x_1^{-1}[-\epsilon, \epsilon]$ for any $x \in X$. Choosing an isotopy of embeddings from the identity on $\bR$ to a diffeomorphism that takes $\bR$ to $(-\epsilon, \epsilon)$ defines a relative homotopy of the map $f$ into the subspace $\psi^{nc}_{\theta}(n, 1)\bp$.
\end{proof}

Our aim is to show that the inclusion
$$\psi^{nc}_{\theta}(\infty, 1)\bp_{\mathbf{Conn}} \lra \psi^{nc}_{\theta}(\infty, 1)\bp$$
is a weak homotopy equivalence, for any dimension $d$ and for all tangential structures $\theta$ such that $\X$ is connected and $S^2$ admits a $\theta$-structure. The idea for doing so is relatively simple, and proceeds by showing the relative homotopy groups $\pi_k(\psi^{nc}_{\theta}(\infty, 1)\bp_{\mathbf{Conn}}, \psi^{nc}_{\theta}(\infty, 1)\bp)$ are trivial.

It is instructive to first consider the case $k=0$. Suppose we are given a $W \in \psi^{nc}_{\theta}(\infty, 1)\bp$. We wish to produce a $\overline{W} \in \psi^{nc}_{\theta}(\infty, 1)\bp_{\mathbf{Conn}}$ and a path from $W$ to $\overline{W}$ in $\psi^{nc}_{\theta}(\infty, 1)\bp$. The condition to belong to $\psi^{nc}_{\theta}(\infty, 1)\bp_{\mathbf{Conn}}$ is the existence of a level set which is connected, so let us choose a regular value $a_0 \in (-1,1)$ of $x_1 : W \to \bR$ and write $W_{a_0} = W \cap x_1^{-1}(a_0)$ for the level set. In general $W_{a_0}$ will not be connected, so we intend to obtain $\overline{W}$ by performing 0-surgery on $W$ near $a_0$, so as to make the level set at $a_0$ connected: the result $\overline{W}$ of this surgery will lie in $\psi^{nc}_{\theta}(\infty, 1)\bp_{\mathbf{Conn}}$.

To carry out 0-surgery on an element of $\psi_\theta(\infty,1)$ we need to explain how to give a height function and a $\theta$-structure on the handles we attach. We do this in the following section.

\subsection{Parametrised 0-surgery}
\label{sec:surgery}

Recall that if $W$ is an (abstract) smooth $d$-manifold and $e: S^0
\times D^d \to W$ is an embedding, then performing 0-surgery on $W$
along $e$ amounts to removing $e(S^0 \times \Int(D^d))$ from $W$ and
gluing in $D^1 \times S^{d-1}$ along $S^0 \times S^{d-1}$.  In our
setup, $W \in \psi^\theta(\infty,1)$ has more structure, the essential
parts of which are the height function $x_1: W \to \R$ and the
$\theta$ structure $l: TW \to \theta^*\gamma$.  We explain how to
extend these to the manifold $\overline{W}$ resulting from
performing surgery.

\subsubsection{$\theta$-structures}
\label{sec:structures}

The following depends critically on our assumptions on $\theta: \X \to BO(d)$, namely that $\X$ be path connected and that $S^d$ admit a $\theta$-structure.
\begin{proposition}\label{prop:extend-structures}
  Let $W$ be a smooth $d$-manifold equipped with a $\theta$-structure
  $l: TW \to \theta^*\gamma$ and let
  \begin{align}
    \label{eq:36}
    e: S^0 \times D^d \to W
  \end{align}
  be an embedding.  Let $\K^e(W)$ denote the result of performing
  surgery on $W$ along $e$.  After possibly changing $e$ by a
  self-diffeomorphism of $D^d$, the $\theta$-structure
  on the submanifold
  \begin{align*}
    W - e(S^0 \times \Int D^d) \subseteq \K^e(W)
  \end{align*}
  extends to a $\theta$-structure on $\K^e(W)$.
\end{proposition}

\begin{proof}
  After cutting out $e(S^0 \times \Int D^d) \subseteq W$ and gluing in
  the $D^1 \times S^{d-1}$, the $\theta$-structure on $W$ gives a
  bundle map
\begin{align}\label{eq:37}
  T(D^1 \times S^{d-1})|_{\partial D^1 \times S^{d-1}} \to
  \theta^*\gamma,
\end{align}
and the problem becomes to extend this to a bundle map over all of
$D^1 \times S^{d-1}$.  This obstruction theoretic question only
depends on the homotopy class of the $\theta$ structure on $W$ over
each sphere $e(\{\pm 1\} \times \partial D^d)$.

The space of $\theta$ structures on $D^d$ is homotopy equivalent to
the fibre of the fibration $\theta: \X \to BO(d)$.  The assumption
that $\X$ be path connected implies that $\pi_1 BO(d)$ acts
transitively on this fibre, so up to possibly changing the sign of a
coordinate in $D^d$, all $\theta$ structures on it are homotopic.  In
particular we can arrange that after possibly composing the embedding
(\ref{eq:36}) with a self-diffeomorphism of $S^0 \times D^d$, the
$\theta$ structures on $D^d$ induced by the embeddings $e(\pm 1, -): D^d
\to W$ are homotopic to the $\theta$ structures on the disks
\begin{align*}
  D^d_\pm = \{x \in S^d | \pm x_1 \leq \tfrac{1}{2}\}
\end{align*}
with respect to some $\theta$ structure $l_0: TS^d \to \theta^*\gamma$
on $S^d$, chosen once and for all.  (Our assumption on $\theta:X \to
BO(d)$ is precisely that such a choice can be made.)  But then the
bundle map~\eqref{eq:37} extends over the cylinder $D^1 \times
S^{d-1}$, since that cylinder can be identified with
\[\{x \in S^d \mid |x_1| \leq \tfrac{1}{2}\}. \qedhere\]
\end{proof}

\subsubsection{Height functions}
\label{sec:level-pres-surg}

Pick once and for all a smooth family of functions 
\begin{align*}
  \lambda_w: [0,1] \to [w,1], \quad w \in [0,\tfrac{1}{2}]
\end{align*}
such that $\lambda(t) = t$ on $\lambda^{-1}([\tfrac{3}{4},1])$, that
$\lambda^{-1}(w) = [0,\tfrac{1}{2}]$, and such that $\lambda'(t) > 0$ on
$\lambda^{-1}(w,1)$.  Define a family of smooth functions $\varphi_w$
by
\begin{align*}
  D^1 \times S^{d-1} &\stackrel{\varphi_w}{\to} D^d, \quad w \in [0,\tfrac{1}{2}]\\
  (t,x) & \mapsto \lambda_w(|t|)x.
\end{align*}
The domain $D^1 \times S^{d-1}$ is the disjoint union of the open
sets $U^+ = (\tfrac{1}{2},1] \times S^{d-1}$ and $U^- = [-1,-\tfrac{1}{2}) \times
S^{d-1}$, and the closed set $C = [-\tfrac{1}{2}, \tfrac{1}{2}] \times S^{d-1}$.  Then the
map $\varphi_w$ restricts to diffeomorphisms
\begin{align}
  \label{eq:31}
  \begin{aligned}
    U^+ &\stackrel{\varphi_w|_{U^+}}{\longrightarrow} D^d - wD^d\\
    U^- &\stackrel{\varphi_w|_{U^-}}{\longrightarrow} D^d - wD^d
  \end{aligned}
\end{align}
and the restriction to $C = [-\tfrac{1}{2},\tfrac{1}{2}]\times S^{d-1}$ can be factored as
\begin{align}\label{eq:1}
  C \xrightarrow{\mathrm{proj}} S^{d-1} \xrightarrow{w} \partial (w
  D^d),
\end{align}
where the second map $w: S^{d-1} \to \partial (w D^d)$ denotes
multiplication by $w$ in $\R^n$.

Let $W$ be a smooth $d$-manifold with a smooth height function $x_1 : W \to \bR$. If $e: S^0 \times D^d \to W$ and $w \in [0,\tfrac{1}{2}]$ is such that $x_1\circ e(+1,x) =
x_1\circ e(-1,x)$ for $|x| \leq w$ we define a map $D^1 \times S^{d-1}
\to \R$ by
\begin{align}\label{eq:38}
  (t,x) & \mapsto
  \begin{cases}
    x_1\circ e(+1,\varphi_w(t,x)) & \text{if $t > -\tfrac{1}{2}$}\\
    x_1\circ e(-1,\varphi_w(t,x)) & \text{if $t < \tfrac{1}{2}$}.
  \end{cases}
\end{align}
This is well defined because of the factorisation~\eqref{eq:1} (so
$\varphi_w(t,x)$ is independent of $t$ for $t \in [-\tfrac{1}{2},\tfrac{1}{2}]$).

The inverses of the diffeomorphisms~(\ref{eq:31}) give an embedding
\begin{align}
  \label{eq:30}
  \begin{aligned}
    S^0 \times (D^d - wD^d) &\xrightarrow{h_w}  D^1 \times S^{d-1}\\
    (\pm 1, x) & \mapsto (\varphi_w|_{U^\pm})^{-1}(x)
  \end{aligned}
\end{align}
whose image is the complement of $C$.  The restriction of $h_w$ to $h:
S^0 \times (D^d - \tfrac{3}{4} D^d) \to D^1 \times S^{d-1}$ is independent of the
parameter $w$.

\begin{definition}\label{defn:surgery}
  Let $W \in \psi^\theta(n,1)$, let $w \in [0,\tfrac{1}{2}]$, and let $e: S^0
  \times D^d \to W$ be an embedding satisfying
  \begin{align*}
    x_1 \circ e(+1,x) = x_1 \circ e(-1,x),\quad \text{when $|x| \leq
      w$.}
  \end{align*}
  \begin{enumerate}[(i)]
  \item Let $\K^e_w(W)$ be the smooth
    manifold obtained by gluing $D^1 \times S^{d-1}$ into $W - e(S^0
    \times wD^d)$ along the embedding~\eqref{eq:30}.  Let
    \begin{align*}
      x_1: \K^e_w(W) \to \R
    \end{align*}
    be the function which agrees with~\eqref{eq:38} on $D^1 \times
    S^{d-1}$ and with the old $x_1$ on $W - e(S^0 \times w D^d)$.
  \item Pick functions $x_{2}, x_{3}, \dots: \K^e_w(W)\to (-1,1)$ extending the coordinate functions on $W - e(S^0 \times D^d)$
    such that the resulting map $x: \K^e_w(W) \to \R^\infty$ is an
    embedding.  Also pick, by
    Proposition~\ref{prop:extend-structures}, an extension of the
    $\theta$ structure on $W - e(S^0 \times w D^d)$ to a $\theta$
    structure on $\K^e_w(W)$ (after possibly changing sign of a
    coordinate on $D^d$).  We use the same notation $\K^e_w(W)$ for
    the resulting element
    \begin{align*}
      \K^e_w(W) \in \psi^\theta(\infty,1).
    \end{align*}
  \end{enumerate}
\end{definition}

In the second part of Definition~\ref{defn:surgery}, the
notation $\K^e_w(W) \in \psi^\theta(\infty,1)$ is slightly imprecise,
because the element $\K^e_w(W)$ depends on more data than just $W$,
$w$, and $e$.  The remaining data (namely the extension of the embedding and the
$\theta$-structure) will not play any role in the applications of the
construction, so we omit it from the notation.  The main properties of
$\K^e_w(W)$ are recorded in the following two lemmas.

\begin{lemma}\label{lemma:thin-surg}
  Let $W$, $w$ and $e$ be as in Definition~\ref{defn:surgery} and let
  $[a,b] = x_1\circ e(+1,wD^d) = x_1 \circ e(-1,wD^d)$.  Then the
  level sets $W \cap x_1^{-1}(t)$ and $\K^e_w(W) \cap x_1^{-1}(t)$ are
  canonically diffeomorphic, for all $t \in \R - [a,b]$.

  In particular in the extreme case $w=0$, level sets agree except at
  the level $a = x_1 \circ e(+1,0) = x_1 \circ e(-1,0)$.
\end{lemma}
\begin{proof}
  The manifold $\K^e_w(W)$ is the disjoint union of the open set $W -
  e(S^0 \times wD^d)$ and the closed set $C = [-\tfrac{1}{2},\tfrac{1}{2}] \times
  S^{d-1}$.  The lemma follows because $x_1: W\to \R$ takes the values
  in $[a,b]$ on $e(S^0 \times wD^d)\subseteq W$ and~\eqref{eq:30}
  takes values in $[a,b]$ on $C \subseteq \K^e_w(W)$.  Outside these
  closed sets, $W$ and $\K^e_w(W)$ are canonically diffeomorphic by a
  level-preserving diffeomorphism.
\end{proof}

The next lemma is concerned with what happens to the level sets that change, in
some cases.  More precisely we will consider $w > 0$ and embeddings
$e: S^0 \times D^d \to W$ that are height preserving, i.e.\ such that
\begin{align}\label{eq:2}
  x_1\circ e(\pm 1,y) = a + y_1, \quad\text{for $|y| \leq w$}
\end{align}
where $a = x_1 \circ e(+1,0) = x_1\circ e(-1,0)$ and $y_1$ is the
first coordinate of $y \in D^d \subseteq \R^d$.  In that case $e$
restricts to an embedding
\begin{align*}
  S^0 \times (w D^d_0) \to W_a = W \cap x_1^{-1}(a),
\end{align*}
where $D^d_0 = D^d \cap x_1^{-1}(0) \cong D^{d-1}$.
\begin{lemma}\label{lemma:thick-surg}
  Let $w > 0$ and $e: S^0 \times D^d \to W$ be as in~\eqref{eq:2}, and
  assume $a$ is a regular value of $x_1: W \to \R$.  Then the level
  set $\K^e_w(W) \cap x_1^{-1}(a)$ is diffeomorphic to the manifold
  obtained by performing 0-surgery on $W_a = W \cap x_1^{-1}(a)$ along
  the embedding
  \begin{align*}
    S^0 \times D^{d-1} \stackrel{\cong}{\longrightarrow} S^0 \times
    (wD^d_0) \stackrel{e}{\longrightarrow} W_a.
  \end{align*}
\end{lemma}
\begin{proof}
  $\K^e_w(W)$ is obtained from $W$ by removing $e(S^0 \times \Int (w
  D^d))$ and gluing in $C = [-\tfrac{1}{2},\tfrac{1}{2}]\times S^{d-1}$.  On the level set
  at $a$, this removes the open set $e(S^0 \times \Int(w D^d_0)) \cong
  S^0 \times D^{d-1}$ and glues in the set
  \begin{align*}
    C \cap (x_1)^{-1}(a) = [-\tfrac{1}{2},\tfrac{1}{2}] \times S^{d-2},
  \end{align*}
  where $S^{d-2} \subseteq S^{d-1}$ is the equator.
\end{proof}

Finally, two extended versions of the construction of $\K^e_w(W)$ from
$W,w,e$.  First a version with multiple surgeries at once.
\begin{definition}\label{defn:multiple-surg}
  Let $W \in \psi^\theta(\infty,1)$ and $w \in [0,\tfrac{1}{2}]$.  Let $\Lambda =
  \{\lambda_1, \dots, \lambda_r\}$ be a finite set and $e: \Lambda
  \times S^0 \times D^d \to W$ an embedding such that
  \begin{align}
    \label{eq:3}
    x_1 \circ e(\lambda_i,+1,x) = x_1 \circ e(\lambda_i,-1,x),\quad \text{when $|x| \leq
      w$.}
  \end{align}
  for all $\lambda_i \in \Lambda$.  Then define
  \begin{align*}
    \K^e_w(W) = \K^{e(\lambda_1,-)}_w \circ \dots \circ
    \K^{e(\lambda_r,-)}_w(W) \in \psi^\theta(\infty,1).
  \end{align*}
  This is well defined because of the disjointness of the embeddings
  $e(\lambda_i,-): S^0 \times D^d \to W$.
\end{definition}

Of course, the properties of levelwise surgery from Lemmas
\ref{lemma:thin-surg} and \ref{lemma:thick-surg} imply similar
properties for multiple levelwise surgery.  Secondly we need a
parametrised version of this construction.
\begin{proposition}\label{prop:parametrised-surg}
  Let $V$ be a contractible manifold, and $f: V \to
  \psi^\theta(n,1)$ and $w : V \to [0,\tfrac{1}{2}]$ smooth maps.  Let
  $\Lambda$ be a finite set and
$e : V \times \Lambda \times S^0\times D^d \to \Gamma(f|_V)$ be an embedding over $V$.
Assume that the triple $(f(x),w(x), e(x))$
  satisfies the assumption of Definition~\ref{defn:multiple-surg} for
  all $x \in V$.  Then there is a smooth map
  \begin{align*}
    V &\to \psi^\theta(\infty,1)\\
    x & \mapsto \K^{e(x)}_{w(x)}(f(x)),
  \end{align*}
  which for each $x$ is constructed as in
  Definition~\ref{defn:multiple-surg}.
\end{proposition}
\begin{proof}[Proof sketch]
  The manifold $f(x)$ depends smoothly on $x$.  If we pick the
  remaining coordinate functions $x_{2}, x_{3}, \dots:
  \K^{e(x)}_{w(x)}(f(x)) \to \R$ in a smooth fashion, the manifold
  $\K^{e(x)}_{w(x)}(f(x))$ will also depend smoothly on $x$.
  Under the assumption that $V$ is contractible, the obstructions to
  extending the $\theta$ structure are the same as in the case where
  $V$ is a point.
\end{proof}

Finally, let us remark that the surgery on elements of
$\psi_\theta(\infty,1)$ constructed in this section preserves the
subspace $\psi_\theta^{nc}(\infty,1)$.  It also preserves the
subspace $\psi_\theta^{nc}(\infty,1)^\bullet$, provided the
surgery data is disjoint from some strip $L_\epsilon \subseteq W$.

\subsection{Proof of Theorem \ref{thm:SubcategoryConnected}}\label{sec:ProofThmA}

We now prove the main result of this section and finish the proof of
Theorem \ref{thm:SubcategoryConnected}.
\begin{theorem}\label{thm:main-of-4.1}
  The relative homotopy groups
  \begin{align}\label{eq:35}
    \pi_k(\psi_\theta^{nc}(\infty,1)^\bullet,
    \psi_\theta^{nc}(\infty,1)^\bullet_\Conn)
  \end{align}
  vanish for all $k$.
\end{theorem}
To prove this theorem we need to construct a null homotopy of an
arbitrary continuous map of pairs
\begin{align*}
  (D^k, \partial D^k) \xrightarrow{f}
  (\psi_\theta^{nc}(\infty,1)^\bullet,
  \psi_\theta^{nc}(\infty,1)^\bullet_\Conn).
\end{align*}
The null homotopy will be constructed using the parametrised surgery
developed in the previous section.  We first give the local
construction.  It is convenient to have a word for the following
property.
\begin{definition}
Let us say that an element $W \in
  \psi_\theta^{nc}(n,1)^\bullet$ is \emph{connected at level
    $a$} if $a$ is a regular value of $x_1: W \to \R$ and if the level
  set $W_a = W \cap x_1^{-1}(a)$ is path connected.  Thus $W \in
  \psi_\theta^{nc}(n,1)^\bullet_\Conn$ if and only if it is
  connected at level $a$ for some $a \in (-1,1)$.
\end{definition}

In the following, we will very often use the following construction.
Let $j: \R \to \R$ be an embedding which is isotopic through embeddings to the identity,
has image $(-10,10)$, and has $j(t) = t$ for $t \in [-1,1]$.  Then let
$s$ be the continuous map
\begin{align}\label{eq:14}
  \begin{aligned}
    \psi_\theta(n,1) & \xrightarrow{s} \psi_\theta(n,1)\\
    M & \mapsto (j \times \mathrm{Id})^{-1}(M).
  \end{aligned}
\end{align}
This ``stretching'' map will be used throughout this section.

\begin{proposition}\label{prop:local-surg}
  Let $V$ be a contractible space and $f: V \to
  \psi_\theta^{nc}(\infty,1)^\bullet$ be a smooth map.  Let $a
  \in (-1,1)$ be a regular value of $x_1: f(x) \to \R$ for all $x \in
  V$.  Finally, let $\Lambda = \{\lambda_1, \dots, \lambda_r\}$ be a
  finite set, and
$$p : V \times \Lambda \times [0,1] \to \Gamma(f|_V)$$
be an embedding over $V$, such that
  \begin{enumerate}[(a)]
  \item\label{item:2} The path $p(x, \lambda,-): [0,1] \to f(x)$ ends
    somewhere in $f(x) \cap x_1^{-1}(a)$, outside the path component
    of the basepoint; each non-basepoint component contains exactly
    one of the ending points $p(x, \lambda,1)$.
  \item\label{item:4} The path $p(x, \lambda,-): [0,1] \to f(x)$
    starts somewhere in $f(x) - x_1^{-1}([-10,10])$.
  \end{enumerate}
  Then there is a homotopy $F: [0,1] \times V \to
  \psi_\theta^{nc}(\infty,1)^\bullet$, such that
  \begin{enumerate}[(i)]
  \item\label{item:5} $F(0,-)$ agrees with the composition
    \begin{align*}
      V \xrightarrow{f} \psi_\theta^{nc}(\infty,1)\bp \xrightarrow{s} \psi_\theta^{nc}(\infty,1)\bp.
    \end{align*}
  \item\label{item:6} $F(1,x)$ is connected at level $a$ for all $x \in V$.  In
    particular,
    \begin{align*}
      F(\{1\} \times V) \subseteq
      \psi_\theta^{nc}(\infty,1)^\bullet_\Conn.
    \end{align*}
  \item\label{item:9} If $f(x) \in
    \psi_\theta^{nc}(\infty,1)^\bullet_\Conn$ then $F(t,x) \in
    \psi_\theta^{nc}(\infty,1)^\bullet_\Conn$ for all $t$.
  \end{enumerate}
\end{proposition}
\begin{proof}
  First construct
$e : [0,1] \times V \times \Lambda \times S^0 \to \Gamma(f|_V)$
 an embedding over $V$, in the following way.  Recall that we have a trivialised
  $L_\epsilon = \R \times (-\epsilon,\epsilon)^{d-1} \times \{0\}
  \subseteq f(x)$ for all $x$.  Then pick points $x_\lambda \in
  (0,\epsilon)^{d-1} \times \{0\}$ arbitrarily, but with all
  $x_\lambda$ different (this is possible because $d-1 > 0$).  Then
  set
  \begin{align*}
    e(t, x,\lambda,+1,) &= p(x, \lambda,t)\\
    e(t, x,\lambda,-1) &= (x_1\circ p(x, \lambda,t),x_\lambda).
  \end{align*}
  This has the property that $e(t, x, -)$ embeds $\Lambda \times S^0$ into a single level set of $x_1: f(x) \to \bR$.

Thicken each embedding $e(t, x,-): \Lambda \times S^0 \to f(x)$ to
  an embedding of $\Lambda \times S^0 \times D^d$.  We denote the
  resulting map by the same letter
$$e : [0,1] \times V \times \Lambda \times S^0 \times D^d \to \Gamma(f|_V).$$
We can arrange that at $t=1$ we have the analogue of \eqref{eq:2}, namely
    \begin{align*}
      x_1 \circ e(1, x,\lambda,\pm 1, v) = a + v_1, \quad \text{for
        $|v| \leq w_0$}
    \end{align*}
    for some $w_0 > 0$.
  
Let $\rho: [0,1] \to [0,1]$ be a smooth function with $[0,\tfrac{1}{2}]
  \subseteq \rho^{-1}(0)$ and $\rho(1) = 1$.  Set $w(t) = \rho(t)w_0$.
  Also pick a function $\tau: [0,1] \to [0,1]$ with $\tau(0) = 0$ and
  $[\tfrac{1}{2},1] \subseteq \tau^{-1}(1)$.  By
  Proposition~\ref{prop:parametrised-surg} we get a smooth
  homotopy
  \begin{align}\label{eq:5}
    \begin{aligned}
      {}[0,1] \times V &\to \psi_\theta^{nc}(\infty,1)^\bullet\\
      (t,x) & \mapsto \K^{e(\tau(t), x)}_{w(t)}(f(x)).
    \end{aligned}
  \end{align}
  By Lemma~\ref{lemma:thick-surg} this homotopy at time 1 performs
  0-surgery on the level set $x_1^{-1}(a)$.  This 0-surgery forms the
  connected sum of the basepoint component with all the other
  components, so the result satisfies~(\ref{item:6}).  At time 0, we
  have performed 0-surgery on points outside of $x_1^{-1}([-10,10])$,
  so we have not changed anything inside that interval.  Therefore we
  can let $F: [0,1]\times V \to
  \psi_\theta^{nc}(\infty,1)^\bullet$ be the composition of
  \eqref{eq:5} with the stretching map $s$ from~(\ref{eq:14}), and
  then $F$ satisfies (\ref{item:5}) and (\ref{item:6}).

Finally, (\ref{item:9}) follows for any pair $(t,x)$ by cases. Either $\rho(t)=0$, so Lemma \ref{lemma:thin-surg} implies there is at most one level set of $F(t,x)$ in $(-1,1)$ different from its corresponding one in $f(x)$, but if $f(x) \in \psi_\theta^{nc}(\infty,1)^\bullet_{\Conn}$ then there is a small open interval of regular values in $(-1,1)$ having connected level sets, so $F(t,x)$ must still have a connected level set. Otherwise $\tau(t)=1$, so $F(t,x)$ is obtained from $f(x)$ by performing 0-surgery on the level set $x_1^{-1}(a) \cap f(x)$ to make it connected, so $F(t,x)$ has a connected level set.
\end{proof}

Our strategy for proving vanishing of the relative homotopy
groups~(\ref{eq:35}) is to apply Proposition~\ref{prop:local-surg}
locally to construct a null homotopy.  More precisely we will use the
following corollary of Proposition~\ref{prop:local-surg} (and its
proof).
\begin{corollary}\label{corollary:homotopy-2}
  Let $f: X \to \psi_\theta^{nc}(\infty,1)^\bullet$ be
  smooth.  Assume that there exists an open cover $X = V_1 \cup
  \dots \cup V_r$ by contractible open sets $V_i \subseteq X$ and
  different $a_i \in (-1,1)$ such that $a_i$ is a regular value of
  $x_1:f(x) \to \R$ for all $x \in V_i$.  Finally, assume there are
  finite sets $\Lambda_i$ and an embedding
$$p_i : V_i \times \Lambda_i \times [0,1] \to \Gamma(f|_{V_i})$$
over $V_i$, such that
  \begin{enumerate}[(a)]
  \item\label{item:10} $(V_i, f|_{V_i}, a_i, \Lambda_i, p_i)$ satisfy
    the assumptions of Proposition~\ref{prop:local-surg}.
  \item\label{item:11} For $i \neq j$, the images of $p_i(x, -)$ and
    $p_j(x, -)$, if both defined, are disjoint.
  \end{enumerate}
  Then given any collection $\{U_i \subseteq V_i\}$ with $\overline{U}_i \subseteq
  V_i$, there is a homotopy $H: [0,1] \times X \to
  \psi_\theta^{nc}(\infty,1)^\bullet$ such that
  \begin{enumerate}[(i)]
  \item $H(0,-)$ agrees with the composition
    \begin{align*}
      X \xrightarrow{f} \psi_\theta^{nc}(\infty,1)\bp \xrightarrow{s} \psi_\theta^{nc}(\infty,1)\bp.
    \end{align*}
  \item $H(1,x)$ is connected at level $a_i$ for all $x \in U_i$.  In
    particular if the $U_i$ cover $X$, then
    \begin{align*}
      H(\{1\} \times X) \subseteq
      \psi_\theta^{nc}(\infty,1)^\bullet_\Conn.
    \end{align*}
  \item If $f(x) \in \psi_\theta^{nc}(\infty,1)^\bullet_\Conn$,
    then $H(t,x) \in \psi_\theta^{nc}(\infty,1)^\bullet_\Conn$ for
    all $t$.
  \end{enumerate}
\end{corollary}
\begin{proof}
  For $i \in \{1,\dots, r\}$, let $[0,1]^{\{i\}} \subseteq [0,1]^r$ be
  the subspace where all coordinates but the $i$th are 0.  From the
  paths $p_i$, the proof of Proposition~\ref{prop:local-surg}
  constructs homotopies $F_i: [0,1]^{\{i\}} \times V_i \to
  \psi_\theta(\infty,1)$ in two steps.  First we used the path $p_i$ to
  construct a map
  \begin{align*}
    e_i: [0,1]^{\{i\}} \times V_i \times \Lambda_i\times S^0 \times
    D^d  \to \Gamma(f_{V_i}),
  \end{align*}
over $V_i$ that is an embedding for each point of $[0,1]^{\{i\}} \times V_i$, then we let $F_i$ be the composition
  \begin{align}\label{eq:9}
    \begin{aligned}
    \xymatrix @R0mm {
      {[0,1]^{\{i\}} \times V_i} \ar[r] & \psi_\theta(\infty,1)\\
      (t,x) \ar@{|->}[r] & \K^{e_i(\tau(t_i), x)}_{w(t_i)}(f(x)).&
    }
    \end{aligned}
  \end{align}
  The homotopies $F_i$ all start at $s \circ f$ so we can glue them
  together to a map
  \begin{align}\label{eq:7}
    \bigcup_i \big([0,1]^{\{i\}} \times V_i\big) \stackrel{F}{\longrightarrow}
    \psi_\theta(\infty,1),
  \end{align}
  where the union is inside $[0,1]^r \times X$.

  Recall that the embedding $e_i(t,x): \Lambda_i \times S^0 \times D^d
  \to f(x)$ was constructed the following way.  On $\{\lambda\} \times
  \{+1\} \times D^d$ it is a disk centered at $p_i(x,\lambda,t)$, and on
  $\{\lambda\} \times \{-1\} \times D^d$ it is constructed in the
  standard strip $L_\epsilon \subseteq f(x)$.  By the disjointness
  assumption (\ref{item:11}) we may suppose that the image of
  $e_i(t,x)$ is disjoint from the image of $e_j(t',x)$ for $x \in V_i
  \cap V_j$ and $t,t' \in [0,1]$, and thus we get an embedding of
  $(\Lambda_i \amalg \Lambda_j) \times S^0 \times D^d$.  Disjointness
  of the $e_i$'s means that we can iterate the
  construction~\eqref{eq:9}.  Namely if $T = \{i_1, \dots, i_k\}$ and
  we set $V_T = V_{i_1} \cap \dots \cap V_{i_k}$, we have the homotopy
  \begin{align}\label{eq:6}
    \begin{aligned}
      {}[0,1]^T \times V_T &\to \psi_\theta(\infty,1)\\
      (t,x) & \mapsto \K^{e_{i_1}(\tau(t_{i_1}), x)}_{w(t_{i_1})} \circ
      \dots \circ \K^{e_{i_k}(\tau(t_{i_k}),x)}_{w(t_{i_k})}(f(x)),
    \end{aligned}
  \end{align}
  where the iterated surgery is defined because the corresponding
  surgery data are disjoint.  Composing~\eqref{eq:6} with the
  stretching map $s$ gives a homotopy $F_T: [0,1]^T \times V_T \to
  \psi_\theta(\infty,1)$.  If we regard $[0,1]^T$ as a subset of $[0,1]^r$
  in the obvious way, the various $F_T$'s are compatible and we can
  glue them to a smooth map extending~\eqref{eq:7}
  \begin{align}
    \label{eq:8}
    \bigcup_{T} \big([0,1]^T \times V_T\big) \xrightarrow{F}
    \psi_\theta(\infty,1).
  \end{align}
  The restriction of $F$ to $\{0\} \times X$ is $s \circ f$, and if
  $F(t,x)$ is defined and $t_i = 1$, then $F(t,x)$ is connected at
  level $a_i$.

  Finally we pick for each $i$ a bump function $\rho_i : X \to [0,1]$
  supported in $V_i$, such that $U_i \subseteq \rho_i^{-1}(1)$.  Let
  $\rho: X \to [0,1]^r$ have the $\rho_i$ as coordinate functions.
  Then the map
  \begin{align*}
    [0,1] \times X &\to [0,1]^r \times X\\
    (t,x) & \mapsto (t \rho(x), x)
  \end{align*}
  has image in the domain of~\eqref{eq:8}, and we let $H$ be the
  composition of the two maps.
\end{proof}

We now return to the proof of Theorem~\ref{thm:main-of-4.1}.  Given a
map of pairs
\begin{align}\label{eq:MapToNullify}
  (D^k, \partial D^k) \xrightarrow{f}
  (\psi_\theta^{nc}(\infty,1)^\bullet,\psi_\theta^{nc}(\infty,1)^\bullet_\Conn),
\end{align}
we must produce a relative homotopy, deforming $f$ to a map into
$\psi_\theta^{nc}(\infty,1)^\bullet_\Conn$.  The homotopy will be
produced by Corollary~\ref{corollary:homotopy-2} once we construct the
data $(V_i,a_i, p_i)$.  We first consider the case $k=0$, in which the
map $p: D^0 \times \Lambda \times [0,1] \to W$
is produced by the following lemma.  It is proved by first producing paths $p(\lambda,-)$
which are embedded but not necessarily disjoint, and then making them
disjoint in Lemma~\ref{lemma:vector-flow}.
\begin{lemma}\label{lemma:existence-paths-point}
  For any $W \in \psi_\theta^{nc}(n,1)^\bullet$ and any $a
  \in (-1,1)$ which is a regular value of $x_1: W \to \R$, there is a
  finite set $\Lambda = \{\lambda_1, \dots, \lambda_r\}$ and an
  embedding $p: \Lambda \times [0,1] \to W$ such that
  \begin{enumerate}[(a)]
  \item\label{item:7} The path $p(\lambda,-): [0,1] \to W$ starts
    somewhere in $W \cap x_1^{-1}(a)$, outside the path component of
    the basepoint, and each non-basepoint component contains exactly
    one of the starting points $p(\lambda,0)$.
  \item\label{item:8} The path $p(\lambda,-): [0,1] \to W$ ends
    somewhere in $f(x) - x_1^{-1}([-10,10])$.
  \end{enumerate}
\end{lemma}
\begin{proof}
  First pick points $p(\lambda,0) \in W \cap x_1^{-1}(a)$, one in
  each non-basepoint component.  The index set $\Lambda$ is then the
  set of non-basepoint components.  By assumption, no compact
  component of $W$ is contained in $x_1^{-1}([-10,10])$, so for each
  $\lambda$ there is a path in $W$ from $p(\lambda,0)$ to a point in
  $W - x_1^{-1}([-10,10])$.  Pick an embedded such path $p(\lambda,-):
  [0,1] \to W$.  These assemble to a map
  \begin{align*}
    p: \Lambda \times [0,1] \to W,
  \end{align*}
  but it need not be an embedding because the images of the paths
  $p(\lambda,-)$ need not be disjoint. However, by transversality we may suppose that the paths $p(\lambda,-)$ do not intersect at their endpoints. In
  Lemma~\ref{lemma:vector-flow} below we will prove that we can change
  each $p(\lambda,-): [0,1] \to W$ by an isotopy of embedded paths
  such that the isotoped paths are disjoint.
\end{proof}
\begin{lemma}\label{lemma:vector-flow}
  Let $W$ be a smooth manifold and $p_1, \dots, p_r: [0,1] \to W$ a
  set of embedded paths, such that
  \begin{align*}
    p_i([0,1]) \cap p_j(\partial [0,1]) = \emptyset\quad \text{for $i
      \neq j$}.
  \end{align*}
  Then there are paths $q_i: [0,1] \to W$ such that $p_i$ is isotopic
  to $q_i$ through an isotopy of embedded paths which is relative to a
  neighbourhood of the endpoints, and such that the
  \begin{align*}
    q_i([0,1]) \cap q_j([0,1]) = \emptyset\quad \text{for $i \neq j$}.
  \end{align*}
\end{lemma}
\begin{proof}
By induction on $r$ we can assume that $q_i = p_i$ already have
  disjoint images for $i < r$.  The derivative $q_i'(t) \in TW$ gives
  a vector field on $q_i([0,1]) \subseteq W$.  We can extend this to a
  vector field $X$ on $W$, with the properties that $X(q_i(t)) =
  q_i'(t)$ for $i = 1,\dots, r-1$.  Then each $q_i$ is the restriction
  of an integral curve $\overline{q}_i$ of $X$.  After multiplying $X$
  with a function $W \to [0,1]$ which is 1 on a neighbourhood of the
  image of the $q_i$'s and vanishes outside a slightly larger
  neighbourhood, we may assume that $X$ has compact support and that
  $\overline{q}_i$ does not intersect $q_r$ outside the image of
  $q_i$.

Finally extend $X$ to a compactly supported vector field on all of $\R^n$, i.e.\ $X: \R^n \to \R^n$.  Let $F: \R \times \R^n \to \R^n$ be the flow of $X$. It is defined everywhere because $X$ is compactly supported, it preserves $W$, and it also fixes a neighbourhood of the end points of $p_r$. Then $q_r(s) = F(1,p_r(s))$ will work, since all intersections between $q_r$ and the $\overline{q}_i$'s have been flowed off of $q_i$.
\end{proof}

Combining Lemma~\ref{lemma:existence-paths-point} and Corollary
\ref{corollary:homotopy-2} produces a null homotopy of~(\ref{eq:MapToNullify}),
proving Theorem~\ref{thm:main-of-4.1} for $k=0$.  For $k > 0$ we need
a parametrised version of Lemma~\ref{lemma:existence-paths-point}.
\begin{proposition}\label{prop:paths-exist}
  For any smooth $f: D^k \to \psi_\theta^{nc}(\infty,1)^\bullet$,
  there exists a covering $D^k = V_1 \cup \dots \cup V_k$ by
  contractible open sets $V_i$, real numbers $a_i$, finite sets
  $\Lambda_i$, and embeddings
$p_i : V_i \times \Lambda_i\times [0,1] \to \Gamma(f|_{V_i})$ over $V_i$
satisfying the assumptions of
  Corollary~\ref{corollary:homotopy-2}.
\end{proposition}
\begin{proof}
  First pick regular values $a_x \in (-1,1)$ of $x_1: f(x) \to \R$, one
  for each $x \in D^k$.  Apply Lemma~\ref{lemma:existence-paths-point}
  for $W = f(x)$ and $a = a_x$ to find a finite set $\Lambda_x$ and an
  embedding
$$p_x : \{x\} \times \Lambda_x \times [0,1] \to f(x)$$
  satisfying (\ref{item:7}) and (\ref{item:8}) of
  Lemma~\ref{lemma:existence-paths-point}.

  Then extend to a map
$p_x : V_x \times \Lambda_x \times [0,1] \to \Gamma(f|_{V_x})$
defined on a neighbourhood $V_x$ of $x$. After possibly
  shrinking the $V_x$, each tuple $(V_x, f|_{V_x} ,\Lambda_x, a_x,
  p_x)$ will satisfy the assumptions in
  Proposition~\ref{prop:local-surg}.  Using compactness of $D^k$ we
  get a finite subcover
  \begin{align*}
    D^k = V_1 \cup \dots \cup V_r,
  \end{align*}
  with corresponding regular values $a_i \in (-1,1)$ and maps
$$p_i : V_i \times \Lambda_i \times [0,1] \to \Gamma(f|_{V_i})$$
satisfying condition (\ref{item:10}) of
  Corollary~\ref{corollary:homotopy-2} but not the disjointness
  condition (\ref{item:11}).

\vspace{2ex}

  Before achieving full disjointness of all paths, let us explain how
  to modify the data $(V_i, a_i, p_i)$ so as to satisfy
  \begin{align}\label{eq:11}
    p_j(\{x\} \times \Lambda_j \times \partial[0,1]) \text{ disjoint from
      $\IM(p_i(x, -))$ for $j < i$.}
  \end{align}
  To achieve this we proceed by induction on $i$, the case $i=1$ being
  vacuous. We can temporarily extend $p_i$ to a map $e_i : V_i \times \Lambda_i \times [0,1] \times D^{d-1} \to \Gamma(f|_{V_i})$ over $V_i$ that is an embedding of $\{x\} \times \Lambda_i \times \{t\} \times D^{d-1}$ into a level set for each $(x, t) \in V_i \times [0,1]$. For each $x \in V_i$ there is a $d_x \in D^{d-1}$ such that $\IM(e_i(x, -, d_x))$ is disjoint from $\bigcup_{j<i}p_j(\{x\} \times \Lambda_j \times \partial[0,1])$. Furthermore, there is a small contractible neighbourhood $x \in V_i^x \subseteq V_i$ where this is still true. Equip $V_i^x$ with the paths
$$p_i^x = e_i(-, d_x) : V_i^x \times \Lambda_i \times [0,1] \to \Gamma(f|_{V_i^x}).$$
We can choose finitely many $\{x_j\}_{j \in J}$ such that $\cup_{j \in J} V_i^{x_j} = V_i$ and the $d_{x_j}$ are distinct points in $D^{d-1}$. Then note that the $\{p_i^{x_j}\}_{j \in J}$ have disjoint images in $\Gamma(f|_{V_i})$.

The $\{p_i^{x_j}\}_{j \in J}$ all start at height $a_i \in (-1,1)$. We wish for the paths to start at different heights: by replacing the paths $p_i^{x_j}$ with their restrictions to a subinterval $[\eta, 1]$, for some small $\eta>0$, we can ensure that they all start at different heights $a_i^{x_j}$. We now replace $(V_i, a_i, p_i)$ by the collection $\{(V_i^{x_j}, a_i^{x_j}, p_i^{x_j})\}_{j \in J}$. This proves the induction step, and hence~(\ref{eq:11}).

\vspace{2ex}

  Finally we pass from the ``endpoint disjointness'' expressed
  in~(\ref{eq:11}) to the actual disjointness of all paths as required
  in (\ref{item:11}) of Corollary~\ref{corollary:homotopy-2}.  This is
  just an easy parametrised version of Lemma~\ref{lemma:vector-flow}.
  Indeed, if we pick an ordering of all the paths involved (i.e.\ of
  the disjoint union of the finite sets $\Lambda_i$) and proceed by
  induction as in the proof of that lemma.  Suppose we have already
  made the first $l$ paths disjoint.  Then as in the proof of
  Lemma~\ref{lemma:vector-flow} we get in the induction step for each
  $x \in X$ a compactly supported vector field $X(x): \R^n \to \R^n$
  preserving $f(x) \subseteq \R^n$.  It is easy to construct this to
  depend smoothly on $x$, and then we can flow the $(l+1)$st path
  with the flow of $X$ to make it disjoint from the previous paths,
  without changing any endpoints.
\end{proof}

\begin{proof}[Proof of Theorem~\ref{thm:main-of-4.1}]
  Starting with the map of pairs~(\ref{eq:MapToNullify}),
  Proposition~\ref{prop:paths-exist} produces a collection $\{(V_i, a_i, p_i)\}$
  satisfying (\ref{item:10}) and (\ref{item:11}) of
  Corollary~\ref{corollary:homotopy-2}, which then produces a relative
  null homotopy.
\end{proof}

\begin{proof}[Proof of Theorem \ref{thm:SubcategoryConnected}]
Lemma \ref{lem:AddingBasepoint}, Lemma \ref{lem:ElimCpctComponents} and Theorem \ref{thm:main-of-4.1} imply that the map
$$\psi^{nc}_\theta(\infty,1)^\bullet_\mathbf{Conn} \to \psi_\theta(\infty,1)$$
is a weak homotopy equivalence. Theorem \ref{thm:HtpyTypeCobCat} and Theorem \ref{thm:MainThmGMTW} identify these spaces with $B \mathcal{C}_\theta\bp$ and $B \mathcal{C}_\theta$ respectively.
\end{proof}

\subsection{Reversing morphisms}\label{sec:ReversingMorphisms}

Let us study the implications of the assumptions on $\theta: \X \to
BO(2)$.  We can arrange that $\X = EO(2) \times_{O(2)} F$, for some
space $F$ equipped with an action
\begin{align}
  \label{eq:25}
  O(2) \times F \to F.
\end{align}
The assumption that $\X$ is path connected says that $\pi_0 O(2)$ acts
transitively on $\pi_0 F$.  Therefore $F$ is either path connected, or
has two (homeomorphic) path components which are permuted by the
action of some reflection in $O(2)$.  To interpret the assumption that
$S^2$ admits a $\theta$-structure, we consider the induced action of unbased
homotopy classes
\begin{align}
  \label{eq:33}
  [S^1,O(2)] \times [S^1, F] \to [S^1,F].
\end{align}
We regard $\pi_0 F \subseteq [S^1,F]$ as the homotopy classes of
constant maps.
\begin{proposition}
  If $S^2$ admits a $\theta$-structure, then the subset $\pi_0(F) \subseteq
  [S^1, F]$ is preserved by elements of $[S^1, O(2)]$ represented by
  maps of even degree.
\end{proposition}
That a map $S^1 \to O(2)$ has even degree means that the corresponding
vector bundle over $S^2$ built by clutching has vanishing second Stiefel--Whitney class.
\begin{proof}
  The free homotopy classes of maps of even degree form a subgroup
  generated by two elements: the element $r$, represented by a
  constant map $S^1 \to O(2)$ to some reflection, and the element $2:
  S^1 \to SO(2) \subseteq O(2)$ represented by a map of degree 2.  It
  suffices to see that these preserve $\pi_0(F) \subseteq [S^1, F]$.
  This is obvious for $r$ (as it is represented by a constant map), so
  we consider 2.  Pick a basepoint of $F$ and consider the long exact
  sequence in homotopy
  \begin{align*}
    \pi_2(F) \to \pi_2 X \to \pi_2 BO(2) \to \pi_1 F \to \pi_1 X \to
    \dots.
  \end{align*}
  The assumption says that a classifying map for $T S^2$ gives an
  element of $\pi_2 BO(2) = \pi_1 O(2)$ which lifts to $\pi_2 X$.
  Equivalently it must vanish in $\pi_1 F$.  But $T S^2$ comes from
  a map $S^1 \to SO(2) \subseteq O(2)$ which has degree 2, hence this
  assumption says that $2 \in [S^1, O(2)]$ acts trivially on $[S^1,F]$.
\end{proof}

We next study $\theta$-structures on $D^2$.  More precisely, if we are
given a structure over $\partial D^2$, how do we know whether it
extends to all of $D^2$?  Such a structure is called a bounding
structure, and we prove that the set of bounding structures is
preserved by a certain construction that we now define.  Let
\begin{align*}
  V \in \Gamma(TD^2|_{\partial D^2})
\end{align*}
be a non-zero vector field, and let
\begin{align*}
  \phi^V: TD^2|_{\partial D^2} \to TD^2|_{\partial D^2}
\end{align*}
be the bundle map which maps $V \mapsto -V$, and acts as the identity
on $V^\perp$.
\begin{corollary}
  Let
  \begin{align*}
    l: TD^2|_{\partial D^2} \to \theta^* \gamma
  \end{align*}
  be a bounding $\theta$-structure (i.e.\ one that extends to $D^2$).  Then $l
  \circ \phi^V$ is also bounding.
\end{corollary}
\begin{proof}
  Give $D^2$ the framing coming from $\R^2$.  With respect to this
  framing, $l: \partial D^2 \to F$ has homotopy class
  \begin{align*}
    l \in [S^1, F]
  \end{align*}
  which is in $\pi_0 F \subseteq [S^1, F]$ because $l$ bounds.  With
  respect to the framing, the bundle map $\phi^V$ corresponds to the
  map $S^1 \to O(2)$ which takes $x \in S^1$ to the reflection taking
  $V(x) \mapsto -V(x)$ and preserves $V(x)^\perp$.  This map depends
  not on the vector $V(x)$, but only on the line $\R V(x)$, so it
  factors through the projection $S^1 \to \RP^1$.  It follows that
  $\phi^V$ corresponds to an element
  \begin{align*}
    \phi^V \in [S^1, O(2)]
  \end{align*} 
  which is of even degree.  Consequently $l \circ \phi^V \in [S^1, F]$
  also bounds, by the previous proposition.
\end{proof}

We now study the extent to which morphisms in $C_\theta$ can be
``turned around''.  The following proposition depends crucially on the
assumption that $S^2$ admit a $\theta$-structure.  It is false for example for
framed manifolds.
\begin{proposition}\label{prop:ReverseMorphisms}
  Let $c_0 = (M_0, l_0) \in C_\theta$ and $c_1 = (M_1,l_1)$ be objects
  and let $(W,l): c_0 \to c_1$ be a morphisms in $C_\theta$ whose
  underlying manifold is connected.  Then there exists a morphism
  $(\overline{W},\overline{l}) : c_1 \to c_0$ whose underlying
  manifold $\overline{W}$ is diffeomorphic to $W$.

In particular, if $c_0$ and $c_1$ are linked by a zig-zag of morphisms, then there are in fact morphisms between them in both directions.
\end{proposition}
\begin{proof}
  Recall that $W \subseteq [0,t] \times \R^n$ for some large $n$, and
  that $W$ is ``cylindrical'' near its boundary.  Let $\varphi: [0,t] \to
  [0,t]$ be the affine map $s \mapsto t-s$, and let $F = \varphi \times
  \mathrm{Id}: [0,t] \times \R^n \times [0,t] \times \R^n$.

  To construct the morphism $c_1 \to c_0$ we first construct the
  underlying manifold as $\overline{W} = F^{-1}(W)$.  If there were no
  tangential structures to worry about, this would be a morphism from
  $c_1$ to $c_0$, but we need to extend the structure given near $\partial
  \overline{W} = c_1 \amalg c_0$ to all of $\overline{W}$.

  The vector $V = \partial/\partial x_1 \in \R^{n+1}$ gives a section
  of $TW$, defined near $\partial W$ (pointing inwards on the incoming
  boundary and outwards on the outgoing).  Since $W$ is connected, we
  can pick a relative CW structure with only one 2-dimensional cell.
  Let $W^1\subseteq W$ be the 1-skeleton in this CW structure.  We can
  extend $V$ to a non-zero section
  \begin{align*}
    V \in \Gamma(TW|_{W^1})
  \end{align*}
  Now let $\phi^V: TW|_{W^1} \to TW|_{W^1}$ be the automorphism constructed
  as before: it takes $V \mapsto -V$ and is the identity on $V^\perp$.
  The ``reflected'' tangential structure
  \begin{align*}
    \overline{l} = l \circ \phi^V: TW|_{W^1} \to \theta^*\gamma,
  \end{align*}
  extends first to a neighbourhood of $W^1 \subseteq W$, and then over
  the 2-cell $D^2 \to W$ by the corollary above.

  Thus we have a tangential structure $\overline{l} : TW \to
  \theta^*\gamma$.  If we use the diffeomorphism $F: W \to
  \overline{W}$ to transport it to a structure on $\overline{W}$, we
  have produced the required morphisms $c_1 \to c_0$.
\end{proof}

\subsection{Proof of Theorem \ref{thm:SubcategoryMonoids}}\label{sec:DimensionTwo}
By Theorem \ref{thm:HtpyTypeCobCat}, we have reduced the first part of Theorem \ref{thm:SubcategoryMonoids} to the following theorem.
\begin{theorem}\label{thm:CToConn}
For tangential structures $\theta : \X \to BO(2)$ such that $\X$ is path connected and $S^2$ admits a $\theta$-structure, the inclusion
$$\psi^{nc}_\theta(\infty,1)^\bullet_\mathbf{C} \to \psi^{nc}_\theta(\infty,1)^\bullet_\mathbf{Conn}$$
is a weak homotopy equivalence of each component onto a component of $\psi^{nc}_\theta(\infty,1)^\bullet_\mathbf{Conn}$.
\end{theorem}
The second part of Theorem \ref{thm:SubcategoryMonoids} will be proved in Proposition \ref{prop:HomotopyCommutative}. Let us first prove Theorem \ref{thm:CToConn} in the case where $\mathbf{C}$ is a disjoint union of path components of $\ob(C_\theta^\bullet)$. We will explain how to remove this assumption in Proposition \ref{prop:PathComponentsOfObjects}.

Suppose we are given a map of pairs
$$f:(D^k, \partial D^k) \lra (\psi^{nc}_{\theta}(\infty, 1)\bp_{\mathbf{Conn}}, \psi^{nc}_{\theta}(\infty, 1)\bp_{\mathbf{C}}).$$
For each point $x \in D^k$, we can choose a regular value $a_x \in (-1,1)$ of $x_1 :f(x) \to \bR$. There is an open neighbourhood $U_x$ of $x$ on which $a_x$ is still regular. We may choose a finite subcover $\{U_i\}_{i \in I}$ of $\{U_x\}_{x \in X}$, and suppose that the $a_i$ are distinct: otherwise we may perturb them so that they are. Choose an $\epsilon >0$ such that the intervals $(a_i-2\epsilon, a_i+2\epsilon)$ are disjoint. By Lemma \ref{Cylindrical}, we may further suppose that for $x \in U_i$ the manifold $f(x)$ is cylindrical in $x_1^{-1}(a_i-\epsilon, a_i+\epsilon)$.

\vspace{2ex}

As we are working in the space of manifolds with basepoint, and the family $D^k$ is compact, there is a $\delta > 0$ such that the strip $L_{3\delta} = \bR \times (-3\delta, 3\delta) \times \{0\}$ lies inside every manifold $f(x)$, and furthermore that $f(x) \cap \bR \times (-3\delta, 3\delta)^{\infty} = L_{3\delta}$. Recall that inside $L_{3\delta}$ the surface has the standard $\theta$-structure. We can change the manifold inside $\bR \times (\delta, 2\delta) \times (-3\delta, 3\delta)^{\infty-1}$ without leaving the space of manifolds with basepoint, as the new manifold will still contain $L_{\delta}$.

The set $\pi_0 \ob( \mathcal{C}_\theta\bp )$ is in natural bijection with $\pi_1 \Fib(\theta)$, as an object $C$ has a natural framing of $\epsilon^1 \oplus TC$, extending the given framing on $(-\epsilon,\epsilon) \times \{0\}$, and this framing is unique up to homotopy. Write $*$ for the object having the standard $\theta$-structure with respect to this framing. For each $U_i$, choosing a $u_i \in U_i$ and setting $M_i = f(u_i) \cap \{a_i\}\times \bR^\infty$ produces a homotopy class $[M_i] \in \pi_1 \Fib(\theta)$ which is independent of the choice of $u_i$.

\begin{lemma}\label{lem:InverseObjects}
Let $\mathbf{C} \subseteq \ob( \mathcal{C}_\theta\bp )$ be a disjoint union of path components: this determines a subset $\pi_0 \mathbf{C} \subseteq \pi_1 \Fib(\theta)$. If an object $M \in \mathcal{C}_\theta\bp$ lies in a component of $B \mathcal{C}_\theta\bp$ that also contains an object in $\mathbf{C}$, there is another object $N$ such that $[M] \cdot [N] \in \pi_0\mathbf{C} \subseteq \pi_1 \Fib(\theta)$, where $\cdot$ denotes the group product in $\pi_1 \Fib(\theta)$. Furthermore, there are morphisms
$$* \lra N \lra *$$
in $\mathcal{C}_\theta\bp$. If $M$ is already in $\mathbf{C}$, we can choose $N$ to be $*$.
\end{lemma}
\begin{proof}
Choose morphisms $E_1 : M \to C \in \mathbf{C}$, and $E_2 : C \to M$ which exist by Proposition \ref{prop:ReverseMorphisms}, and an object $M^{-1}$ such that $[M^{-1}] = [M]^{-1} \in \pi_1 \Fib(\theta)$. Scale the morphisms $E_1$ and $E_2$ to have length 1. Consider the morphism $I_{M^{-1}} = (1, \bR \times M^{-1}) : M^{-1} \to M^{-1}$. We may cut this along $\bR \times \{0\}$ to obtain a surface with $\theta$-structure and height function, $\tilde{I}_{M^{-1}}$, having two identical boundary components, which we will call $\widetilde{M^{-1}}$ and which are intervals with a $\theta_1$-structure that is standard near their ends.

Let $\epsilon>0$ be small enough that the morphisms $E_1$ and $E_2$  both satisfy $E_i \cap \bR \times (-3\epsilon, 3\epsilon)^\infty = L_{3\epsilon}$. In particular, the morphisms come from $\mathcal{C}_\theta^{3\epsilon}$. There is a height-preserving embedding $e: \tilde{I}_{M^{-1}} \to \bR \times [\epsilon, 2\epsilon] \times [0, 3\epsilon)^\infty$ sending $\partial \tilde{I}_{M^{-1}}$ to $\bR \times \{\epsilon, 2\epsilon\} \times \{0\}$ with image that agrees with $L_{3\epsilon}$ as a $\theta$-manifold near $(\bR \times \bR - (\epsilon, 2\epsilon) ) \times \bR^\infty$.

Let $\tilde{E}_i = (E_i - \bR \times (\epsilon, 2\epsilon)\times\{0\}) \cup e(\tilde{I}_{M^{-1}}) \in \psi_\theta^{nc}(\infty, 1)\bp_\Conn$. Then $(1, \tilde{E}_1)$ and $(1, \tilde{E}_2)$ are morphisms in $\mathcal{C}_\theta\bp$, between 
$(M - (\epsilon, 2\epsilon)\times\{0\}) \cup e(\widetilde{M^{-1}})$ and $(C - (\epsilon, 2\epsilon)\times\{0\}) \cup e(\widetilde{M^{-1}})$. The first of these objects is in the path component of $*$, and the second is defined to be $N$.
\end{proof}

We will now construct a $\theta$-manifold $D$ with height function, which we will glue into the manifolds $f(x)$ inside $L_{3\delta}$, similarly to the proof of Lemma \ref{lem:InverseObjects}. Let $C_i$ be the composition $* \to N_i \to *$ obtained from the above Lemma applied to $M_i$, and scaled so that each morphism has length $\epsilon$. Let $\tilde{C}_i$ denote the $\theta$-manifold with height function obtained by cutting $C_i$ along $\bR \times \{0\}$, and adding $a_i-\epsilon$ to the height function. The boundary of $\tilde{C}_i$ can be naturally identified with the boundary of $\bR \times [0,1]$, where the height function is given by the first coordinate. There is a height-preserving embedding $e_i : \tilde{C}_i \to \bR \times [\delta, 2\delta] \times [0, 3\delta)^\infty$ sending $\partial \tilde{C}_i$ to $\bR \times \{\delta, 2\delta \}$ with image that agrees with $L_{3\delta}$ as a $\theta$-manifold near
the boundary of $[a_i - \epsilon, a_i+\epsilon] \times [\delta, 2\delta] \times \bR^\infty$.

Let $D$ be the $\theta$-manifold obtained from $L_{3\delta}$ by cutting out $[a_i - \epsilon, a_i+\epsilon] \times [\delta, 2\delta] \times \{0\}$ and gluing in $e_i(\tilde{C}_i) \cap [a_i - \epsilon, a_i+\epsilon] \times [\delta, 2\delta] \times \bR^\infty$. The following figure shows the part of $D$ near $\bR \times (\delta, 2\delta) \times \bR^\infty$.

\begin{center}
\includegraphics[bb = 148 512 500 649, scale=1]{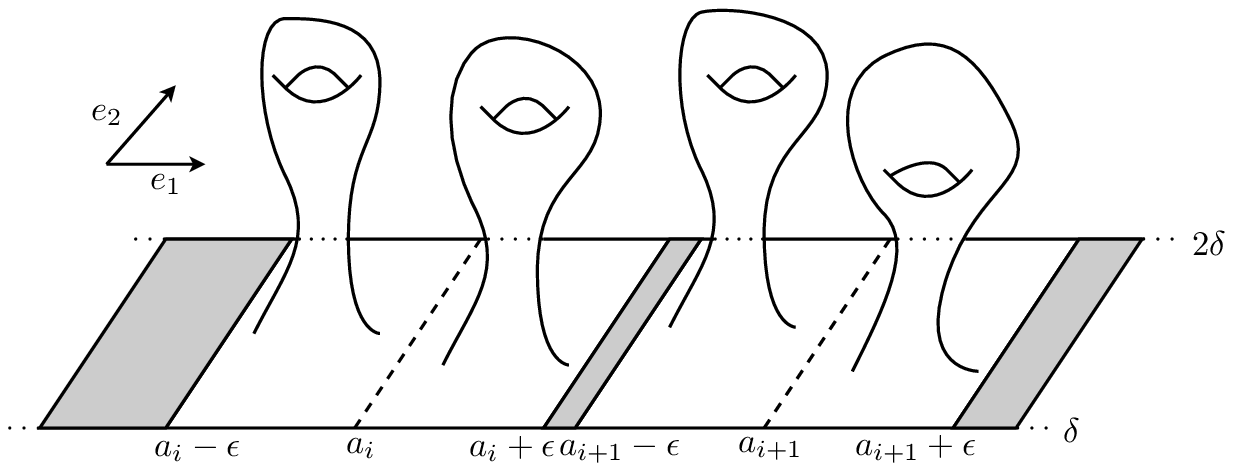}
\end{center}
For each $x \in D^k$ we can form $\bar{f}(x) = (f(x) - L_{3\delta}) \cup D$, which defines a continuous map $\bar{f} : D^k \to \psi_\theta(\infty,1)\bp_\Conn$. If $x \in U_i$, then $[\bar{f}(x) \cap \{a_i\} \times \bR^\infty] = [M_i] \cdot [N_i] \in \pi_1 \Fib(\theta)$ which is in $\pi_0 \mathbf{C}$ by construction. Thus $\bar{f}$ is a map into $\psi_\theta(\infty,1)\bp_{\mathbf{C}}$, and we must show it is relatively homotopic to $f$. We will do so by changing the height function on $D$. For ease of notation we will write $h$ instead of $x_1$ for the height function.

\vspace{2ex}

We first deform the height function on $D$ so that $D - L_{3\delta}$ has total height 0. Let $c_1 : \bR \to \bR$ be a smooth function such that $(t_i -\epsilon, t_i + \epsilon) \subseteq c_1^{-1}(t_i)$, and $c_1' \geq 0$. Let $c_s = s \cdot c_1 + (1-s) \cdot \mathrm{Id}$. Choose a smooth function $\varphi$ on $D$ that is identically 1 on $D-L_{3\delta}$, and identically 0 near $\bR \times \bR - (\delta,2\delta) \times \bR^\infty$. We can modify the height function on $D$ by
$$h_s(d) = \varphi(d) \cdot (c_s \circ h(d) - h(d)) + h(d).$$
Then $h_0 = h$ and if $c \in e_i(\tilde{C}_i) - L_{3\delta} \subseteq D$ then $h_1(c)=t_i$. The level sets of $h_1$ in $D$ at any points except $\{t_1, ...,t_n\}$ are then intervals with the standard $\theta$-structure.

We now define another path of height functions on $D$,
$$h'_s(d) = \varphi(d) \cdot 20 \cdot s + h_1(d).$$
Then $h'_0 = h_1$. If $c \in e_i(\tilde{C}_i) - L_{3\delta} \subseteq D$ then $h'_1(d) = t_i + 20 \in (19, 21)$. With this height function, $D$ has the standard tangential structure between heights $-10$ and $10$.

Applying these paths of height functions to the copy of $D$ inside each $\bar{f}(x)$ gives a path from $\bar{f}(x)$ to a manifold that agrees with $f(x)$ inside $x_1^{-1}(-10,10)$, so composing with the stretching map (\ref{eq:14}) gives a path to $f(x)$.

\vspace{2ex}

We must explain why these are \textit{relative} homotopies. Suppose $f(x) \in \psi^{nc}_\theta(\infty,1)^\bullet_\mathbf{C}$, so there is a $t \in (-1,1)$ and an $\eta>0$ with $f(x) \cap \{t'\} \times \bR^\infty \in \mathbf{C}$ for all $t' \in (t-\eta, t+\eta)$. If $t$ is not in $\cup_i (a_i-\epsilon, a_i+\epsilon)$, then $\bar{f}(x)$ agrees with $f(x)$ near $t$, so the level set at $t$ of $\bar{f}(x)$ is in $\mathbf{C}$. The path $h_s$ only changes the level sets in $\cup_i (a_i-\epsilon, a_i+\epsilon)$, so also always has a level set in $\mathbf{C}$. Finally, the path $h'_s$ only has finitely many level sets different from those of $f(x)$ at any time. Thus the final manifold always has a $t' \in (t-\eta, t+\eta)$ with level set in $\mathbf{C}$.

On the other hand, suppose $t \in (a_i-\epsilon, a_i+\epsilon)$. Then $M_i \in \mathbf{C}$ already, and we have thus chosen $N_i$ to be trivial. Thus the level sets of $f(x)$ and $\bar{f}(x)$ at $t$ are the same, and in $\mathbf{C}$. The path $h_s$ does not change this. Finally, the path $h'_s$ only has finitely many level sets in $(-10,10)$ different from $f(x)$ at a time. Thus the final manifold always has a $t'$ with level set in $\mathbf{C}$.

This concludes the proof of Theorem \ref{thm:CToConn}, in the case that $\mathbf{C}$ consists of whole path components of $\psi_{\theta_{d-1}}(\infty,0)$. The general case follows by:

\begin{proposition}\label{prop:PathComponentsOfObjects}
Let $\mathbf{C}$ be a collection of elements of $\psi_{\theta_{d-1}}(\infty,0)$, and $\mathbf{C}'$ be the collection of all elements of $\psi_{\theta_{d-1}}(\infty,0)$ in the same path component as an element in $\mathbf{C}$. Then the inclusion
$$\psi^{nc}_\theta(\infty,1)^\bullet_\mathbf{C} \to  \psi^{nc}_\theta(\infty,1)^\bullet_{\mathbf{C}'}$$
is a weak homotopy equivalence.
\end{proposition}
\begin{proof}
We will show that the relative homotopy groups of the inclusion vanish. The argument is almost identical to that of Proposition \ref{prop:Psi0IntoPsiDot}, so we will explain to what extent it differs.

Take a map $f : (D^m, \partial D^m) \to (\psi^{nc}_\theta(\infty,1)^\bullet_{\mathbf{C}'} ,  \psi^{nc}_\theta(\infty,1)^\bullet_\mathbf{C})$. As usual, there is a finite cover of $D^m$ by contractible open sets $U_i$ and values $a_i \in (-1,1)$ such that $a_i$ is a regular value of the height function on $f(x)$ for $x \in U_i$, and the level sets at $a_i$ are in $\mathbf{C}'$. We may suppose the $a_i$ are distinct, and choose an $\epsilon >0$ so that the intervals $(a_i-2\epsilon, a_i+2\epsilon)$ are disjoint. By Lemma \ref{Cylindrical}, we may further suppose that $f(x \in U_i)$ is cylindrical in $x_1^{-1}(a_i-\epsilon, a_i+\epsilon)$. It is important to note that the homotopy of Lemma \ref{Cylindrical} does not change which level sets occur in $f(x)$, only the heights at which they occur: thus it gives a relative homotopy.

Thus there is a map $\lambda_i : U_i \to \psi_{\theta_{d-1}}(\infty,0)\bp$ so that $f(x)$ and $\bR \times \lambda_i(x)$ are equal in $\Psi_\theta(x_1^{-1}(a_i-\epsilon, a_i+\epsilon))$. The map $\lambda_i$ takes values in $\mathbf{C}'$. As $U_i$ is contractible we can pick a smooth homotopy $\Lambda_i : [0,1]\times U_i \to \mathbf{C}'$ with $\Lambda_i(0,-)=\lambda_i$ and $\Lambda_i(1,-) \in \mathbf{C}$. We use these functions exactly as in the proof of Proposition \ref{prop:Psi0IntoPsiDot}.

It remains to see why this gives a relative homotopy. Over $U_i$, elements $f(x)$ are cylindrical in $x_1^{-1}(a_i-\epsilon, a_i+\epsilon)$. The homotopy constructed in Proposition \ref{prop:Psi0IntoPsiDot} is constant outside of $x_1^{-1}(a_i-2\delta, a_i+2\delta)$, for some $0 < 3\delta < \epsilon$. In particular, throughout the homotopy there are only new types of level sets added: if $f(x)$ has a level set in $\mathbf{C}$, then either it occurs at a height outside of $(a_i-2\delta, a_i+2\delta)$ and so is unchanged by the homotopy, or it occurs at a height inside $(a_i-2\delta, a_i+2\delta)$, in which case it also occurs at a height inside $(a_i+2\delta, a_i+\epsilon)$, which must be unchanged by the homotopy.
\end{proof}

Theorem \ref{thm:CToConn} implies the first part of Theorem \ref{thm:SubcategoryMonoids}. It remains to see homotopy commutativity of the endomorphism monoids in $\mathcal{C}_\theta\bp$; to that end we give yet another model for these spaces.

The monoid of endomorphisms of an object $C \in \mathcal{C}_\theta\bp$ is homotopy equivalent to $\coprod_W B\Diff_\theta(W, L\cup \partial W)$, where $W$ ranges over connected compact $\theta$-surfaces with two boundary circles, both identified with $C$. We can cut such surfaces along $L$ to obtain surfaces with one boundary component. More precisely, given a pointed map $\ell : S^1 \to \Fib(\theta)$, define
$$\tilde{\ell} : \bR^2 \overset{\pi_2}\to \bR \to S^1 \overset{\ell}\to \Fib(\theta)$$
where $\bR \to S^1$ collapses the complement of $(0,1)$ to the basepoint. Define an element $B_\ell=(\bR^2 \times \{0\}, \tilde{\ell}) \in \Psi_\theta(\bR^\infty)$. Let $\mathcal{M}(\ell) \subseteq \Psi_\theta(\bR^\infty) \times \bR$ be the subspace of pairs $(W, t)$ where $W$ is connected and agrees with $B_\ell$ on an open neighbourhood of the complement of $(0,t)\times (0,1) \times \bR^\infty$. This is a monoid via
$$(W_1, t_1) \circ (W_2, t_2) = (W_1 \cup (W_2 + t_1 \cdot e_1), t_1+t_2).$$

An object $C \in \mathcal{C}_\theta\bp$ determines a loop $\ell_C : S^1 \to \Fib(\theta)$, after choosing a parametrisation of the underlying circle (of which there are a contractible choice, as it is determined near the basepoint). There is an object $C'$ in the same path component of $\ob( \mathcal{C}_\theta\bp )$ as $C$ and a map of monoids $\mathcal{M}(\ell_C) \to \End_{\mathcal{C}_\theta\bp}(C')$ given by replacing $B_\ell$ by a cylinder $\bR \times C'$ that agrees with $B_\ell$ near $\bR \times [0,1] \times \{0\}$. It is a weak homotopy equivalence as they are classifying spaces for equivalent families. Furthermore, the monoids $\End_{\mathcal{C}_\theta\bp}(C')$ and $\End_{\mathcal{C}_\theta\bp}(C)$ are equivalent in the following sense: choosing a path $C \leadsto C'$ in $\ob (\mathcal{C}_\theta\bp)$ gives a weak homotopy equivalence $\End_{\mathcal{C}_\theta\bp}(C') \to \End_{\mathcal{C}_\theta\bp}(C)$ that is not quite a map of monoids, but is a map of $H$-spaces (and in fact a map of $A_\infty$ spaces).

\begin{proposition}\label{prop:HomotopyCommutative}
Let $\theta : \X \to BO(2)$ be any tangential structure. Then the monoid of endomorphisms of an object $C \in \mathcal{C}_\theta\bp$ is homotopy commutative.
\end{proposition}
\begin{proof}
By the above discussion, it is enough to prove the homotopy commutativity of $\mathcal{M}(\ell)$. Consider first the case where $\ell$ is the constant map to $* \in \Fib(\theta)$, the basepoint which determines the standard tangential structure. Then $\tilde{\ell}=*$ also and so $B_\ell = \bR^2 \times \{0\}$ with the standard $\theta$-structure induced by the framing given by the coordinate directions. Thus $\mathcal{M}(*)$ has an obvious action of the little 2-cubes operad. Its usual monoid multiplication is homotopic to that induced by the operad, so its usual multiplication is homotopy commutative (in fact, $E_2$).

There is a map $\ell_* : \mathcal{M}(\ell_1) \to \mathcal{M}(\ell \cdot \ell_1)$ given on $(W, t)$ by scaling $W-B_{\ell_1}$ by $\tfrac{1}{2}$ in the $e_2$ direction, and then taking the union with $B_{\ell \cdot \ell_1}$. This is a map of monoids, and composing with $\ell^{-1}_*$ gives a map $\mathcal{M}(\ell_1) \to \mathcal{M}(\ell^{-1} \cdot \ell \cdot \ell_1)$. This is easily seen to have a homotopy inverse, by picking a null homotopy of $\ell^{-1} \cdot \ell$ (though the inverse is not a map of monoids). Thus the maps $\ell_*$ are homotopy equivalences. In particular $\mathcal{M}(\ell) \to \mathcal{M}(\ell^{-1} \cdot \ell) \simeq \mathcal{M}(*)$ is a monoid map and an equivalence, so $\mathcal{M}(\ell)$ is equivalent as a monoid to a homotopy commutative monoid, so is homotopy commutative itself. 
\end{proof}

\section{Applications}\label{sec:Applications}

As mentioned in the introduction, taking the tangential structure $\theta : BSO(2) \to BO(2)$ gives a category $\mathcal{C}_{SO(2)}\bp$ having a unique path component, and hence cobordism class, of objects. The endomorphism monoid of any object is homotopy equivalent to $\mathcal{M} = \coprod_{g \geq 0} B\Diff^+(F_g, \partial F_g = S^1)$ with the pair of pants product. By \cite{EE}, components of the group $\Diff^+(F_g, \partial F_g = S^1)$ are contractible, so we can replace it by the associated mapping class group $\Gamma_{g,1} = \pi_0 \Diff^+(F_g, \partial F_g = S^1)$. The discussion in \S\ref{sec:harer-stab-mads} proves

\begin{MainCor}[Madsen--Weiss \cite{MW}]\label{cor:MadsenWeiss}
There is a homology equivalence
\begin{align*}
  \bZ \times B\Gamma_{\infty,1} \to \Omega^\infty \MT{SO}{2},
\end{align*}
where $\Gamma_{\infty,1}$ is the limit of the mapping class groups
$\Gamma_{g,1}$ as $g \to \infty$.
\end{MainCor}

A similar result holds for non-orientable or spin surfaces, though the components of the relevant monoid are slightly more complicated.

\begin{MainCor}[Wahl \cite{Wahl} Theorem B]\label{Cor:Unoriented}
Let $\mathcal{N}_g$ be the mapping class group of a non-orientable surface of genus $g$. Then there is a homology equivalence
$$\bZ \times B \mathcal{N}_\infty \to \Omega^\infty \MT{O}{2}.$$
\end{MainCor}

\begin{proof}
We apply Theorem \ref{thm:SubcategoryMonoids} with no tangential structure. The category $\mathcal{C}_{O(2)}\bp$ has a unique cobordism class of objects. The relevant homotopy commutative monoid is
$$\mathcal{M} = \coprod_{F} B\Diff(F, \partial F = S^1)$$
where the disjoint union is over all diffeomorphism types of unoriented surfaces with one boundary component, and Theorem \ref{thm:SubcategoryMonoids} identifies its classifying space as $\Omega^{\infty-1}\MT{O}{2}$.

The monoid $\pi_0\mathcal{M}$ is in bijection with the set of pairs $(g_1, g_2) \in \bN^2$ where one entry is always 0. The element $(g,0)$ corresponds to $F_{g, 1} = \#^g T - D^2$ and $(0,g)$ corresponds to $N_{g,1} = \#^g \bR \bP^2 - D^2$. Then addition in this monoid is given by usual addition if both elements are in the same factor, and by the formula $(g, 0) + (0, h > 0) = (0, h+2g)$ otherwise.

By the discussion in \S\ref{sec:harer-stab-mads}, the homology of the group completion of $\mathcal{M}$ is the same as that of the telescope $\mathcal{M}_\infty$ of
$$\mathcal{M} \overset{\cdot(0,3)} \lra \mathcal{M} \overset{\cdot(0,3)} \lra \mathcal{M} \overset{\cdot(0,3)} \lra \cdots$$
as $\pi_0\mathcal{M}$ is generated by $(1,0)$ and $(0,1)$, whose product is $(0,3)$.

There is a submonoid $\mathcal{M}'$ of $\mathcal{M}$ consisting only of the non-orientable surfaces (and the disc), and translation by $(0,1)$ sends $\mathcal{M}$ into $\mathcal{M}'$, so $\mathcal{M}'$ is cofinal. Thus $\mathcal{M}_\infty \simeq \mathcal{M}'_\infty$ and this is equivalent to $\bZ \times B \mathcal{N}_\infty$, the classifying-space of the stable non-orientable mapping class group.
\end{proof}

\begin{MainCor}[Harer \cite{HSpin}, Bauer \cite{Bauer}]\label{CorSpin}
Let $\hat{\Gamma}^{\mathfrak{s}}_{g, 1}$ be the (extended) spin mapping class group of an orientable surface of genus $g$ with one boundary circle, and with spin structure $\mathfrak{s}$ that agrees with the trivial spin structure on the boundary circle.

There is a homology equivalence
$$\bZ \times \bZ/2 \times B\hat{\Gamma}^{\sigma}_\infty \lra \Omega^\infty \MT{Spin}{2}$$
where $\hat{\Gamma}^{\sigma}_{\infty,1}$ is the stable spin mapping class group.
\end{MainCor}
\begin{proof}
We apply Theorem \ref{thm:SubcategoryMonoids} with the tangential structure $\theta : B\mathit{Spin}(2) \to BO(2)$. Write $\gamma^s$ for the universal vector bundle over $B\mathit{Spin}(2)$. Note that $B\mathit{Spin}(2) \simeq BSO(2)$ is connected and $S^2$ admits a spin structure, so the hypotheses of the theorem are fulfilled. The category $\mathcal{C}_{\mathit{Spin}(2)}\bp$ has two cobordism classes of objects, corresponding to a trivial and a nontrivial spin structure on the circle. The monoid of endomorphisms of the trivial circle is
$$\mathcal{M} = \coprod_{F} B\Diff^{\mathit{Spin}}(F, \partial F) = \coprod_{F} \Bun^\partial(TF, \gamma^s) \moddd \Diff(F, \partial F).$$
where the disjoint union is over all diffeomorphism types of surfaces with one boundary component. If $F$ is not orientable, $\Bun^\partial(TF, \gamma^s) = \emptyset$ and it plays no role. If $F$ is an oriented surface, the space $\Bun^\partial(TF, \gamma^s)$ is homotopy equivalent to the space of pairs $(\tau, \ell)$ of a map $\tau : F \to BSO(2)$ classifying the tangent bundle and a relative lift $\ell$
\begin{diagram}
 & & K(\bZ/2\bZ, 1) \\
& & \dTo \\
S^1& \rTo^* & B\mathit{Spin}(2) \\
\dTo & \ruTo^\ell[dotted] & \dTo \\
F& \rTo^{\tau} & BSO(2).
\end{diagram}
This space can be (non-canonically, and non-equivariantly, for the obvious action of $\Diff^+(F, \partial F)$ on $H^1(F, \partial F; \bF_2)$) identified with $\Map(F, \partial F ; K(\bZ/2\bZ, 1), *) \simeq H^1(F, \partial F; \bF_2) \times K(\bZ/2\bZ, 1)$ as there is a contractible space of such maps $\tau$. Thus
$$B\Diff^{\mathit{Spin}}(F, \partial F) \simeq H^1(F, \partial F; \bF_2) \times K(\bZ/2\bZ, 1) \moddd \Diff^+(F, \partial F).$$
The group $\Diff^+(F, \partial F)$ acts on $\pi_0(\Map(F, \partial F ; K(\bZ/2\bZ, 1), *)) \cong H^1(F, \partial F ; \bF_2)$ with two orbits (distinguished by their Arf invariant). Pick representatives $\mathfrak{s}_0$ and $\mathfrak{s}_1$ for these orbits. If $\Diff^{\mathfrak{s}_i}(F, \partial F)$ is the stabiliser of $\mathfrak{s}_i \in H^1(F, \partial F ; \bF_2)$, then there are fibre sequences
$$K(\bZ/2\bZ, 1) \lra \Bun^\partial_{\mathfrak{s}_i}(TF, \gamma^s) \moddd \Diff^{\mathfrak{s}_i}(F, \partial F) \lra B\Diff^{\mathfrak{s}_i}(F, \partial F).$$
The contractibility \cite{EE} of the components of $\Diff(F, \partial F)$ means that all three spaces are $K(\pi,1)$'s. Write $\hat{\Gamma}^{\mathfrak{s}_i}_{g,1}$ for the fundamental group of the middle space, when $F$ is an oriented surface of genus $g$. This is by definition the \textit{extended spin mapping class group}.

Now $B\Diff^{\mathit{Spin}}(F, \partial F) = \coprod_{i=0,1} \Bun^\partial_{\mathfrak{s}_i}(TF, \gamma^s) \moddd \Diff^{\mathfrak{s}_i}(F, \partial F)$ and so
$$\mathcal{M} = \coprod_{g \geq 0} \coprod_{i=0,1} B \hat{\Gamma}^{\mathfrak{s}_i}_{g,1}.$$
The components of this monoid are $\bN \times \bZ/2\bZ$, so it can be group-completed by inverting just $(1,0)$: this corresponds to gluing on a torus with a spin structure of Arf invariant 0. The telescope $\mathcal{M}_\infty$ is then $\bZ \times \bZ/2\bZ \times B \hat{\Gamma}^{\sigma}_{\infty,1}$.
\end{proof}

If $\theta : \X \to BO(2)$ is a tangential structure with $\X$ connected and such that $S^2$ admits a $\theta$-structure, and $Y$ is a path connected space, then $\theta \times Y = \theta \circ \pi_{\X} : \X \times Y \to BO(2)$ again has these properties. This allows us to add maps to a background space $Y$ to any such tangential structure.

Let us develop this in the case of ordinary orientations, where the tangential structure is $\theta : BSO(2) \to BO(2)$. The space $\mathcal{S}_{g,1}(Y; *)$ of Cohen--Madsen \cite{CM} is precisely the space $B\Diff^{\theta \times Y}(F_{g, 1}, \partial F_{g, 1})$, where the map to $Y$ is fixed to be the constant map to the basepoint on the boundary. Cobordism classes of objects in $\mathcal{C}_{\theta \times Y}\bp$ are in natural bijection with $H_1(Y, \bZ)$, via the map that sends  a circle in $Y$ to the homology class it represents. Let $C$ be a representative of the trivial homology class.

The monoid of endomorphisms of $C$ in $\mathcal{C}_{\theta \times Y}^\bullet$ is $\mathcal{M} = \coprod_{g \geq 0} \mathcal{S}_{g,1}(Y; *)$. For a general space $Y$ this monoid has tremendously many components, and so it is not clear that one can form the stabilisation $\mathcal{M}_\infty$, as one may not be able to group complete $\pi_0(\mathcal{M})$ by inverting finitely many elements. However, in several interesting one \textit{can} form $\mathcal{M}_\infty$.

Note that there is a homotopy fiber sequence
\begin{equation}\label{MappingSpaceFibration}
\Omega^2 Y \to \map((F_{g,1}, \partial F_{g,1}), (Y, *)) \to (\Omega Y)^{2g}
\end{equation}
onto those components of $(\Omega Y)^{2g}$ represented by tuples $(a_1, b_1, ..., a_g, b_g) \in \pi_1(Y)^{2g}$ such that $\prod_{i=1}^g [a_i, b_i] = 1$.

\subsection{$Y$ simply connected} In this case $\pi_0(\mathcal{S}_{g,1}(Y; *)) \cong \pi_2(Y) \cong H_2(Y ; \bZ)$, and $\pi_0(\mathcal{M}) = \bN \times H_2(Y ; \bZ)$ as a monoid. This can be group completed by inverting just $(1,0) \in \bN \times H_2(Y ; \bZ)$, and so we can form $\mathcal{M}_\infty$ as the mapping telescope of
$$\mathcal{M} \overset{\cdot(1,0)}\to \mathcal{M} \overset{\cdot(1,0)}\to \mathcal{M} \overset{\cdot(1,0)}\to \cdots.$$
The discussion in \S\ref{sec:harer-stab-mads} and Theorem \ref{thm:SubcategoryMonoids} gives

\begin{MainCor}[Cohen--Madsen \cite{CM}]
Let $Y$ be a simply connected space. Then there is a homology equivalence
$$\bZ \times \hocolim_{g \to \infty}\mathcal{S}_{g,1}(Y; *) \to \Omega^\infty \MT{SO}{2} \wedge Y_+$$
where the colimit is formed by the map $\cdot(1,0)$ above.
\end{MainCor}

\subsection{$Y=BG$, $G$ a cyclic group} In this case
$$\pi_0(\mathcal{S}_{g,1}(BG; *)) = \pi_0(\Bun^\partial(TF_{g,1}, \theta^*\gamma)) \times_{\Diff(F_{g,1})} \pi_0(\map^\partial(F_{g,1}, BG))$$
and the homotopy fiber sequence (\ref{MappingSpaceFibration}) gives that $\map^\partial(F_{g,1}, BG) \simeq G^{2g} \simeq H^1(F_{g,1}, \partial F_{g,1} ; G)$ so
\begin{equation}\label{eq:ComponentsOfMonoid}
\pi_0(\mathcal{M}) = \coprod_{g \geq 0} G^{2g} / \Diff^+(F_{g,1})
\end{equation}
where composition is by concatenating (representatives of equivalence classes of) tuples of elements of $G$. Note that this monoid is generated by the $g=1$ elements. This much holds for any abelian group $G$.

\begin{lemma}
Suppose that $G$ is a cyclic group generated by an element $1$. Then in $G^2/\Diff^+(F_{1,1})$ there are relations
$$[a \pm b,b] = [a,b] = [a, b \pm a].$$
In $G^4/\Diff^+(F_{2,1})$ there are relations
$$[0,1]\cdot[a,b] = [0,1]\cdot[a-1, b]$$
and
$$[0,1]\cdot[0,b] = [0,1]\cdot[0,1]$$
in the monoid $\pi_0(\mathcal{M})$.
\end{lemma}
\begin{proof}
$\Diff^+(F_{1,1})$ acts on $G^2$ through the group $Sp(2, \bZ)$, to which it surjects. The obvious group elements give the relations claimed.
$\Diff^+(F_{2, 1})$ acts on $G^4$ through the group $Sp(4, \bZ)$, to which it surjects. Define an element
\[ t = \left( \begin{array}{cccc}
1 & 1 & 0 & -1 \\
0 & 1 & 0 & 0 \\
0 & -1 & 1 & 1 \\
0 & 0 & 0 & 1 \end{array} \right)\] 
in this group. Then
$$t(0,1,a,b) = (1-b, 1, a-1+b, b) \sim (0,1, a-1, b).$$
We also have
$$t(0,1,0,b) = (1-b, 1, b-1, b) \sim (0,1,b-1, 1) \sim (0,1,0,1).$$
\end{proof}
Thus inverting $[0,1]$ gives $[a,b] = [0,b]$ in the group-completion, and implies $[0,b]$ is invertible. Thus it is enough to just invert $[0,1]$ to group-complete the monoid, so we may form
$$\mathcal{M}_\infty = \hocolim(\mathcal{M} \overset{\cdot [0,1]}\to \mathcal{M} \overset{\cdot [0,1]}\to \mathcal{M} \overset{\cdot [0,1]}\to \cdots)$$
whereby Theorem B gives

\begin{MainCor}
Let $G$ be a cyclic group. Then there is a homology equiavalence
$$\bZ \times \hocolim_{g \to \infty} \mathcal{S}_{g,1}(BG; *) \to \Omega^\infty \MT{SO}{2} \wedge BG_+$$
where the colimit is formed by the map $\cdot [0,1]$ above.
\end{MainCor}

Let us spell this result out in certain cases. Recall that $\Diff^+(F_{g,1}) \simeq \pi_0(\Diff^+(F_{g,1}))$ which is the mapping class group $\Gamma_{g,1}$ of $F_{g,1}$. When $G=\bZ$, we have $\Gamma_{g,1}$ acting on $\bZ^{2g}$, and let us write $\Gamma'_{g,1}$ for the subgroup of $\pi_0(\Diff^+(F_{g,1}))$ which fixes the primitive element $(0,1,0,1,\dots,0,1) \in \bZ^{2g}$. The statement of the corollary is then that
$$\bZ \times \hocolim(\cdots \to B\Gamma'_{g,1} \overset{\cdot [0,1]}\to B\Gamma'_{g+1,1} \to \cdots) \to \Omega^\infty \MT{SO}{2} \wedge S^1_+$$
is a homology equivalence. This identifies the stable homology of the subgroup $\Gamma'_{g,1} < \Gamma_{g,1}$ which fixes a primitive element of $H^1(F_{g,1}, \partial F_{g,1} ; \bZ)$.

When $G= \bZ/n\bZ$, we have $\Gamma_{g,1}$ acting on $(\bZ/n\bZ)^{2g}$, and let us write $\Gamma'_{g,1}(n)$ for the subgroup of $\pi_0(\Diff^+(F_{g,1}))$ which fixes the primitive element $(0,1,0,1,\dots,0,1) \in (\bZ/n\bZ)^{2g}$. The statement of the corollary is then that
$$\bZ \times \hocolim(\cdots \to B\Gamma'_{g,1}(n) \overset{\cdot [0,1]}\to B\Gamma'_{g+1,1}(n) \to \cdots) \to \Omega^\infty \MT{SO}{2} \wedge B\bZ/n\bZ_+$$
is a homology equivalence.  On the other hand, $B\Gamma_{g,1}'(n)$ can be identified with the moduli space of Riemann surfaces of genus $g$ with a single framed point, equipped with an $n$-fold cyclic unbranched cover. This identifies the stable homology of these moduli spaces.


A similar analysis holds for any space $Y$ with cyclic fundamental group. The question of whether $\pi_0(\mathcal{M})$ may be group completed by inverting finitely many elements for general groups $G=\pi_1(Y)$ seems to be difficult.

\bibliographystyle{amsalpha}
\bibliography{MMM}

\end{document}